\documentclass{article}

\usepackage[english]{babel}
\usepackage{xcolor}
\usepackage{authblk}

\usepackage[letterpaper,top=2cm,bottom=2cm,left=3cm,right=3cm,marginparwidth=1.75cm]{geometry}

\usepackage[utf8]{inputenc}
\usepackage{mathtools}
\usepackage{amsmath}
\usepackage{amsthm}
\usepackage{algorithm}
\usepackage{algorithmic}
\usepackage{amssymb}
\usepackage{graphicx}
\usepackage[colorlinks=true, allcolors=blue]{hyperref}
\usepackage{tikz}

\usepackage{caption}
\usepackage{subcaption}

\DeclareMathOperator{\domain}{dom}

\def\R{{\mathbb R}}
\def\calA{{\mathcal A}}  
  
\def\calG{{\mathcal G}}  \def\calI{{\mathcal I}}
  
 \def\calN{{\mathcal N}} 
\def\calP{{\mathcal P}}  
 \def\calT{{\mathcal T}}

\DeclarePairedDelimiter{\skp}{\langle}{\rangle}

\newcommand{\res}{\mathop{\hbox{\vrule height 7pt width .5pt depth 0pt
\vrule height .5pt width 6pt depth 0pt}}\nolimits}
\newcommand{\mres}{\mathbin{\vrule height 1.4ex depth 0pt width
0.13ex\vrule height 0.13ex depth 0pt width 1.0ex}}

\newtheorem{theorem}{Theorem}[section]

\newtheorem{proposition}[theorem]{Proposition}

\newtheorem{assumption}[theorem]{Assumption}
\newtheorem{definition}[theorem]{Definition}
\newtheorem{example}[theorem]{Example}
\newtheorem{remark}[theorem]{Remark}

    \allowdisplaybreaks

\title{Dynamic Optimal Transport with Optimal Preferential Paths}
\author[1]{Marcello Carioni}
\author[2]{Juliane Krautz}
\author[2,3]{Jan-F. Pietschmann}
\affil[1]{{
\small
Department of Applied Mathematics, University of Twente, P.O. Box 217, 7500 AE Enschede,
The Netherlands. Email: m.c.carioni@utwente.nl
}}
\affil[2]{{
\small
University\"{a}t Augsburg, Institut f\"ur Mathematik, Universit\"{a}tsstra\ss e 12a, 86159 Augsburg, Germany. Emails: \{juliane.krautz, jan-f.pietschmann\}@uni-a.de
}}
\affil[3]{{
\small
Centre for Advanced Analytics and Predictive Sciences (CAAPS), University of Augsburg,
Universit\"{a}tsstr. 12a, 86159 Augsburg, Germany. 
}}
\date{}
\begin{document}
\maketitle

\begin{abstract}
We study a dynamic optimal transport type problem on a domain that consists of two parts: a compact set $\Omega \subset \R^d$ (bulk) and a non-intersecting and sufficiently regular curve $\Gamma \subset \Omega$. On each of them, a Benamou-Brenier type dynamic optimal transport problem is considered, yet with an additional mechanism that allows the exchange (at a cost) of mass between bulk and curve. In the respective actions, we also allow for non-linear mobilities. We first ensure the existence of minimizers by relying on the direct method of calculus of variations and we study the asymptotic properties of the minimizers under changes in the parameters regulating the dynamics in $\Omega$ and $\Gamma$.
Then, we study the case when the curve $\Gamma$ is also allowed to change, being the main interest in this paper. To this end, the Tangent-Point energy is added to the action functional in order to preserve the regularity properties of the curve and prevent self-intersections. Also in this case, by relying on suitable compactness estimates both for the time-dependent measures and the curve $\Gamma$, the existence of optimizers is shown.
We extend these analytical findings by numerical simulations based on a primal-dual approach that illustrate the behaviour of geodesics, for fixed and varying curves.
\end{abstract}


\section{Introduction}
Originally formulated by Monge in the 18th century and later refined by Kantorovich in the 20th century, Optimal Transport offers a mathematical framework for comparing and transforming probability distributions by minimizing the cost of moving mass from one distribution to another. 
Given a compact set $\Omega \subset \R^d$ and two probability measures $\mu \in P(\Omega)$ and $\nu \in P(\Omega)$ and choosing the Euclidean distance as the cost function results in the 
$2$-Wasserstein distance defined as  
\begin{align}
    W_2^2(\mu,\nu) = \sqrt{\inf_{\pi \in \Pi(\mu,\nu)} \int |x-y|^2 d\pi(x,y)}
\end{align}
where $\Pi(\mu,\nu)$ is the set of transport plans, that are all probability measures in the product space $\Omega\times \Omega$ having $\mu$ and $\nu$ as marginals. 
While the $2$-Wasserstein distance focuses on static settings, many real-world problems require transport models that evolve over time. This need has motivated the study of dynamic optimal transport introduced in the pioneering work by Benamou and Brenier \cite{benamou2000computational}. In this work it has been shown that the $2$-Wasserstein distance admits a reformulation as the following
dynamical fluid-mechanics problem  
\begin{align}
    W_2(\mu,\nu) = \inf_{(\rho_t, J_t)} \int_0^1\int_{\Omega}  \left|\frac{dJ_t}{d\rho_t}\right|^2\, d\rho_t dt
\end{align}
where the pair $(\rho_t, J_t)$ satisfies the continuity equation
\begin{align}
    \partial_t \rho_t + {\rm div}(J_t) = 0
\end{align}
in the distributional sense, with $\rho_{0} = \mu$ and $\rho_{1} = \nu$. Such dynamic formulation of the Wasserstein distance has proven useful in a great variety of contexts and it is a fundamental ingredient to develop a gradient flow theory in Wasserstein spaces \cite{Otto2001,JKO98,AGS2008}. Successful examples of applications are consequences of the usefulness of an optimal transport prior in understanding dynamics and includes among others applications to crowd modelling and traffic congestion \cite{maury2010macroscopic, carlier2008optimal, stephanovitch2024optimal}, dynamic inverse problems \cite{schmitzer2019dynamic, bredies2020optimal, bredies2023generalized, dawood2010continuity}, optimal control models and algorithms in machine learning \cite{onken2021ot}. 
Moreover, variants of the Benamou-Brenier approach have been explored to include more general action functionals \cite{Dolbeault_2008, Lisini2010, Carrillo2010}, dissipations \cite{chizat2018unbalanced, liero2018optimal, kondratyev2016new} and non-local interactions \cite{slepvcev2023nonlocal}.

In this paper, we contribute to the above mentioned theory by studying a dynamic optimal transport type problem where mass can be transported both in a bulk domain $\Omega$ as well as on a curve $\Gamma \subset \Omega$, which has co-dimension of at least one. 
In our model the mass between the bulk and the curve can be exchanged, while the total mass in the system remains a conserved quantity. 
The cost of transport is measured by means of an action functional that includes the kinetic energy of the time-dependent distribution in the bulk and on the curve, respectively, as well as the additional cost due to exchange of mass between bulk and curve. From an intuitive standpoint our mathematical framework models the optimal evolution of a population (for instance metropolitan traffic) when there exists a path (such as a highway) where the transport is faster (or cheaper), but it requires a cost to access the path. This could model highways where, due to a high traffic concentration at the toll booths, the access to the highway requires a non-negligible time investment. 
In our analysis we first consider the preferential curve $\Gamma$ to be fixed, where the goal is to find the optimal evolution of the population given initial and final configurations. This setting is strongly related to previous works that coupled transport in a domain with transport on its boundary, \cite{monsaingeon2021new} as well as recent approaches to transport on metric graphs, \cite{Erbar2022,BurgerHumpertPietschmann2023}.
In a second step, we extend the formalism by allowing the curve to vary, making its shape part of the optimization problem. Such problems can be seen in the context of optimal planning, where the preferential path is optimized according to initial and final distributions and under the prior assumption that the population is moving following an optimal transport dynamic.

We now describe the problem using the mathematical formalism we employ in the paper. We consider $\Omega$ a compact, connected subset of $\R^d$, $\Gamma = \gamma([0,1])$ the image of a regular, non self-intersecting curve of class $C^1([0,1];\Omega)$. The distribution of mass on the bulk and in $\Gamma$ is modelled via a pair of time-dependent measures $(\rho_t,\mu_t)$, indexed by $t\in [0, 1]$. Here, $\rho_t \in M(\Omega)$ represents the mass distributions in the  bulk while $\mu_t \in M(\Gamma)$ denotes the distribution on the preferential path $\Gamma$. 
Since the total mass will be conserved in our model, we assume a compatibility condition for initial and final distributions, meaning that $\rho_0(\Omega) + \mu_0(\Gamma) = 1$ and $\rho_1(\Omega) + \mu_1(\Gamma) = 1$. In what follows, this will be denoted by $(\rho_0, \mu_0) \in \mathcal{P}_{\rm adm}(\Gamma)$ and $(\rho_1, \mu_1) \in \mathcal{P}_{\rm adm}(\Gamma)$.
To model the transportation of mass, we denote by $J_t$ and $V_t$ the momentums associated to $\rho_t$ and $\mu_t$. Together, $(\rho_t,\mu_t, J_t, V_t)$ are assumed to satisfy two continuity equations which are coupled in order to allow exchange between bulk and curve while preserving the total mass:
\begin{align}\label{eq:CE_intro}
    \frac{\partial \rho_t}{\partial t}+\nabla \cdot J_t=0, \quad &\text{on} \quad \Omega, \qquad \frac{\partial \mu_t}{\partial t}+ \nabla_\Gamma \cdot V_t=f_t\quad \text{on} \ \  \Gamma, \\
    & J_t \cdot n_{\Gamma} = f_t \quad \text{in } \Gamma \nonumber.
\end{align}
We refer to Section \ref{sec:conn} for a more rigorous definition of \eqref{eq:CE_intro} given in a weak sense and also for a generalization of \eqref{eq:CE_intro} where 
 an exchange mechanism that acts at the end points of $\Gamma$ is added. We already note that this formulation allows for mass being present on $\Gamma$ and in $\Omega$ at the same spatial location. All curves that satisfy the above for given initial and final data $(\rho_0,\mu_0) \in \mathcal{P}_{\rm adm}(\Gamma)$ and $(\rho_1,\mu_1) \in \mathcal{P}_{\rm adm}(\Gamma)$ are denoted by ${\rm CE}(\Gamma)$.

Next we introduce an action functional which measures the time accumulated cost of transport in bulk and curve, respectively, but also the cost of mass transfer. Inspired by the works \cite{Dolbeault_2008, Lisini2010} we employ mobilities $m_\Omega$ and $m_\Gamma$ to regulate the dynamics on $\Omega$ and $\Gamma$.
 The action functional is then defined as  
\begin{align}\label{eq:BB}\tag{BB}
\calA_\Gamma(J_t, V_t, \rho_t, \mu_t, f_t) &= \int_{\Omega} \frac{|J_t|^2}{m_\Omega(\rho_t)} \, d\lambda_\Omega 
+ \alpha_1  \int_{\Gamma} \frac{|V_t|^2}{m_\Gamma(\mu_t)} \, d\lambda_\Gamma
+ \alpha_2  \int_{\Gamma} \frac{|f_t|^2}{m_\Gamma(\mu_t)} d\lambda_\Gamma.
\end{align}
Here, $\lambda_\Omega,\,\lambda_\Gamma$ are reference measures, necessary as the mobility yields integrands that are no more $1$-homogeneous. 
Note also that with a slight abuse of notation we have identified the measures $J_t, V_t, \rho_t, \mu_t, f_t$ with their densities respect to the reference measures. We refer to Section \ref{sec:dynamicformulation} for a rigorous definition. In addition $\alpha_1, \alpha_2 >0$ are parameters that either control the cost of transport on $\Omega$ and $\Gamma$, respectively ($\alpha_1)$, or that of being transported onto or off from $\Gamma$ ($\alpha_2$). For a fixed curve $\Gamma$ we thus consider the dynamic  variational problem
\begin{align}\label{eq:opt_fixed_intro}
\inf_{(J_t, V_t, \rho_t, \mu_t, f_t) \in {\rm CE}(\Gamma)} \int_0^1 \calA_\Gamma(J_t, V_t, \rho_t, \mu_t, f_t) dt.
\end{align}
In our analysis we prove that \eqref{eq:opt_fixed_intro} admits a minimizer and study its asymptotic behaviour when $\alpha_2 \rightarrow 0$ and $\alpha_2 \rightarrow +\infty$. In this first part of our work, our analysis follows closely the one carried out in \cite{monsaingeon2021new}, where the exchange of mass and different dynamics occur on the boundary of $\Omega$. However, due to the introduction of mobilities and since the preferential path can be located arbitrarily in $\Omega$, many arguments differ from \cite{monsaingeon2021new}. In particular, contrarily to \cite{monsaingeon2021new} where a duality approach is used, we rely on the direct method of Calculus of Variations that allows to handle the usage of mobilities.

In the second part of this work, we extend the problem by making the shape of $\Gamma$ part of the optimization problem. One of the main challenges of this approach is that, in order to define a continuity equation on $\Gamma$, the curve $\gamma$ needs to be regular and non self-intersecting. As a consequence, our variational problem should have specific compactness properties that enforce the minimizing curve $\gamma_{\rm min}$ to be regular and non self-intersecting. To this end, we augment the action functional by adding a regularization term $R : C^1([0,1];\Omega) \rightarrow [0,\infty]$ that includes the following terms. The first term is the so called Tangent-Point energy \cite{STRZELECKI_2012} of the curve $\gamma$ defined as 
\begin{align}
    E^p_{\rm tp}(\gamma) = \int_0^{1} \int_0^{1} \frac{|\gamma'(s)||\gamma'(t)|}{r_{\rm tp}^{\gamma}(\gamma(s), \gamma(t))^p} \, ds dt
\end{align}
where $r_{\rm tp}^{\gamma}(x,y)$ denotes the radius of the smallest sphere tangent to $x\in\Gamma$ and passing through $x,y\in\Gamma$, see \autoref{fig:TangentPointCircle}. 
The Tangent-Point energy was originally defined for closed curves or knots with the goal of characterizing their self-avoidance properties \cite{STRZELECKI_2012}. 
Following the same principle, we apply the Tangent-Point energy to open curves, ensuring that limits of minimizing sequences of the augmented variational problem are one dimensional embedded submanifolds in $\Omega$ with boundary and thus parametrized by a non self-intersecting curve.
We also show that such an energy has good compactness properties. Indeed, by adapting the results from \cite{Blatt2013} and \cite{blatt_regularity_2012} that hold for closed curves, we show that, under arclength parametrization, the uniform bound on $E^p_{\rm tp}$ implies a uniform bound in $W^{2-1/p,p}([0,1];\Omega)$, compactly embedding in $C^1([0,1];\Omega)$. 
\begin{figure}
    \centering
    \begin{tikzpicture}
        \draw[domain=-1:1.3,smooth,variable=\x] plot ({\x},{\x^3 - \x}) node[right] {$\gamma(t) = t^3 - t$};
        \draw[blue] (1/2,1/2) circle[radius=sqrt(2)/2];
        \draw (-.5,.5) -- (.5,-.5);
        \fill (0,0) circle(1pt) node[below left] {$(0,0)$};
        \fill (1,0) circle(1pt) node[right] {$(1,0)$};
        \end{tikzpicture}
    \caption{Circle tangent to the image of $\gamma(t)=t^3-t$ at $t=0$ and passing through $(t,\gamma(t))$ for $t=1$ }
    \label{fig:TangentPointCircle}
\end{figure}
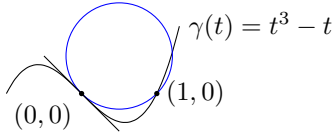
Relying on such results we augment the action functional by the additional regularization
\begin{align}\label{eq:}
    R (\gamma) := 2^{-p}E^p_{\rm tp}(\gamma) -\log|\gamma(0) - \gamma(1)| + \int_0^{1} |\gamma'(t)| \, dt.
\end{align}
The functional $R$ is a combination of the Tangent-Point energy and two additional terms. The second one is added for technical reasons, since our model is defined for open curves allowing for normal flows at the end points. Other modelling options, such as extending the problem to closed curves would be possible. The third term is instead natural and necessary to avoid minimizing sequences of curves with unbounded length. By augmenting the action functional \eqref{eq:opt_fixed_intro} with the regularization $R$ we are able to prove, for $p>2$, existence of an optimal curve $\gamma$ and optimal time-dependent measures $(J_t, V_t, \rho_t, \mu_t, f_t) \in {\rm CE}(\Gamma)$ that jointly minimize \eqref{eq:opt_fixed_intro}.

In the third part of the work, we develop a numerical scheme to compute both the time-dependent measures $(J_t, V_t, \rho_t, \mu_t, f_t) \in {\rm CE}(\Gamma)$ as solution of \eqref{eq:opt_fixed_intro} for a fixed curve $\gamma$ as well as for an optimal curve $\gamma$ that is then part of the optimization problem. In particular, we rely on an augmented Lagrangian formulation of the problem. 

\subsection*{Contributions and outline}
The main contributions of our work are the following:
\begin{itemize}
\item Provide a rigorous formulation of a set of coupled continuity equations for bulk-curve transport on arbitrary, regular curves of co-dimension at least one (Definition \ref{def:conteq}).
\item Establish existence of minimizers for the dynamic transport cost problems introduced above, both for a fixed (Theorem \ref{ExistenceFixedGamma}) as well as for a varying curve (Theorem \ref{thm:compn}).
\item Study certain limits when cost parameters in the action tend to zero or infinity (Section \ref{ParamLim}).
\item Provide a numerical scheme based on the combination of a finite element discretization and a primal-dual optimization algorithm and demonstrate its abilities via several examples (Section \ref{Numerics}).
\end{itemize}
The remainder of the paper is structured as follows. In Section \ref{sec:FixedPath}, we analyse the case of a fixed preferential path. We introduce the notion of weak solutions to the coupled system of continuity equations in Subsection \ref{sec:conn} as well as the rigorous formulation of the dynamic problem in Subsection \ref{sec:dynamicformulation} and show well-posedness in Subsection \ref{sec:WellPsdnFP}. Section \ref{ParamLim} is devoted to the study of limits of cost-parameters tending to zero or infinity. These limits establish a connection between the coupled problem and classical (uncoupled) Optimal Transport with mobilities. The minimization problem with varying paths is analysed in Section \ref{OptPrefPath}, including the study of the Tangent-Point energy for non-closed curves in Subsection \ref{TPEcurves}, the proof of existence of minimizers in Subsection \ref{WellPsdnVP} and a discussion of linear mobilities in Subsection \ref{LinMobilities}. In Section \ref{Numerics}, we conclude by introducing a numerical algorithm that is applied it to several examples.

\subsection{Further related work}

\paragraph*{Dynamic optimal transport with mass exchange}
The basis for our formulation of the transport problem is the dynamic formulation introduced by Benamou and Brenier in \cite{benamou2000computational}. The extension of this formalism in order to implement mass exchange became popular in the context of unbalanced transport, \cite{chizat2018unbalanced, liero2018optimal, kondratyev2016new}. There, an additional reaction term added to the continuity equation allowed for a change of mass over time. Closer to our work is the approach of \cite{monsaingeon2021new}, which already included the coupling of a bulk domain with a lower-dimensional one, in this case the boundary of the domain. In this work, the author includes an exchange mechanism of mass between the two that is based on a non-homogeneous flux boundary condition in the bulk and a related reaction term on the boundary. This idea was later applied to metric graphs with reservoirs that may carry mass, \cite{BurgerHumpertPietschmann2023, Fazeny2024}. Contrary to our approach, both \cite{monsaingeon2021new} and \cite{BurgerHumpertPietschmann2023} establish existence of minimizers via convex duality instead of the direct method. The structure of gradient flows on metric graphs with mass exchange at the vertices was recently explored in \cite{Heinze2024}.

\paragraph*{Non-linear mobilities}
The first study of dynamic transport distances with non-linear mobilities is due to \cite{Dolbeault_2008}, where power-like concave mobilities were introduced. This was later extended to mobilities with unbounded domains, \cite{Carrillo2010}, as well as volume filling mobilities in \cite{Lisini2010}. Convexity properties along geodesics in the presence of mobilities have been analyzed in \cite{Carrillo2010}.

\paragraph*{Tangent point energies}\label{DiscTPE}
The problem of self-avoidance of closed curves is well known in knot theory and several functionals preventing self-intersections have already been studied.
Some standard choices include Möbius energies introduced by O'Hara in \cite{ohara_family_1992} or Ropelength defined as the ratio of length over thickness of a curve \cite{gonzalez_global_1999}. 
However, these functionals are numerically difficult to handle and only few schemes have been investigated, e.g. \cite{walker_shape_2016} in the context of shape optimization. They either contain the evaluation of several suprema (in the case of Ropelength) or they need geodesic distances and allow for almost intersecting curves (in the case of the Möbius energies). 
Therefore, finding energies "in between" is of interest.
In \cite{gonzalez_global_1999}, the authors proposed a new family of functionals called Tangent-Point energies. They can be represented using an integral formulation without the evaluation of geodesic distances. Moreover, no supremum is needed in the definition making it a good candidate for numerical applications. 

In recent years there has been interest in the analytical properties of Tangent-Point energies, see \cite{STRZELECKI_2012} and \cite{Blatt2013}, \cite{blatt_regularity_2012}. It turns out that curves with finite Tangent-Point energy gain $C^{1,\alpha}$-regularity for $\alpha\in(0,1)$ and their images define embedded one dimensional submanifolds.

\paragraph*{Numerical algorithms}
There exists a plephora of algorithms that is based on the dynamic formulation of optimal transport, notably even the main focus of the original work by Benamou and Brenier \cite{benamou2000computational} was the design of an augmented Lagrangian (ALG) approach that allows to enforce the continuity equation constraint. More recently, the ALG approach has been adopted to finite volumes in \cite{Cances2020_FVSchemeJKO}.
More relevant for this work are proximal splitting algorithms which were used in \cite{Papadakis2014}, combined with staggered grid discretizations to enforce the continuity equation constraints.
For our work, we extend the algorithm introduced in \cite{benamou2000computational} to the bulk-curve setting using a finite element discretization of space-time domains introduced in \cite{fu_high_2023}. 
Finally, we mention a related approach due to \cite{Carrillo2022_PrimalDual} which employs piecewise constant approximations and proximal splitting algorithms to discretize the JKO scheme. However, in \cite{Carrillo2022_PrimalDual}, the continuity equation constraints are enforced only in a relaxed sense. The resulting formulation then involves the minimization of a sum of three convex functionals, which is done by the algorithm developed in \cite{Yan2018}.

As for the numerical treatment of the tangent point energy, \cite{bartels_simple_2018} and \cite{bartels_stability_2018} introduced a scheme treating it as a penalty term when minimizing the elastic bending energy of a curve. In \cite{yu_repulsive_2021} a gradient descent method using Sobolev-Slobodeckij inner products has been proposed and applied to different problems like curve packing or graph drawing to avoid self-collision. Due to the presence of additional curve-dependent terms in the minimization, we employ a gradient descent method, numerically approximating the gradient by finite differences.

\subsection{Notation}\label{Notation}

In what follows, the term \emph{non self-intersecting} is used to denote injective curves.
Given $\Omega$ a compact subset of $\R^d$ and $\Gamma = \gamma([0,1])$ the image of a curve of class $C^1([0,1];\Omega)$ contained in $\Omega$ which is non self-intersecting and regular, i.e. $\vert\gamma'\vert\neq 0$, we define
\begin{align*}
    X_\Omega = [0,1] \times \Omega, \quad X_\Gamma = [0,1] \times \Gamma.
\end{align*}
For a Euclidean space $A$ we denote by $M(A)$, $M_+(A)$ and $P(A)$ the set of Borel measures, non-negative Borel measures and probability measures on $A$. In the vector-valued case, for $N \in \mathbb{N}$, we employ the notation $M(A,\R^N)$ and $M_+(A,\R^N)$, respectively. The corresponding total variation norm is denoted by $\|\cdot\|_{M(A)}$ or $\|\cdot\|_{M(A, \R^N)}$.
Further, we denote by $M_{[0,1]}(A)$ a Borel family of measures in $M(A)$ indexed by $t \in [0,1]$ and $M_{[0,1]}(A, \R^N)$ in the vector-valued case.
Let $T_p\Gamma$ be the tangent space of $\Gamma$ at the point $p\in\Gamma$. We represent any measure $V \in M(\Gamma,T_p\Gamma)$ as 
\begin{align*}
  V = \tau_\Gamma(p) \mathcal{V}
\end{align*}
where $\tau_\Gamma(p)$ is the tangent vector to $\Gamma$ in $p \in \Gamma$ and $\mathcal{V}  \in M(\Gamma, \R)$. Additionally, we represent any measure $G^i \in M(\{\gamma(i)\}, T_{\gamma(i)} \Gamma)$ as
\begin{align*}
   G^i = \tau_{\Gamma}(\gamma(i)) \mathcal{G}^i 
\end{align*}
where $\mathcal{G}^i \in  M(\{\gamma(i)\}, \R)$ for $i=0,1$.
We will consider time-dependent families of measures on $\Omega$ and $\Gamma$ together with their momentums. In particular, we summarize here the different objects relevant in our analysis: 
\begin{itemize}
    \item[i)] $\rho_t \in M_{[0,1]}(\Omega)$ measures the flow of observables moving in $\Omega$ and $J_t \in M_{[0,1]}(\Omega,\R^d)$ its momentum.
    \item[ii)] $\mu_t \in M_{[0,1]}(\Gamma)$ measures the flow of observables moving in $\Gamma$ (with a different dynamic) and $V_t \in M_{[0,1]}(\Gamma,T_p\Gamma)$ its momentum.
    \item[iii)] $f_t \in M_{[0,1]}(\Gamma)$ measures the normal exchange rate of mass between $\Omega$ and $\Gamma$.
    \item[iv)] $G_t^0 \in M_{[0,1]}(\{\gamma(0)\},T_{\gamma(0)}\Gamma)$ and $G_t^1 \in M_{[0,1]}(\{\gamma(1)\},T_{\gamma(1)}\Gamma)$ measure the tangential exchange of mass between $\Omega$ and $\Gamma$ at the boundary of $\Gamma$.
\end{itemize}

\begin{remark}
Even if it holds that  $\mathcal{G}^i = c_i\delta_{\gamma(i)}$ and therefore we could identify it with the scalar $c_i \in \R$, we will prefer to keep the notation $\mathcal{G}^i \in  M(\{\gamma(i)\}, \R)$, since we believe that it helps better identifying the physical meaning of the quantities. 
\end{remark}

We also need extensions of measures defined in the following way.

\begin{definition}[Extension of measures]
\label{def:ex}
For $N \in \mathbb{N}_0$ and given $\mu \in M(\Gamma, \R^N)$ we define its zero extension to $\Omega$, denoted by $\bar \mu \in M(\Omega, \R^N)$ as
\begin{align*}
\bar \mu (\varphi) = \int_\Gamma \varphi \big|_{\Gamma} d \mu 
\end{align*}
for every $\varphi \in C(\Omega)$. Moreover, given $V \in M(\Gamma, T_p\Gamma)$ we define its zero extension (in $\Omega$ and in the normal direction) $\bar V \in M(\Omega, \R^d)$ as $\bar V = \tau_\Gamma(p) \bar{\mathcal{V}}$.
\end{definition}

This allows us to define the space of admissible tuples
$(J_t,V_t,\rho_t,\mu_t,f_t,G_t^0,G^1_t)$ as 
\begin{align*}
    \mathcal{D}_{\rm adm}(\Gamma)= M_{[0,1]}(\Omega,\R^d) \times M_{[0,1]}(\Gamma,T_p\Gamma) \times M_{[0,1]}&(\Omega) \times M_{[0,1]}(\Gamma) \times 
    M_{[0,1]}(\Gamma) \\
    & \times M_{[0,1]}(\{\gamma(0)\}, T_{\gamma(0)} \Gamma) \times M_{[0,1]}(\{\gamma(1)\}, T_{\gamma(1)} \Gamma)
\end{align*}
and the set of admissible initial measures as 
\begin{align}
    \mathcal{P}_{\rm adm}(\Gamma) = \{(\rho,\mu) \in M_+(\Omega) \times M_+(\Gamma) : \rho + \bar \mu \in P(\Omega)\}.
\end{align}

\subsection{Sobolev-Slobodeckij spaces}

In order to minimize over the preferential path in Section~\ref{OptPrefPath}, we need compactness in a suitable normed space. This space will turn out to be $C^{1,\alpha}([0,1];\R^d)$ for $\alpha\in (0,1)$. The compactness however will follow by a compact embedding from Sobolev-Slobodeckij spaces, also called fractional Sobolev spaces, into these Hölder spaces. Following \cite{DINEZZA2012521}, we give a short introduction to Sobolev-Slobodeckij spaces and state the necessary embedding theorem. 

\begin{definition}[Sobolev-Slobodeckij spaces]
    Let $U\subset \R^d$ be open and let $s\in(0,1)$ and $p\in[1,+\infty)$ be given. For $f:U\to\R$ we define the \emph{Gagliardo seminorm}
    \begin{align*}
        [f]^p_{W^{s,p}(U)} := \int\limits_U \int\limits_U \frac{| f(x) - f(y) |^p}{| x - y|^{d+sp}}\, dx dy
        \intertext{and the corresponding norm by}
        \| f\|^p_{W^{s,p}(U)} := \|f \|^p_{L^p(U)} + [f]^p_{W^{s,p}(U)}.
    \end{align*}
    The \emph{Sobolev-Slobodeckij space $W^{s,p}(U)$} is given as 
    \begin{align*}
        W^{s,p}(U) := \{ f \in L^p(U) \ : \ \|f\|_{W^{s,p}(U)} < +\infty \}
    \end{align*}
    which is a Banach space for the norm defined above. 
\end{definition}

These spaces act as an intermediate space between $L^p(U)$ and $W^{1,p}(U)$ and can be extended to arbitrary $k+s$ for $k\in\mathbb{N}_0$ and $s\in(0,1)$. 

\begin{definition}
    Let $U\subset \R^d$ be open and let $s\in(0,1)$, $k\in\mathbb{N}_0$ and $p\in[1,+\infty)$ be given. The space 
    \begin{align*}
        W^{k+s,p}(U) := \{f\in W^{k,p}(U) \ : \ D^{\vec{\alpha}} f \in W^{s,p}(U) \text{ for all } \vec{\alpha}\in\mathbb{N}^d \text{ s.t. } |\vec{\alpha}|=k \}
    \end{align*}
    equipped with the norm
    \begin{align*}
        \| f \|_{W^{k+s,p}(U)} := \left(\|f\|^p_{W^{k,p}(U)} + \sum_{|\vec{\alpha}|=k} \left[ D^{\vec{\alpha}} f \right]^p_{W^{s,p}(U)}\right)^\frac{1}{p}
    \end{align*}
    is a Banach space.
\end{definition}

If the mapping is vector-valued instead, the space is defined component-wise. 
The following embedding theorem holds true by \cite[Theorem 4.57, Theorem 4.58]{demengel_functional_2012}.

\begin{theorem}\label{CmpctEmbedding}
    Let $U\subset\R$ be a given domain with Lipschitz boundary. For $p\in[1,+\infty)$, $k\in\mathbb{N}_0$ and $s\in(0,1)$ such that $kp,sp > 1$ there exists a constant $C(k,s,p,U)>0$ with
    \begin{align*}
        \|f\|_{C^{k,\alpha}(U)} := \|f\|_{C^k(U)} + \sum\limits_{|\vec{\alpha}|=k}\sup\limits_{x\neq y} \frac{|D^{\vec{\alpha}}f(x) - D^{\vec{\alpha}}f(y)|}{|x-y|^\alpha} \leq C(k,s,p,U) \|f\|_{W^{k+s,p}(U)} 
    \end{align*}
    for $\alpha = s-\frac{1}{p}$ and $f\in W^{k+s,p}(U)$ and the embedding $W^{k+s,p}(U) \to C^{k,\alpha'}(U)$ is compact for $\alpha'<\alpha$.
\end{theorem}

For maps $f:U\to\R^N$ we obtain similar embeddings by considering each component separately.
The compactness of the embedding will be crucial for showing existence of minimizers for varying paths.

\section{Fixed preferential path}\label{sec:FixedPath}

In this section we consider a (fixed) preferential path $\gamma : [0,1] \rightarrow \Omega$, regular, non self-intersecting and of class $C^1([0,1];\Omega)$ (in short $C^1$-curve) with image denoted by $\Gamma \subset \R^d$.
We aim to model the dynamic of a population under the assumption that travelling on the preferential path $\Gamma$ is cheaper; however, accessing (or leaving) the preferential path is associated with an additional cost. Such a modelling choice can be justified by economic reasons (the highway is expensive to take) or time reasons (there is often much more traffic to access or leave the highway, for example in Italian highways).

\subsection{Continuity equations and mass exchange}\label{sec:conn}

In this section we define rigorously the coupled continuity equations anticipated in the introduction. Since we will prefer to define the continuity equation on $\Gamma$ using its parametrization, we will need the notion of push-forward of measures by $\gamma^{-1}$ defined as follows.

\begin{definition}\label{def:push}
Given a non self-intersecting curve $\gamma : [0,1] \rightarrow \Omega$ with image $\Gamma$ and a measure $\nu \in M(\Gamma, \mathbb{R}^N)$ for $N \in \mathbb{N}$, we define $\nu^\gamma \in M([0,1], \mathbb{R}^N)$ as 
\begin{align*}
    \nu^\gamma =  \gamma^{-1}_{\#} \nu.
\end{align*}
\end{definition}
We are now ready to state the continuity equations that regulate the exchange of mass between $\Omega$ and $\Gamma$.  

\begin{definition}[Coupled continuity equations]\label{def:conteq}
Given a regular, non self-intersecting $C^1$-curve $\gamma : [0,1] \rightarrow \Omega$ with image $\Gamma$  we say that $(J_t,V_t,\rho_t,\mu_t,f_t,G^0_t,G_t^1) \in  \mathcal{D}_{\rm adm}(\Gamma)$ satisfies the coupled  continuity equations with initial and final data $(\rho_0,\mu_0) \in \mathcal{P}_{\rm adm}(\Gamma)$ and $(\rho_1,\mu_1) \in \mathcal{P}_{\rm adm}(\Gamma)$ if 
\begin{align}\label{eq:cont1}
\int_0^1\int_{\Omega} \partial_t & \varphi(t,x)\, d\rho_t dt  + \int_0^1\int_{\Omega} \nabla \varphi(t,x) \cdot dJ_t dt - \int_0^1\int_\Gamma \varphi(t,x)\, df_t dt \nonumber\\
& - \int_0^1 \int_{\{\gamma(0)\}} \varphi(t,x)\,d \mathcal{G}_t^0 dt + \int_0^1 \int_{\{\gamma(1)\}} \varphi(t,x)\, \,d \mathcal{G}_t^1 dt = \int_\Omega \varphi(1,x) d\rho_1 - \int_\Omega \varphi(0,x) d\rho_0
\end{align}
for every $\varphi \in C^1(X_\Omega)$ and     
\begin{align}\label{eq:cont2}
\int_0^1& \int_{0}^1 \partial_t \psi(t,s)  \, d\mu_t^\gamma dt  + \int_0^1\int_0^1 \partial_s \psi(t,s) |\gamma'(s)|^{-1} \, d\mathcal{V}_t^\gamma dt + \int_0^1\int_0^1 \psi(t,s)\, df_t^\gamma dt \nonumber \\
& + \int_0^1 \int_{\{0\}
}\psi(t,s)\,d (\mathcal{G}_t^0)^\gamma dt - \int_0^1 \int_{\{1\}}\psi(t,s)\,d (\mathcal{G}_t^1)^\gamma dt= \int_0^1 \psi(1,s) d\mu^\gamma_{1} - \int_0^1 \psi(0,s) d\mu^\gamma_{0}
\end{align}
for every $\psi \in C^1([0,1]^2)$. We refer to such conditions as ${\rm CE}(\Gamma)$ and if $(J_t,V_t,\rho_t,\mu_t,f_t,G^0_t,G^1_t)$ satisfies ${\rm CE}(\Gamma)$ we write $(J_t,V_t,\rho_t,\mu_t,f_t,G^0_t,G^1_t) \in {\rm CE}(\Gamma)$, where we keep implicit the dependence on initial and final data.
\end{definition}

\begin{remark}
We point out that we are allowing mass to be exchanged tangentially at the boundary of $\Gamma$ at rate $\mathcal{G}^0$ (entering $\Gamma$) and $\mathcal{G}^1$ (exiting $\Gamma$). Moreover, we also note that Definition \ref{def:conteq} is independent of the choice of the injective parametrization for $\Gamma$ but it might reverse the interpretation of $\mathcal{G}^0$ and $\mathcal{G}^1$. This justifies the notation ${\rm CE}(\Gamma)$. Moreover it holds that
\begin{align}\label{eq:geometry}
    \int_0^1\int_0^1 \partial_s \psi(t,s) |\gamma'(s)|^{-1} \, d\mathcal{V}_t^\gamma dt = \int_0^1\int_\Gamma \nabla_\Gamma \psi(t,p) \cdot dV_t dt
\end{align}
where $\nabla_\Gamma$ is the gradient of the function $x \mapsto \psi(t,x)$ on the embedded manifold $\Gamma$. This motivates that Definition \ref{def:conteq} is nothing but the parametrized version of \eqref{eq:CE_intro}. The identity \eqref{eq:geometry} can be easily justified noticing that a local base of the tangent plane of $\Gamma$ in $p \in \Gamma$ is given by $\frac{\partial}{\partial x}\vert_p = \gamma'(\gamma^{-1}(p))$. Therefore from the definition of the Riemannian gradient we obtain
\begin{align}
\nabla_\Gamma \psi(t,p) = g\left(\frac{\partial}{\partial x}\vert_p, \frac{\partial}{\partial x}\vert_p\right)^{-1}\partial_s \psi(t,s)\vert_{\gamma^{-1}(p)} \frac{\partial}{\partial x}\vert_p 
\end{align}
implying that $\nabla_\Gamma \psi(t,p) \cdot \tau_\Gamma(p) = \partial_s \psi(t,s)\vert_{\gamma^{-1}(p)} |\gamma'(\gamma^{-1}(p))|^{-1}$ and thus \eqref{eq:geometry}.
\end{remark}

Under the previous coupled continuity equations, it is possible to prove that the total mass is conserved and more specifically the total density satisfies a global continuity equation.

\begin{proposition}
Given a regular, non self-intersecting $C^1$-curve $\gamma : [0,1] \rightarrow \Omega$ with image $\Gamma$, suppose that $(J_t,V_t,\rho_t,\mu_t,f_t,G^0_t,G^1_t) \in {\rm CE}(\Gamma)$. Denoting by $\eta_t = \rho_t + \bar \mu_t$ and $W_t = J_t + \bar V_t$ it holds that 
    \begin{align}
        \partial_t \eta_t + \nabla \cdot W_t = 0 \quad \text{in } \ X_{\Omega}
    \end{align}
    distributionally with initial value $\rho_0 + \bar \mu_0$ and final value $\rho_1 + \bar \mu_1$.
\end{proposition}   
\begin{proof}
   Given an arbitrary test function $\varphi \in C^1(X_\Omega)$ and defining $\psi \in C^1([0,1]^2)$ as
\begin{align}
    \psi(t,s) = \varphi(t,\gamma(s))
\end{align}
it holds that $\partial_s \psi(t,s) = \nabla \varphi(t,\gamma(s)) \cdot \gamma'(s)$. Therefore, using the definition of push-forward, we obtain 
\begin{align*}
& \int_0^1\int_{\Omega} \partial_t \varphi(t,x)\, d\eta_t dt + \int_0^1\int_{\Omega} \nabla \varphi(t,x)\cdot dW_t dt  \\
& =  \int_0^1\int_{\Omega} \partial_t \varphi(t,x)\, d\rho_t dt +  \int_0^1\int_{\Omega} \partial_t \varphi(t,x)\, d\bar \mu_t dt + \int_0^1\int_{\Omega} \nabla \varphi(t,x) \cdot dJ_t dt +   \int_0^1\int_{\Omega} \nabla \varphi(t,x)\cdot d\bar V_t dt\\
& = \int_0^1\int_{\Gamma} \partial_t \varphi(t,x)\, d\mu_t dt + \int_0^1\int_{\Gamma} \nabla \varphi(t,x) \cdot \frac{\gamma'(\gamma^{-1}(x))}{|\gamma'(\gamma^{-1}(x))|} d \mathcal{V}_t dt+ \int_0^1\int_\Gamma \varphi(t,x)\, df_t dt \\
& \qquad + \int_\Omega \varphi(1,x) d\rho_1 - \int_\Omega \varphi(0,x) d\rho_0 + \int_0^1 \int_{\{\gamma(0)\}} \varphi(t,x)\,d \mathcal{G}_t^0 dt - \int_0^1 \int_{\{\gamma(1)\}} \varphi(t,x)\, \,d \mathcal{G}_t^1 dt\\
& = \int_0^1\int_{0}^1 \partial_t \psi(t,s)\, d \mu_t^\gamma dt + \int_0^1\int_{0}^1 \nabla  \varphi(t,\gamma(s)) \cdot \frac{\gamma'(s)}{|\gamma'(s)|}   d \mathcal{V}^\gamma_t dt + \int_0^1\int_{0}^1 \psi(t,s)\, df_t^\gamma dt \\ &\qquad + \int_\Omega \varphi(1,x) d\rho_1 - \int_\Omega \varphi(0,x) d\rho_0
 +  \int_0^1 \int_{\{0\}} \psi(t,s)\,d (\mathcal{G}_t^0)^\gamma dt - \int_0^1 \int_{\{1\}} \psi(t,s)\, \,d (\mathcal{G}_t^1)^\gamma dt\\
& = \int_\Gamma \psi(1,x) d\mu_1 - \int_\Gamma \psi(0,x) d\mu_0 + \int_\Omega \varphi(1,x) d\rho_1 - \int_\Omega \varphi(0,x) d\rho_0
\end{align*}
as we wanted to prove. 
\end{proof}

\subsection{Dynamic formulation}\label{sec:dynamicformulation}
In this section we define the variational problem that is regulating the evolution of the measures $(J_t,V_t,\rho_t,\mu_t,f_t,G^0_t,G^1_t) \in  \mathcal{D}_{\rm adm}(\Gamma)$ in time. Such measures should obey the coupled continuity equations ${\rm CE}(\Gamma)$ while minimizing an action functional.
In our case, it is constructed as a generalization of the kinetic energy due to the introduction of a mobility function, separately for the kinetic energy of $\rho$ on $\Omega$, the kinetic energy of $\mu$ on $\Gamma$ and the exchange of mass between $\Omega$ and $\Gamma$. 
We start with a definition of admissible mobility functions as introduced in \cite{Lisini2010}.
 \begin{definition} A function $m$ is called \emph{admissible mobility function} if  $m: [0,+\infty) \rightarrow [0,+\infty) \cup\{-\infty\}$ is an upper semicontinuous, concave function with $\domain(m)\coloneqq \{x\in[0,+\infty) : m(x) >-\infty \} =[a, b]$ for $a,b \geq 0$, $a<b$ and $m(z)>0$ for every $z \in (a, b)$. 

\end{definition}
\begin{remark}\label{rem:bound}
    The most prominent example of this class of mobilities is given by volume filling mobilities like $m(z)=z(1-z)$ defined in $[0,1]$ or more general $m(z) = (a-z)(b-z)$ defined in $[a,b]$ (in both cases extended to $-\infty$ for $z \notin [0,1]$ and $z \notin [a,b]$ respectively). These mobilities give rise to upper and lower bounds on the densities, implicitly imposing additional constraints. A more detailed discussion of mobilities can be found in e.g. \cite{Lisini2010}, \cite{Dolbeault_2008} or \cite{Carrillo2010}.
    Note in particular that for any admissible mobility $m$ it holds that 
    \begin{align}\label{eq:finite_sup}
        \sup_{z\in \rm dom(m)} m(z) <\infty.
    \end{align}
\end{remark}

We are now ready to introduce rigorously the dynamic variational problem, which is defined as a minimization of an action functional on tuples satisfying the coupled continuity equation.

\begin{definition}\label{def:adm_mobility}
Given a regular, non self-intersecting $C^1$-curve $\gamma : [0,1] \rightarrow \Omega$ with image $\Gamma$, we define the variational dynamic problem as 
    \begin{align}\label{eq:BB2}\tag{BB2}
& \inf_{(J_t, V_t, \rho_t, \mu, f, G_t^0, G_t^1) \in {\rm CE}(\Gamma)}  \int_0^1 \calA_\Gamma(J_t, \mathcal{V}_t, \rho_t, \mu_t, f_t, \mathcal{G}_t^0, \mathcal{G}_t^1)\, dt
\end{align}
where 
\begin{align}\label{eq:actionformula}
& \calA_\Gamma = \int_{\Omega}\Psi_\Omega\left(\frac{d\rho_t}{d\lambda_\Omega}, \frac{dJ_t}{d\lambda_\Omega} \right) \, d\lambda_\Omega 
+ \int_{\Gamma}\left[ \alpha_1\Psi_\Gamma\left(\frac{d\mu_t}{d\lambda_\Gamma},\frac{d\mathcal{V}_t}{d\lambda_\Gamma}\right) + \alpha_2\Psi_\Gamma\left(\frac{d\mu_t}{d\lambda_\Gamma}, \frac{df_t}{d\lambda_\Gamma}\right)\right]\,d\lambda_\Gamma\\
& \qquad + \alpha_3 \int_{\{\gamma(0)\} \cup \{\gamma(1)\}}\left[\Psi_\Gamma\left(\frac{d(\mu_t \mres\gamma(0))}{d\lambda_E},\frac{d\mathcal{G}_0}{d\lambda_E}\right) +\Psi_\Gamma\left(\frac{d(\mu_t \mres  \gamma(1))}{d\lambda_E},\frac{d\mathcal{G}_1}{d\lambda_E}\right)\right] d\lambda_E\nonumber
\end{align}
for $\alpha_1, \alpha_2, \alpha_3 > 0 $. In the previous definition $\lambda_\Omega \in M_+(\Omega)$ is any measure such that $\rho \ll \lambda_\Omega$ and $|J| \ll \lambda_\Omega$, $\lambda_\Gamma \in M_+(\Gamma)$ is any measure such that $\mu \ll \lambda_\Gamma$, $|\mathcal{V}_t| \ll \lambda_\Gamma$ and $f \ll \lambda_\Gamma$, and $\lambda_E \in M_+(\{\gamma(0)\} \cup \{\gamma(1)\})$ is any measure such that $\mu \mres  \gamma(i) \ll \lambda_E$ for $i=0,1$.  
Moreover, $\Psi_\Omega$ and $\Psi_\Gamma$ are action functionals defined as
\begin{align}
    \Psi_\Omega(z,w) = \left\{
    \begin{array}{cc}
         \frac{|w|^2}{m_\Omega(z)}& m_\Omega(z) \neq 0, z \in {\rm dom}(m_\Omega)  \\
         0 & w = 0 \text{ and } m_\Omega(z) = 0\\
         +\infty & (m_\Omega(z) = 0 \text{ and} \ w \neq 0) \text{ or } z \notin {\rm dom}(m_\Omega)
    \end{array}
    \right.
\end{align}
and
\begin{align}
    \Psi_\Gamma(z,w) = \left\{
    \begin{array}{cc}
         \frac{|w|^2}{m_\Gamma(z)}& m_\Gamma(z) \neq 0, z \in {\rm dom}(m_\Gamma)   \\
         0 & w = 0 \text{ and } m_\Gamma(z) = 0 \\
         +\infty &  (m_\Gamma(z) = 0 \text{ and} \ w \neq 0) \text{ or } z \notin {\rm dom}(m_\Gamma)
    \end{array}
    \right.
\end{align}
for admissible mobilities $m_\Omega$ and $m_\Gamma$ with domain bounded by $a_\Omega$, $b_\Omega$ and $a_\Gamma$, $b_\Gamma$ respectively.
\end{definition}

\begin{remark}
\label{rem:lsc}
     Owing to the results in \cite{Dolbeault_2008,Lisini2010} the functional $\calA$ defined in \eqref{eq:BB2} is convex and lower semicontinuous with respect to the narrow convergence of measures. In particular, for proving the narrow lower semicontinuity, we refer the reader to \cite[Lemma 3.9]{Dolbeault_2008} that can be applied to each term in \eqref{eq:actionformula} for a constant sequence of reference measures and using the fact that the narrow topology coincides with the weak one on compact domains.
\end{remark}

If $\calA_\Gamma$ is finite, we can conclude further regularity for $(J_t, V_t, \rho_t, \mu_t, f_t, G_t^0, G^1_t) \in {\rm CE}(\Gamma)$. 

\begin{proposition}\label{ExistenceDensities}

    Suppose that $\int_0^1 \calA_\Gamma(J_t,\mathcal{V}_t,\rho_t,\mu_t,f_t, \mathcal{G}^0_t, \mathcal{G}^1_t) dt< \infty$ for some $(J_t, V_t, \rho_t, \mu_t, f_t, G_t^0, G^1_t) \in {\rm CE}(\Gamma)$. Then $\rho_t$ and $\mu_t$  admit a narrowly continuous representative.
    Moreover, it holds that 
        \begin{align}
            a^\Omega \le \frac{d\rho_t}{d\lambda_\Omega} \le b^\Omega,\  \lambda_\Omega-a.e.,\qquad \text{and} \qquad  a^\Gamma \le \frac{d\mu_t}{d\lambda_\Gamma} \le  b^\Gamma, \ \lambda_\Gamma-a.e.
        \end{align}
        for every $t \in [0,1]$, where $(a^\Omega,b^\Omega)$ and $(a^\Gamma, b^\Gamma)$ are the respective constants for $m_\Omega$ and $m_\Gamma$ in Definition \ref{def:adm_mobility}. 
        \end{proposition}
\begin{proof}
 The proof follows as in \cite[Section 2, Definition 3]{Lisini2010} and \cite[Section 2, Lemma 1]{Lisini2010}.
\end{proof}

Note that for general mobilities it is not true that the infimum in \eqref{eq:BB2} is always finite \cite{Dolbeault_2008}. This depends on the interplay between the reference measures and the initial and final values. However, we point out that in case of standard dynamic optimal transport (i.e. with linear mobility) such a result follows from a similar argument to the one of \cite[Lemma 3.7]{monsaingeon2021new} that is based on the existence of Wasserstein geodesics and Fischer-Rao geodesics.

\subsection{Well-posedness of the dynamic formulation}\label{sec:WellPsdnFP}
In this section
we will employ the Direct method of Calculus of Variations to prove that \eqref{eq:BB2} admits at least one minimizer. 
The basic result needed for that is a compactness result (in time) for the time-dependent family of measures $\rho_t$ and $\mu_t$. In particular, the fundamental ingredient of such compactness result is the following Hölder estimate for $\rho_t$ and $\mu_t$. 
\begin{proposition}[Hölder estimates]\label{prop:holder}
    Given a regular, non self-intersecting $C^1$-curve $\gamma : [0,1] \rightarrow \Omega$, suppose that for $(J_t,V_t,\rho_t,\mu_t,f_t,G^0_t,G^1_t) \in {\rm CE(\Gamma)}$ there holds that $\int_0^1 \calA_\Gamma(J_t,\mathcal{V}_t,\rho_t,\mu_t,f_t, \mathcal{G}^0_t,\mathcal{G}^1_t)dt < M$ for a constant $M  > 0$. Then there exists a constant $C(\lambda_\Omega, \lambda_\Gamma, b^\Omega, b^\Gamma,M) >0$ only depending on the choice of reference measures, $b^\Omega$, $b^\Gamma$ and $M$, such that for all $s,t \in [0,1]$ it holds that
    \begin{align}\label{eq:est1}
        \|\rho_t - \rho_s\|_{C^1(\Omega)^*} &\leq  C(\lambda_\Omega, \lambda_\Gamma, b^\Omega, b^\Gamma,M)|t-s|^{1/2},
        \\\label{eq:est2}
        \|\mu^\gamma_t - \mu^\gamma_s\|_{C^1([0,1])^*} &\leq C(\lambda_\Omega, \lambda_\Gamma, b^\Omega, b^\Gamma,M)|t-s|^{1/2}.
    \end{align}
    Here, $C^1(\Omega)^*$ denotes the dual space of $C^1(\Omega)$ and $C^1([0,1])^*$ the dual space of $C^1([0,1])$.
\end{proposition}

\begin{proof}
    The proof follows from classical arguments (see for example \cite[Proposition 3.5]{monsaingeon2021new} or \cite[Lemma A.2]{bredies2020optimal}). However, we will sketch it for estimate \eqref{eq:est2} since in our notation it involves the push-forward of the measures on the interval $[0,1]$. Given $\varphi \in C^1([0,1])$ note that the measurable map $t \mapsto \int_0^1 \varphi(s) d\mu^\gamma_t$ is weakly differentiable in $(0,1)$. Indeed for $\xi \in C_c(0,1)$ it holds that
    \begin{align*}
        \int_0^1\int_0^1 \xi'(t) \varphi(s) d\mu^\gamma_t(s)\, dt & = - \int_0^1\int_0^1  \xi(t) \varphi'(s) |\gamma'(s)|^{-1}d\mathcal{V}^\gamma_t\, dt - \int_0^1\int_0^1 \xi(t) \varphi(s) df_t^\gamma \\
        & = \int_0^1\int_0^1   \xi(t)\left(\varphi'(s) |\gamma'(s)|^{-1}d\mathcal{V}^\gamma_t - \varphi(s) df^\gamma_t\right)\,dt,
    \end{align*}
    implying that the weak derivative is equal to $t \mapsto \int_0^1 \varphi'(s) |\gamma'(s)|^{-1}d\mathcal{V}^\gamma_t - \varphi(s) df^\gamma_t$. Now, let $s,t\in[0,1]$ be given and suppose that $s\leq t$. By the fundamental theorem of calculus, it holds that 
    \begin{align*}
    |\mu^\gamma_t(\varphi) - \mu^\gamma_s(\varphi)| & \leq \int_s^t \int_0^1 |\varphi'(x)| |\gamma'(x)|^{-1}d\mathcal{V}^\gamma_\tau + |\varphi(x)| df^\gamma_\tau\, d\tau\\
        & \leq \Tilde{C}\|\varphi\|_{C^1([0,1])} \int_s^t \int_\Gamma \left|\frac{d\mathcal{V}_\tau}{d\lambda_\Gamma}\right| + \left| \frac{df_\tau}{d\lambda_\Gamma}\right|\, d\lambda_\Gamma\, d\tau \\
        & \leq \sqrt{2} \| \lambda_\Gamma\|^{\frac{1}{2}}_{M(\Gamma)}\Tilde{C}\|\varphi\|_{C^1([0,1])} |t-s|^{1/2} \left(\int_0^1 \int_\Gamma\left|\frac{d\mathcal{V}_\tau}{d\lambda_\Gamma}\right|^2 + \left| \frac{df_\tau}{d\lambda_\Gamma}\right|^2\, d\lambda_\Gamma\, d\tau \right)^{1/2}\\
        & \leq \sqrt{2} \| \lambda_\Gamma\|^{\frac{1}{2}}_{M(\Gamma)}\Tilde{C}\|\varphi\|_{C^1([0,1])} |t-s|^{1/2} \bigg(\int_0^1 \int_\Gamma m_\Gamma\left(\frac{d\mu_\tau}{d\lambda_\Gamma}\right)\Psi_\Gamma\left(\frac{d\mu_\tau}{d\lambda_\Gamma}, \frac{d\mathcal{V}_\tau}{d\lambda_\Gamma}\right) \\
        & \qquad \qquad \qquad \qquad \qquad \qquad \qquad \qquad \qquad + m_\Gamma\left(\frac{d\mu_\tau}{d\lambda_\Gamma}\right)\Psi_\Gamma\left(\frac{d\mu_\tau}{d\lambda_\Gamma}, \frac{df_\tau}{d\lambda_\Gamma}\right)\, d\lambda_\Gamma\, d\tau\bigg)^{1/2}\\
        & \leq C(\lambda_\Omega, \lambda_\Gamma, b^\Omega, b^\Gamma,M) \|\varphi\|_{C^1([0,1])} |t-s|^{1/2}
    \end{align*}
     where we used Hölder's inequality and Proposition \ref{ExistenceDensities} together with the properties of $m_\Gamma$. This concludes the proof. 
\end{proof}

The previous proposition is the main ingredient for the following compactness result.
\begin{theorem}[Compactness]\label{thm:compact}
Given a regular, non self-intersecting $C^1$-curve $\gamma : [0,1] \rightarrow \Omega$ and a sequence $(J_t^n,V_t^n,\rho_t^n,\mu_t^n,f_t^n,(G_t^0)^n,(G_t^1)^n) \in {\rm CE}(\Gamma)$ such that 
\begin{align*}
    \sup_{n\in \mathbb{N}} \int_0^1 \calA_\Gamma(J_t^n,\mathcal{V}_t^n,\rho_t^n,\mu_t^n,f_t^n,(\mathcal{G}^0_t)^n,(\mathcal{G}_t^1)^n)dt < \infty,
\end{align*}
there exists $(J_t,V_t,\rho_t,\mu_t,f_t, G_t^0,G_t^1) \in {\rm CE}(\Gamma)
$ such that, up to subsequences, 
\begin{itemize}
    \item [i)] $\rho^n_t \rightharpoonup \rho_t$ narrowly in $\Omega$ for all $t \in [0,1]$
    \item [ii)] $\mu^n_t \rightharpoonup \mu_t$ narrowly in $\Gamma$ for all $t \in [0,1]$
\end{itemize}
and
\begin{itemize}
    \item [iii)] $J^n \rightharpoonup J$  narrowly in $[0,1] \times \Omega$
    \item [iv)] $\mathcal{V}^n \rightharpoonup \mathcal{V}$ narrowly in $[0,1] \times \Gamma$
    \item [v)] $f^n  \rightharpoonup f$ narrowly in $[0,1] \times \Gamma$
    \item [vi)] $(\mathcal{G}^i)^n \rightharpoonup \mathcal{G}^i$ narrowly in $[0,1] \times \{\gamma(i)\}$ for $i=1,2$
\end{itemize}
Here, the un-subscripted quantities denote the measures defined on the product spaces given by their respective disintegration.

\end{theorem}

\begin{proof}
We start by proving $i)$ and $ii)$. First note that due the definition of $\mathcal{A}_\Gamma$ and the narrow continuity of $\rho^n_t$ and $\mu^n_t$ inferred from Proposition \ref{ExistenceDensities} it holds that $\rho_t^n \ll \lambda_{\Omega}$ and $\mu_t^n \ll \lambda_{\Gamma}$ for every $t\in [0,1]$. Moreover, thanks to Proposition \ref{ExistenceDensities} it is also true that $a^\Omega \le \frac{d\rho^n_t}{d\lambda_\Omega} \le b^\Omega$ $\lambda_\Omega$-a.e. and $a^\Gamma \le \frac{d\mu^n_t} {d\lambda_\Gamma} \le  b^\Gamma$ $\lambda_\Gamma$-a.e. for every $t \in [0,1]$ and every $n$. In particular
    \begin{align}
        \sup_n\sup_t \|\rho_t^n\|_{M(\Omega)} < \infty, \quad \sup_n\sup_t \|\mu_t^n\|_{M(\Gamma)} < \infty
    \end{align}
    and thus, by the properties of the push-forward we also have
    \begin{align}
       \sup_n\sup_t \|(\mu_t^\gamma)^n\|_{M([0,1])} < \infty.
    \end{align}
     Therefore, thanks to Proposition \ref{prop:holder} and by applying a generalized Arzelà-Ascoli theorem (\cite[Proposition A.4]{bredies2020optimal}) we conclude that there exists a narrowly continuous family of Borel measures $\rho_t \in M(\Omega)$ and $\mu^\gamma_t \in M([0,1])$ such that the following narrow convergence holds
     \begin{align}\label{eq:convpu}
        \rho^n_t \rightharpoonup \rho_t, \quad (\mu^\gamma_t)^n \rightharpoonup \mu^\gamma_t, \quad \text{for every } t \in [0,1]. 
     \end{align}
     Define now $\mu_t = \gamma_\# \mu^\gamma_t \in M(\Gamma)$. Moreover, $\mu_t^n \rightharpoonup \mu_t$ for every $t \in [0,1]$. Indeed,
     \begin{align*}
         \lim_{n\to\infty} \int_\Gamma \varphi((\gamma \circ \gamma^{-1})(x)) d(\mu^n_t - \mu_t) = \lim_{n\to\infty}\int_0^1 \varphi(\gamma(s))d((\mu^\gamma_t)^n - \mu^\gamma_t) = 0.  
     \end{align*}
     
     We now prove $iii)$, $iv)$, $v)$ and $vi)$. 
    Starting with $iv)$, recall that $|\mathcal{V}_t^n| \ll \lambda_\Gamma$ for almost every $t \in [0,1]$. Then 
    \begin{align}
        \int_0^1 \int_\Gamma \left|\frac{d\mathcal{V}^n_t}{d\lambda_\Gamma}\right| d \lambda_\Gamma\, dt &\leq \int_0^1 \int_\Gamma \sqrt{m_\Gamma\left(\frac{d\mu^n_{t}}{d\lambda_\Gamma}\right)}\sqrt{\Psi_\Gamma\left(\frac{d\mu^n_{t}}{d\lambda_\Gamma}, \frac{d\mathcal{V}^n_{t}}{d\lambda_\Gamma} \right)}d\lambda_\Gamma\, dt \notag\\
        & \leq \left(\int_0^1 \int_\Gamma  m_\Gamma\left(\frac{d\mu^n_{t}}{d\lambda_\Gamma}\right) d\lambda_\Gamma dt\right)^{1/2}
       \left(\int_0^1 \int_\Gamma  \Psi_\Gamma\left(\frac{d\mu^n_{t}}{d\lambda_\Gamma}, \frac{d\mathcal{V}^n_{t}}{d\lambda_\Gamma} \right) d\lambda_\Gamma dt\right)^{1/2}. \label{eq:equibnd}
    \end{align}
   Therefore from the fact that $a^\Gamma\leq \frac{d\mu^n_{t}}{d\lambda_\Gamma} \leq b^\Gamma$ $\lambda_\Gamma$-a.e. and Remark \ref{rem:bound}, it follows that $\sup_n |\mathcal{V}^n|([0,1]\times\Gamma) <\infty$. Thus, by Prokhorov's theorem, there exists $\mathcal{V} \in M([0,1]\times\Gamma)$ such that, up to subsequence, $\mathcal{V}^n \rightharpoonup \mathcal{V}$. Since the uniform bounds for the total variation of $J^n$, $f^n$ and $\mathcal{G}_i^n$ are obtained similarly, the remaining convergences $iii)$, $v)$ and $vi)$ follow using similar arguments. It remains to show that $(J_t,V_t,\rho_t,\mu_t,f_t, G_t^0,G_t^1) \in {\rm CE}(\Gamma)$. First note that $\mathcal{V}^n$ is equibounded by \eqref{eq:equibnd} and equi-integrable by the same estimates. Therefore we have that 
   for every Borel $I \subset [0,1]$ it holds that $|\mathcal{V}|(I\times \Gamma) = \int_{I} p(t)$ for a suitable $p\in L^1([0,1])$ (see \cite[Lemma 4.5]{Dolbeault_2008} or \cite[Theorem 8.1]{gamkrelidze2013principles}).  
   This implies by the disintegration theorem \cite[Theorem 2.28]{ambrosio2000functions} that $\mathcal{V} \in M_{[0,1]}(\Gamma)$.
   With similar reasoning it holds that $f \in M_{[0,1]}(\Gamma)$ and $\mathcal{G}^i \in M_{[0,1]}(\{\gamma(i)\})$, so that $(J_t,V_t,\rho_t,\mu_t,f_t, G_t^0,G_t^1) \in \mathcal{D}_{\rm adm}(\Gamma)$.   
  It remains to prove that $(J_t,V_t,\rho_t,\mu_t,f_t, G_t^0,G_t^1)$ solves the coupled system \eqref{eq:cont1} and \eqref{eq:cont2}. Note that the narrow convergence $\mathcal{V}^n \rightharpoonup \mathcal{V}$ and $f^n \rightharpoonup f$ implies the narrow convergence of the push-forwards $(\mathcal{V}^\gamma)^n \rightharpoonup \mathcal{V}^\gamma$ and $(f^\gamma)^n \rightharpoonup f^\gamma$ (respectively) due to the continuity of $\gamma$. Therefore, by additionally using \eqref{eq:convpu} and the fact that $\gamma$ is assumed to be regular, passing to the limit in both continuity equations in Definition \ref{def:conteq} we conclude that $(J_t,V_t,\rho_t,\mu_t,f_t, G_t^0,G_t^1) \in \rm CE(\Gamma)$. 
\end{proof}
We are now ready to prove existence of minimizers for \eqref{eq:BB2}.

\begin{theorem}\label{ExistenceFixedGamma}
Given a regular, non self-intersecting $C^1$-curve $\gamma : [0,1] \rightarrow \Omega$, suppose that \eqref{eq:BB2} is finite. Then there exists a minimizer $(J_t,V_t,\rho_t,\mu_t,f_t, G_t^0,G_t^1) \in {\rm CE}(\Gamma)$ of the dynamic problem \eqref{eq:BB2}.  
\end{theorem}

\begin{proof}
    The proof follows directly from Theorem \ref{thm:compact} and Remark \ref{rem:lsc} by the application of the Direct Method of Calculus of Variations. Given a minimizing sequence $(J_t^n,V_t^n,\rho_t^n,\mu_t^n,f_t^n,(G_t^0)^n,(G_t^1)^n) \in {\rm CE}(\Gamma)$, up to extracting a subsequence, it holds that 
    \begin{align}
    \sup_{n\in \mathbb{N}} \int_0^1 \calA_\Gamma(J_t^n,\mathcal{V}_t^n,\rho_t^n,\mu_t^n,f_t^n,(\mathcal{G}_t^0)^n,(\mathcal{G}_t^1)^n)dt  < \infty.
    \end{align}
   We can then apply Theorem \ref{thm:compact} to extract a subsequence (not relabelled) such that the convergence $i) - vi)$ holds. Then, note that since $\mathcal{A}_\Gamma \geq 0$, by Fatou's lemma it holds that 
   \begin{align*}
       \liminf_{n\to\infty} \int_0^1 \calA_\Gamma(J_t^n,\mathcal{V}_t^n,\rho_t^n,\mu_t^n & ,f_t^n,(G_t^0)^n,(G_t^1)^n)dt \geq \int_0^1 \liminf_{n\to\infty} \calA_\Gamma(J_t^n,\mathcal{V}_t^n,\rho_t^n,\mu_t^n,f_t^n,(\mathcal{G}_t^0)^n,(\mathcal{G}_t^1)^n)dt \\
       & \geq \int_0^1 \calA_\Gamma(J_t,\mathcal{V}_t,\rho_t,\mu_t,f_t,\mathcal{G}_t^0,\mathcal{G}_t^1)dt
   \end{align*}
   where the last inequality follows from Remark \ref{rem:lsc}. This concludes the proof applying the Direct Methods of Calculus of Variations.
\end{proof}

\section{Parameter scaling limits}\label{ParamLim}

In this section, following \cite[Section 6]{monsaingeon2021new}, we aim to investigate the asymptotic behaviour of \eqref{eq:BB2} with respect to the parameters $\alpha_2$ and $\alpha_3$ establishing a connection to uncoupled transport problems.

We start by considering $\alpha_2, \alpha_3 \to+\infty$. The proof of the limit follows the arguments in \cite{monsaingeon2021new}. However, we will state the proof again since we are dealing with additional coupling mechanisms and mobilities here. First we define the natural limiting objects:
\begin{align}\label{eq:nopar}
W^2_\Omega(\rho_0,\rho_1) =
\inf_{(J_t,\rho_t) \in {\rm CE}_\Omega }\int_0^1\int_{\Omega}\Psi_\Omega\left(\frac{d\rho_t}{d\lambda_\Omega}, \frac{dJ_t}{d\lambda_\Omega} \right) \, d\lambda_\Omega,
\end{align}
\begin{align}\label{eq:nopar2}
W^2_\Gamma(\mu_0,\mu_1) =
\inf_{(V_t,\mu_t) \in {\rm CE}_\Gamma }\int_0^1\int_{\Gamma} \Psi_\Gamma\left(\frac{d\mu_t}{d\lambda_\Gamma},\frac{d\mathcal{V}_t}{d\lambda_\Gamma}\right) \, d\lambda_\Gamma
\end{align}
where the constraints $(J_t,\rho_t) \in {\rm CE}_\Omega$ and $(V_t,\mu_t) \in {\rm CE}_\Gamma$ means that the pairs $(J_t,\rho_t)$ and $(V_t,\mu_t)$ satisfy the continuity equation distributionally on $\Omega$ and in $\Gamma$ respectively. The quantities \eqref{eq:nopar} and \eqref{eq:nopar2} denote the classical uncoupled Wasserstein distances on the two domains $\Omega$ and $\Gamma$ with the non-linear mobilities $m_\Omega$ and $m_\Gamma$ from Definition \ref{def:adm_mobility}.
Note that without redefining it, we will make the dependence on the parameters $\alpha_i$ in $\mathcal{A}_\Gamma$ explicit as $\mathcal{A}_\Gamma^{\alpha_1,\alpha_2,\alpha_3}$.
In what follows we will say that we have compatible mass for the initial and final data if $\|\mu_0\|_{M(\Gamma)} = \| \mu_1\|_{M(\Gamma)}$ (by definition of $\mathcal{P}_{\rm adm}(\Gamma)$ this is equivalent to $\|\rho_0\|_{M(\Omega)} = \|\rho_1\|_{M(\Omega)}$).
We also remark that we often assume the limiting infima $W^2_\Gamma(\mu_0,\mu_1)$ and $W^2_\Omega(\rho_0,\rho_1)$ to be finite. This implies, a fortiori, that we have compatible mass. However, for sake of clarity we mention both assumptions.

\begin{theorem}
   Let $\gamma:[0,1]\to\Omega$ be a regular, non self-intersecting $C^1$-curve and $\alpha_1>0$. Suppose that we have compatible masses for the initial and final data, $W^2_\Omega(\rho_0,\rho_1) < \infty$ and $W^2_\Gamma(\mu_0,\mu_1) < \infty$. Then for every sequence of weights $(\alpha_2^n,\alpha_3^n)_{n\in\mathbb{N}}$ with $\alpha_2^n, \alpha_3^n\to\infty$ as $n\to\infty$ it holds that 
    \begin{align}\label{eq:thesislimit}
        \inf_{(J_t, V_t, \rho_t, \mu_t, f_t, G_t^0, G_t^1) \in {\rm CE}(\Gamma)}  \int_{0}^{1} \calA_\Gamma^{\alpha_1, \alpha^n_2, \alpha^n_3}(J_t, \mathcal{V}_t, \rho_t, \mu_t, f_t, \mathcal{G}_t^0, \mathcal{G}_t^1)\; dt \overset{n\to\infty}{\longrightarrow} W^2_{\Omega}(\rho_0, \rho_1) + \alpha_1 W^2_\Gamma (\mu_0, \mu_1).
    \end{align}
    Moreover, the minimizing pairs $(\rho^n_t, J^n_t)$ and $(\mu^n_t, V^n_t)$ converge as in $i)$, $ii)$, $iii)$, $iv)$ (see Theorem \ref{thm:compact}), up to subsequences, to minimizers  of \eqref{eq:nopar} and \eqref{eq:nopar2}.
    
If the masses are not compatible, it holds that for every sequence of weights $(\alpha_2^n,\alpha_3^n)$
    \begin{align*}
        \inf_{(J_t, V_t, \rho_t, \mu_t, f_t, G_t^0, G_t^1) \in {\rm CE}(\Gamma)}  \int_{0}^{1} \calA_\Gamma^{\alpha_1, \alpha^n_2, \alpha^n_3}(J_t, \mathcal{V}_t, \rho_t, \mu_t, f_t, \mathcal{G}_t^0, \mathcal{G}_t^1)\; dt\overset{n\to\infty}{\longrightarrow} \infty.
    \end{align*}
\end{theorem}

\begin{proof}
    We start by showing a lower bound on the action. 
    First, let 
    \begin{align*}
        m_{\sup} &:= \sup\limits_{z\in[a_\Gamma, b_\Gamma]} m_\Gamma(z)
    \end{align*}
    which is finite by \eqref{eq:finite_sup}. 
    Suppose now that for given $\alpha_1,\alpha^n_2,\alpha^n_3$ the infimum \eqref{eq:BB2} is finite. Then, thanks to Theorem \ref{ExistenceFixedGamma}, there exists a minimizer of \eqref{eq:BB2} named $(J^n_t, V^n_t, \rho^n_t, \mu^n_t, f^n_t, (G^0_t)^n, (G^1_t)^n)$. In this case, 
    using the non-negativity of all occurring terms, Proposition \ref{ExistenceDensities} and Jensen's inequality we obtain the estimate
    \begin{align*}
        &\inf_{(J_t, V_t, \rho_t, \mu_t, f_t, G^0_t, G^1_t) \in {\rm CE}(\Gamma)}  \int_{0}^{1} \calA_\Gamma^{\alpha_1, \alpha_2^n, \alpha_3^n}(J_t, \mathcal{V}_t, \rho_t, \mu_t, f_t, \mathcal{G}_t^0, \mathcal{G}_t^1)\; dt \\
        \geq& \alpha_2^n \int_0^1 \int_\Gamma \Psi_\Gamma\left(\frac{d\mu^n_t}{d\lambda_\Gamma}, \frac{d f^n_t}{d\lambda_\Gamma} \right) d\lambda_\Gamma dt \\
        & \qquad \qquad + \alpha_3^n \int_0^1 \int_{\{\gamma(0)\} \cup \{\gamma(1)\}}\Psi_\Gamma\left(\frac{d(\mu^n_t \mres  \gamma(0))}{d\lambda_E}, \frac{d (\calG^0_t)^n}{d\lambda_E} \right) + \Psi_\Gamma\left(\frac{d(\mu^n_t \mres  \gamma(1))}{d\lambda_E}, \frac{d (\calG^1_t)^n}{d\lambda_E} \right) d\lambda_E dt\\
        \geq& \frac{\alpha_2^n}{m_{\sup}} \int_0^1 \int_\Gamma \left| \frac{df^n_t}{d\lambda_\Gamma} \right|^2 d\lambda_\Gamma dt + \frac{\alpha_3^n}{m_{\sup}} \int_0^1 \int_{\{\gamma(0)\} \cup \{\gamma(1)\}}\left| \frac{d(\calG^0_t)^n}{d\lambda_E} \right|^2 + \left| \frac{d(\calG^1_t)^n}{d\lambda_E}  \right|^2 d\lambda_E dt\\
        \geq& \frac{\min(\alpha_2, \alpha_3)}{m_{\sup} \max(\|\lambda_\Gamma\|_{M(\Gamma)}, \|\lambda_E\|_{M(\{\gamma(0)\} \cup \{\gamma(1)\})})} \left( \int_0^1 \|f^n_t\|_{M(\Gamma)}^2 + \|(\calG_t^0)^n\|_{M(\{\gamma(0)\})}^2 + \|(\calG^1_t)^n\|_{M(\{\gamma(1)\})}^2 \,dt \right).
    \end{align*}
    Moreover, testing \eqref{eq:cont2} with $\psi\equiv 1$ implies
    \begin{align*}
        &\int_0^1\|f^n_t\|^2_{M(\Gamma)} + \|(\calG^0_t)^n\|^2_{M(\{\gamma(0)\})} + \|(\calG^1_t)^n\|^2_{M(\{\gamma(1)\})}\, dt \\
        \geq& \frac{1}{3}\left(\int_0^1\|f^n_t\|_{M(\Gamma)} + \|(\calG^0_t)^n\|_{M(\{\gamma(0)\})} + \|(\calG^1_t)^n\|_{M(\{\gamma(1)\})}\, dt\right)^2 \\
        \geq& \frac{1}{3}\left| \|\mu_1\|_{M(\Gamma)} - \|\mu_0\|_{M(\Gamma)} \right|^2.
    \end{align*}
    Combining these estimates yields the divergence for incompatible masses as $\alpha^n_2, \alpha^n_3\to\infty$. 

    Now, we assume the compatibility condition to hold true. Let $(\rho^*_t,J^*_t)$ and $(\mu^*_t, V^*_t)$ be minimizers of \eqref{eq:nopar}, \eqref{eq:nopar2} that exist thanks to the boundedness of the infima and following similar arguments to the ones from Theorem \ref{ExistenceFixedGamma}. 
    Note that $(J^*_t, V^*_t, \rho^*_t, \mu^*_t, 0,0,0) \in {\rm CE}(\Gamma)$. Therefore
    \begin{align}\label{eq:bbo}
        \inf_{(J_t, V_t, \rho_t, \mu_t, f_t, G^0_t, G^1_t) \in {\rm CE}(\Gamma)}   \int_{0}^{1} \calA_\Gamma^{\alpha_1, \alpha_2^n, \alpha_3^n}(J_t, \mathcal{V}_t, \rho_t, \mu_t, f_t, \mathcal{G}_t^0, \mathcal{G}_t^1)\; dt &\leq W^2_{\Omega}(\rho_0, \rho_1) + \alpha_1 W^2_\Gamma (\mu_0, \mu_1),
    \end{align}
    implying that for $(\alpha_2^n,\alpha_3^n)$ the infimum \eqref{eq:BB2} is finite as well.
    Combining this estimate with the one established above yields
    \begin{align}
        &\int_0^1 \| f_t^n\|_{M(\Gamma)}^2 + \| (\calG_t^0)^n\|_{M(\{\gamma(0)\})}^2 + \| (\calG_t^1)^n\|_{M(\{\gamma(1)\})}^2 \, dt \nonumber \\
        \leq
        & \frac{m_{\sup} \max(\|\lambda_\Gamma\|_{M(\Gamma)}, \|\lambda_E\|_{M(\{\gamma(0)\} \cup \{\gamma(1)\})})}{\min(\alpha_2^n, \alpha_3^n)} \nonumber \\
        & \qquad \qquad \qquad \qquad \inf_{(J_t, V_t, \rho_t, \mu_t, f_t, G^0_t, G^1_t) \in {\rm CE}(\Gamma)}  \int_{0}^{1} \calA_\Gamma^{\alpha_1, \alpha_2, \alpha_3}(J_t, \mathcal{V}_t, \rho_t, \mu_t, f_t, \mathcal{G}_t^0, \mathcal{G}_t^1)\; dt \label{eq:boundes}\\
        \leq& \frac{m_{\sup} \max(\|\lambda_\Gamma\|_{M(\Gamma)}, \|\lambda_E\|_{M(\{\gamma(0)\} \cup \{\gamma(1)\})})}{\min(\alpha_2^n, \alpha_3^n)} \left( W^2_{\Omega}(\rho_0, \rho_1) + \alpha_1 W^2_\Gamma (\mu_0, \mu_1) \right) \notag
    \end{align}
    recalling that $f_t^{n}, (\calG_t^0)^{n}, (\calG_t^1)^{n}$ correspond to the action minimizer for the parameter choices $\alpha^n_2, \alpha^n_3>0$, that exist due to Theorem \ref{ExistenceFixedGamma}.  
    From \eqref{eq:boundes}, and since $\alpha^n_2,\alpha^n_3 \rightarrow \infty$, we can conclude boundedness of the sequence of measures $(f^n)_n$, $\left((\mathcal{G}^0)^{n}\right)_n$ and $\left((\mathcal{G}^1)^{n}\right)_n$, allowing us to extract a (relabelled) converging subsequence such that $f^{n}\rightharpoonup 0$ and $(G^0)^{n}$, $(G^1)^{n}\rightharpoonup 0$ narrowly. 
    From \eqref{eq:bbo} and using Theorem \ref{thm:compact}, we are able to conclude the existence of another relabelled subsequence $(J_t^n,V_t^n,\rho_t^n,\mu_t^n,f_t^n,(G_t^0)^n,(G_t^1)^n) \in {\rm CE}(\Gamma)$ converging as in Theorem \ref{thm:compact}.  
   Finally, from the lower-semicontinuity of the action from Remark \ref{rem:lsc} and estimate \eqref{eq:bbo} we deduce \eqref{eq:thesislimit}, up to subsequences. Note that since our argument can be applied to any starting subsequence obtaining the same limit, we conclude that \eqref{eq:thesislimit} holds for the full sequence.
\end{proof}

Next, we consider $\alpha_2, \alpha_3\to 0$. We expect a similar behaviour to \cite[Theorem 7]{monsaingeon2021new}. Due to different costs for transport in $\Omega$ and along $\Gamma$, we will not be able to simplify the limit to be the geodesic on the whole domain. Both terms will remain in the limit. 

\begin{theorem}
     Let $\gamma:[0,1]\to\Omega$ be a regular, non self-intersecting $C^1$-curve,  $\alpha_1>0$ and $(\alpha_2^n, \alpha_3^n) \rightarrow (0,0)$ for $n\to\infty$. Then, it holds that 
    \begin{align*}
        & \inf_{(J_t, V_t, \rho_t, \mu_t, f_t, G^0_t, G^1_t) \in {\rm CE}(\Gamma)}  \int_{0}^{1} \calA_\Gamma^{\alpha_1, \alpha^n_2, \alpha^n_3}(J_t, \mathcal{V}_t, \rho_t, \mu_t, f_t, \mathcal{G}_t^0, \mathcal{G}_t^1)\; dt \overset{n\to\infty}{\longrightarrow} \\
        & \qquad \qquad \qquad \qquad \qquad \qquad \qquad \inf_{(J_t, V_t, \rho_t, \mu_t, f_t, G^0_t, G^1_t) \in {\rm CE}(\Gamma)}   \int_{0}^{1} \calA_\Gamma^{\alpha_1, 0, 0}(J_t, \mathcal{V}_t, \rho_t, \mu_t, f_t, \mathcal{G}_t^0, \mathcal{G}_t^1)\; dt
    \end{align*}
    Moreover, the minimizers converge, up to subsequences, as in Theorem \ref{thm:compact}.
\end{theorem}

\begin{proof}
    If we are able to verify the assumptions of the fundamental theorem of $\Gamma$-convergence, the claim would follow. To do so, we restrict the problem to the case $\alpha_2, \alpha_3\in (0,1)$ and we verify the following statements:
    \begin{enumerate}
        \item The sequence of functionals $\int_{0}^{1} \calA_\Gamma^{\alpha_1, \alpha_2^n, \alpha_3^n}(J_t, \mathcal{V}_t, \rho_t, \mu_t, f_t, \calG_t^0, \calG_t^1)\, dt$ $\Gamma$-converges to the functional $\int_{0}^{1} \calA_\Gamma^{\alpha_1, 0, 0}(J_t, \mathcal{V}_t, \rho_t, \mu_t, f_t, \calG_t^0, \calG_t^1)\, dt$ as $\alpha_2, \alpha_3\to 0$.
        \item For all $\Lambda>0$ there exists a compact set $K_\Lambda$ such that $\{\int_{0}^{1} \calA_\Gamma^{\alpha_1, \alpha_2, \alpha_3}(J_t, \mathcal{V}_t, \rho_t, \mu_t, f_t, \calG_t^0, \calG_t^1)\; dt \leq \Lambda\}\subset K_\Lambda$ for all $0<\alpha_2, \alpha_3<1$.
    \end{enumerate}
    In order to prove $\Gamma$-convergence, we use the inequality
    \begin{align}\label{eq:ee}
        \int_{0}^{1} \calA_\Gamma^{\alpha_1, 0, 0}(J_t, \mathcal{V}_t, \rho_t, \mu_t, f_t, \calG_t^0, \calG_t^1)\, dt \leq \int_{0}^{1} \calA_\Gamma^{\alpha_1, \alpha_2, \alpha_3}(J_t, \mathcal{V}_t, \rho_t, \mu_t, f_t, \calG_t^0, \calG_t^1)\, dt
    \end{align}
    which holds true by non-negativity of each summand. 
    Thus, for any $(J_n^{n}, \mathcal{V}_t^{n}, \rho_t^{n}, \mu_t^{n}, f_t^{n}, ({G_t^0)^{n}, (G_t^1)^{n}})$ converging according to $i) - vi)$ (see Theorem \ref{thm:compact}) to $(J_t, \mathcal{V}_t, \rho_t, \mu_t, f_t, G_t^0, G_t^1)$, the lower semicontinuity of the action (Remark \ref{rem:lsc}), \eqref{eq:ee} and Fatou's Lemma imply that
    \begin{align*}
        \int_{0}^{1} \calA_\Gamma^{\alpha_1, 0, 0}(J_t, \mathcal{V}_t, \rho_t, \mu_t, f_t, {\calG_t^0}, {\calG_t^1})\; dt &\leq \liminf\limits_{n}\int_{0}^{1} \calA_\Gamma^{\alpha_1, 0, 0}(J_t^{n}, \mathcal{V}_t^{n}, \rho_t^{n}, \mu_t^{n}, f_t^{n}, (\calG_t^0)^{n}, (\calG_t^1)^{n})\; dt \\
        &\leq \liminf\limits_{n} \int_{0}^{1} \calA_\Gamma^{\alpha_1, \alpha^n_2, \alpha^n_3}(J_t^{n}, \mathcal{V}_t^{n}, \rho_t^{n}, \mu^{n}, f^{n}, (\calG_t^0)^{n}, (\calG_t^1)^{n})\; dt .
    \end{align*}
    Moreover, for any solution of the coupled continuity equations we are able to choose the constant sequence as a recovery sequence, proving $\Gamma$-convergence. 
    The equi-compactness follows from the inclusion of the sublevel sets
    \begin{align*}
    \left\{\int_{0}^{1} \calA_\Gamma^{\alpha_1, \alpha_2, \alpha_3}(J_t, \mathcal{V}_t, \rho_t, \mu_t, f_t, \calG_t^0, \calG_t^1)\; dt \leq \Lambda \right\}&\subset \left\{\int_{0}^{1} \calA_\Gamma^{\alpha_1, 1, 1}(J_t, \mathcal{V}_t, \rho_t, \mu_t, f_t, \calG_t^0, \calG_t^1)\; dt \leq \Lambda \right\}
    \end{align*}
    for $\alpha_2, \alpha_3 < 1$ and arbitrary $\Lambda>0$. From Theorem \ref{thm:compact} we conclude compactness of these sets with respect to the convergences $i) - vi)$. This allows us to apply \cite[Corollary 7.20]{dal2012introduction} and concludes the proof. 
\end{proof}

\section{Optimization over the preferential path}\label{OptPrefPath}

Here, we consider the dynamic problem where we allow for an optimal placement of the preferential path $\gamma$ according to initial and final distributions.
For technical reasons, we will consider curves defined in the interval $[0,L]$ for $L>0$, instead of $[0,1]$.
We assume initial and final data to be given by the probability measures $\eta_0, \eta_1 \in \mathcal{P}(\Omega)$, making them independent of the choice of $\Gamma$. By defining $\mu_0 := \eta_0 |_\Gamma$, $\rho_0 := \eta_0 - \bar \mu_0$ for any specific choice of $\Gamma\subset\Omega$ and the pair $(\mu_1, \rho_1)$ in the same way, we obtain an admissible pair of initial and final data as in the case of a fixed preferential path. 
Note that choosing different $\gamma$ modifies the space of admissible distributions and vector fields in $\rm CE (\Gamma)$.

Under the same assumptions of the previous section we consider the following variational problem.
\begin{definition}[Dynamic problem for varying preferential path]
Given a functional $R:C^{1}([0,L];\Omega) \to \R$ we consider the minimization problem
\begin{align}\label{eq:BBgamma}\tag{$BB_\gamma$}
\inf_{\gamma \in C^{1}([0,1];\Omega)}\inf _{(J_t, V_t, \rho_t, \mu_t, f_t, G_t^0, G_t^1) \in {\rm CE(\Gamma)}}  \int_{0}^{1} \calA_\Gamma (J_t, \mathcal{V}_t, \rho_t, \mu_t, f_t, \mathcal{G}_t^0, \mathcal{G}_t^1)\;d t + c R(\gamma)
\end{align}
for a parameter $c>0$, where we use initial and final data constructed as in the beginning of this section. 
\end{definition}

Our goal is to choose $R$ to satisfy the following properties: 
\begin{itemize}
    \item[i)] It is repulsive, meaning that it penalizes self-intersections of the curve. 
    \item[ii)] It preserves regularity in the sense that curves such that $R(\gamma) < \infty$ are continuously differentiable, regular and embedded and therefore parametrize a $C^1$-manifold.
    \item[iii)] It prevents closing of the curve at the endpoints in order for the boundary conditions to be well-defined. We could get rid of this condition if we impose the additional constraint $\calG^0_t = \calG^1_t$. In this case, the coupling terms in the weak formulation cancel for closed paths and continuous test functions. 
\end{itemize}
While different choices are possible, in the remaining part of this section we will consider the penalty 
\begin{align}\label{eq:penalty}
    R (\gamma) := 2^{-p}E_{\rm tp}^p(\gamma) -\log|\gamma(0) - \gamma(L)| + \int_0^{L} |\gamma'(t)| \, dt.
\end{align}
The first summand $E_{\rm tp}^p$ is called Tangent-Point energy of $\gamma$, see for example \cite{STRZELECKI_2012}, and it is responsible for preventing the self-intersection of the curve. We will define it rigorously in the next subsection. This choice is particularly suited for numerical implementations as discussed in Subsection \ref{DiscTPE}.
Moreover, we add the additional term $ -\log|\gamma(0) - \gamma(L)|$ prohibiting the closedness of the curve. This is added for technical as well as for physical reasons, since we expect the optimal path not to be closed. Finally we penalize the length of $\gamma$.
The constant parameter $c>0$ balances the penalty against the action functional, allowing for more interesting geometries than linear curves, which are favoured by the Tangent-Point energy.

We make the following assumptions regarding the behaviour of the action on varying curves.

\begin{assumption}\label{ass:additional}
Assume that 
\begin{enumerate}
\item The mobility $m_\Gamma$ (and consequently $\Psi_\Gamma$) is independent of the path $\Gamma$.
\item $\sup_{\gamma} \|\lambda_\Gamma\|_{M(\Gamma)} < \infty$, $\sup_{\gamma}  \|\lambda_E\|_{M(\{\gamma(0)\} \cup \{\gamma(L)\})} < \infty$, where both supremums are taken over the set 
\begin{align*}
    \{\gamma: \gamma \in C^1([0,L],\Omega), \text{regular and non self-intersecting}\}.
\end{align*}
\item For every $\gamma^n \rightarrow \gamma$ uniformly, it holds that $\lambda_{\Gamma^n} \rightharpoonup \lambda_{\Gamma}$, $\lambda_{E^n} \rightharpoonup \lambda_{E}$.
\end{enumerate}
\end{assumption}

\begin{remark}
    Note that the first condition assumes that the mobilities on preferential paths are always the same. This is a natural assumption since we expect that preferential paths are build in the same way. We believe that more general assumptions are possible, however we do not explore them in this paper. The second and third condition are instead of purely technical nature and necessary to retain compactness and lower semicontinuity of the action functional for varying preferential paths. 
    To be more precise, compactness and lower semicontinuity will follow from the structure of the action functional similar to the case of fixed preferential paths. The additional assumptions then ensure that these properties are compatible with varying domains. 
    One particular choice of such measures can be obtained through a reference measure on the domain $[0,L]$. Then, defining $\lambda_\Gamma$ and $\lambda_E$ using the pushforward under any curve $\gamma\in C([0,L];\Omega)$ is admissible. However, by not restricting the analysis to this specific case, we are able to include different reference measures as well.
\end{remark}

\subsection{The Tangent-Point energy for non-closed curves}\label{TPEcurves}

In this section we recall the definition of the Tangent-Point energy \cite{STRZELECKI_2012} stating and adapting the preliminary results needed in our analysis. Given $L>0$ and a curve $\gamma \in C^1([0,L]; \Omega)$ with image $\Gamma$, we denote by $r_{\rm tp}^{\gamma}(x,y)$ the radius of the unique circle tangent to $x\in\Gamma$ and passing through $x,y\in\Gamma$, see Figure~\ref{fig:TangentPointCircle} for an illustration. We use the convention that this radius is set to be zero if $x=y$, and is infinite if the vector $x-y$ is parallel to the tangent of $\Gamma$ in $x$. Moreover, note that $r_{\rm tp}^{\gamma}$ is defined only almost everywhere due to the possible presence of self-intersections. For $p>2$ the tangent point energy is defined as follows.

\begin{definition}[Tangent-Point energy of a curve]
Given $\gamma \in C^1([0,L]; \Omega)$ we define its Tangent-Point energy as
\begin{align}
E_{\rm tp}^p (\gamma) := \int_0^L \int_0^L\frac{|\gamma'(s)||\gamma'(t)|}{r_{\rm tp}^{\gamma}(\gamma(s), \gamma(t))^p} \, ds dt.
\end{align}
\end{definition}
Note that this energy is independent of the parametrization and could be defined for $\Gamma$ instead. However in our setting it is useful to work with an explicit parametrization. Additionally, the Tangent-Point energy is non-negative and attains its minimum for straight lines.

In \cite[Theorem 1.1]{STRZELECKI_2012} it was shown that the Tangent-Point energy is repulsive, meaning that for arclength parametrized curves finite energy implies that the image of the curve is a one dimensional embedded submanifold, possibly with boundary, that has no self-intersections. 
Regarding regularity, \cite{Blatt2013} has shown that for closed curves finiteness of $E^p_{\rm tp}$ is equivalent to the finiteness of the $W^{2-\frac{1}{p},p}$-seminorm. In the next proposition we will closely follow the proof of \cite[Proposition 2.5]{blatt_regularity_2012} to prove such a bound for non-closed curves.

\begin{proposition}[Energy spaces of the penalty functional]\label{EnergySpaces}
    Let $\gamma:[0,L]\to\R^d$ be a given arclength parametrized $C^{1}([0,L];\R^d)$ curve with length $L>0$ and $R(\gamma) \leq M$ for a constant $M>0$.
    Then $\gamma\in W^{2-\frac{1}{p},p}([0,L];\R^d)$, its image is an embedded one dimensional submanifold of $\R^d$ with boundary and the following estimate holds 
    \begin{align}\label{eq:estimate}
        [\gamma']^p_{W^{1 - \frac{1}{p}, p}([0,L];\R^d)} \leq C(p,M)
    \end{align}
    for a constant $C(p,M) > 0$ only depending on $p$ and $M$ and in particular independent of $\gamma$.
\end{proposition}
\begin{proof} 
    By \cite[Theorem 1.1]{STRZELECKI_2012} the image $\Gamma = \gamma([0,L])$ is a one dimensional embedded submanifold that has a boundary because of $R(\gamma)\leq M$.

    The additional regularity of the parametrization is shown as in \cite[Section 2]{Blatt2013} or \cite[Proposition 2.5]{blatt_regularity_2012} using open curves instead of closed ones. We will state the proof again for completeness.
    Given $\gamma \in C^{1}([0,L];\R^d)$ let $0<\delta \leq L$ be the biggest constant such that for all $s,t\in[0,L]$ with $|s-t|\leq \delta$
    \begin{align*}
        |\gamma'(s) - \gamma'(t)| \leq \frac{\sqrt{2}}{2} .
    \end{align*}
    Note that, by continuity of $\gamma'$, either $\delta=L$ or there exists a pair $s^*,t^*\in[0,L]$ with $|s^*-t^*| = \delta$ and $|\gamma'(s^*) - \gamma'(t^*) | = \frac{\sqrt{2}}{2}$. As in the proof of \cite[Proposition 2.5]{blatt_regularity_2012} for $|s-t|\leq \delta$ it holds that
    {\begin{align}
       \frac{3}{4} |s-t| |\gamma'(s) - \gamma'(t)| \leq |\left(\Pi_{\gamma(s)} - \Pi_{\gamma(t)}\right)(\gamma(s) - \gamma(t))|, 
    \end{align}
    }
    where
    \begin{align}
    \Pi_{\gamma(\tau)}v = v - \skp{v, \gamma'(\tau)}\gamma'(\tau)
    \end{align}
    is the projection onto the normal space to $\Gamma$ at $\gamma(\tau)$ for $\tau\in[0,L]$.
    
    Now, following e.g. \cite{blatt_regularity_2012} we are able to give an explicit representation of the tangent point radius as
    \begin{align}\label{eq:repradius}
        r_{\rm tp}^{\gamma}(\gamma(s), \gamma(y)) = \frac{|\gamma(s) - \gamma(t)|^2}{2{\rm dist}(\gamma(s)-\gamma(t), T_{\gamma(s)}\Gamma)} = \frac{|\gamma(s) - \gamma(t)|^2}{2|\Pi_{\gamma(s)}\left( \gamma(s) - \gamma(t)\right)|}.
    \end{align}
    Therefore we can estimate the energy $E_{\rm tp}^p(\gamma)$ by
    \begin{align}
        E^p_{\rm tp}(\gamma) &\geq \int_0^L \int_0^L \frac{2^{p}|\Pi_{\gamma(s)}\left( \gamma(s) - \gamma(t)\right)|^{p}}{|\gamma(s) - \gamma(t)|^{2p}} \, ds dt \nonumber\\
        &= \frac{2^p}{2}\left(\int_0^L \int_0^L \frac{|\Pi_{\gamma(s)}\left( \gamma(s) - \gamma(t)\right)|^{p}}{|\gamma(s) - \gamma(t)|^{2p}} \, ds dt + \int_0^L \int_0^L \frac{|\Pi_{\gamma(t)}\left( \gamma(t) - \gamma(s)\right)|^{p}}{|\gamma(s) - \gamma(t)|^{2p}} \, ds dt \right)  \nonumber\\
        &\geq 2^{p-1}\left(\int_0^L \int\limits_{\{|s-t|\leq\delta\}} \frac{|\Pi_{\gamma(s)}\left( \gamma(s) - \gamma(t)\right)|^{p}}{|\gamma(s) - \gamma(t)|^{2p}} \, ds dt + \int_0^1 \int\limits_{\{|s-t|\leq\delta\}} \frac{|\Pi_{\gamma(t)}\left( \gamma(t) - \gamma(s)\right)|^{p}}{|\gamma(s) - \gamma(t)|^{2p}} \, ds dt \right)  \nonumber\\
        &\geq \int_0^L \int\limits_{\{|s-t|\leq\delta\}} \frac{|(\Pi_{\gamma(s)} - \Pi_{\gamma(t)})(\gamma(s) - \gamma(t))|^p}{|\gamma(s) - \gamma(t)|^{2p}}\, ds dt  \nonumber\\
        &\geq \int_0^L \int\limits_{\{|s-t|\leq\delta\}} \left( \frac{3}{4} \right)^p |s-t|^p \frac{|\gamma'(s) - \gamma'(t)|^p}{|\gamma(s) - \gamma(t)|^{2p}}\, ds dt  \nonumber\\
        &= \left( \frac{3}{4} \right)^p \int_0^L \int\limits_{\{|s-t|\leq\delta\}} \frac{|\gamma'(s) - \gamma'(t)|^p}{|s-t|^{p}} \left( \frac{|s-t|}{|\gamma(s) - \gamma(t)|} \right)^{2p}\, ds dt  \nonumber\\
        &\geq  \left( \frac{3}{4} \right)^p\int_0^L \int\limits_{\{|s-t|\leq\delta\}} \frac{|\gamma'(s) - \gamma'(t)|^p}{|s-t|^{p}}\, ds dt, \label{eq:estdelta}
    \end{align}
    where the last inequality follows because of the choice of arclength parametrization which gives $|\gamma(s) - \gamma(t)| \leq |s-t|$.
    In conclusion we obtain 
    \begin{align}
        \int_0^L \int_0^L \frac{|\gamma'(s) - \gamma'(t)|^p}{|s-t|^p}\, ds dt &= \int_0^L \int_{\{|s-t|\leq \delta\}} \frac{|\gamma'(s) - \gamma'(t)|^p}{|s-t|^p}\, ds dt + \int_0^L \int_{\{|s-t|> \delta\}} \frac{|\gamma'(s) - \gamma'(t)|^p}{|s-t|^p}\, ds dt \nonumber \\
        &\leq \left( \frac{3}{4} \right)^pE^p_{\rm tp}(\gamma) + \int_0^L \int\limits_{\{|s-t|>\delta\}} \frac{2}{|s-t|^p}\, ds dt\nonumber \\
        &< \left( \frac{3}{4} \right)^p E^p_{\rm tp}(\gamma) + \frac{2L^2}{\delta^p}\label{eq:boundelta}
    \end{align}
    implying $\gamma \in W^{2-\frac{1}{p}, p}([0,L];\R^d)$.

    In order to conclude, we need to find a uniform lower bound for $\delta>0$.
    In the following, we will assume that $\delta < L$ (otherwise the result is immediate) and we will show that for a suitable constant $C(\alpha)>0$ the estimate
    \begin{align}{\label{eq:supbound}}
        |\gamma'(s) - \gamma'(t)| &\leq C(\alpha) E_{\rm tp}^p (\gamma)^\frac{1}{p} |s-t|^\alpha
    \end{align}
    holds for every $s,t \in [0,L]$ such that $|s-t| \leq \delta$ where $\alpha = 1- \frac{2}{p}$. This estimate would imply
    \begin{align*}
        \frac{\sqrt{2}}{2} = |\gamma'(s^*) - \gamma'(t^*)| &\leq C(\alpha) E_{\rm tp}^p (\gamma)^\frac{1}{p} \delta^\alpha  \leq C(\alpha) M^\frac{1}{p}\delta^\alpha
    \end{align*}
    and therefore the lower bound  
    \begin{align*}
        \delta &\geq \left( \frac{\sqrt{2} }{2C(\alpha)} \right)^\frac{1}{\alpha} M^q
    \end{align*}
     with $q = \frac{1}{p\alpha} = \frac{1}{p-2}$. Together with \eqref{eq:boundelta} this gives the desired estimate \eqref{eq:estimate}.
    
    For $s\in[0,L]$ and $r>0$ we denote the integral mean by 
    \begin{equation*}
        \gamma'_{B_r(s)} := \frac{1}{|B_r(s) \cap (0,L)|}\int\limits_{B_r(s)  \cap (0,L)} \gamma'(y) \, dy
    \end{equation*}
    for $B_r(s)$ denoting the open ball of radius $r$ centered at $s$. Note that $2r \geq |B_r(s) \cap (0,L)|\geq r$ for $r\in[0,L]$.
    The proof of \eqref{eq:supbound} is divided into three steps. The first step is an estimate of $\int_{B_r(s)\cap(0,L)} |\gamma'(t) - \gamma'_{B_r(s)}| \, dy$ from above for $r \in (0,\frac{\delta}{2})$. 
    Then, we will use this bound to estimate $|\gamma'(\cdot+\tau) - \gamma'(\cdot)|$ for $\tau<\frac{\delta}{4}$ and in the last step, we extend this estimate for arbitrary $\tau\leq \delta$. 
    
    \textbf{Step 1:}
    For any $s\in[0,L]$, $r\in(0,\frac \delta 2)$ and a generic constant $C>0$ only depending on $p$ it holds that 
    \begin{align*}
        &\frac{1}{| B_r(s) \cap (0,L) |} \int\limits_{B_r(s) \cap (0,L)}|\gamma'(y) - \gamma'_{B_r(s)}|  \, dy \\
        =& \frac{1}{| B_r(s) \cap (0,L) |} \int\limits_{B_r(s) \cap (0,L)}\left|\gamma'(y) - \frac{1}{| B_r(s) \cap (0,L) |}\int\limits_{B_r(s) \cap (0,L)} \gamma'(z) \, dz \right| \, dy\\
        \leq& \frac{1}{| B_r(s) \cap (0,L) |^2} \int\limits_{B_r(s) \cap (0,L)} \int\limits_{B_r(s) \cap (0,L)} |\gamma'(y) - \gamma'(z) | \, dy dz\\
        \leq& \left( \frac{1}{| B_r(s) \cap (0,L) |^2} \int\limits_{B_r(s) \cap (0,L)} \int\limits_{B_r(s) \cap (0,L)} |\gamma'(y) - \gamma'(z) |^p \, dy dz \right)^{\frac{1}{p}}\\
        \leq& 2| B_r(s) \cap (0,L) |^{1-\frac{2}{p}} \left(\, \int\limits_{B_r(s) \cap (0,L)} \int\limits_{B_r(s) \cap (0,L)} \frac{ |\gamma'(y) - \gamma'(z) |^p}{|y-z|^p} \, dy dz \right)^{\frac{1}{p}}\\
        \leq& C| B_r(s) \cap (0,L) |^{\alpha} \left(E_{\rm tp}^p(\gamma)\right)^{\frac{1}{p}},
    \end{align*}
    where we used the Hölder inequality together with the estimate \eqref{eq:estdelta} that is valid since $r\in(0,\frac \delta 2)$. 
    We thus obtain that for every $s\in [0,L]$ and every $r\in(0,\frac \delta 2)$
    \begin{align}\label{eq:intbound}
        \int\limits_{B_r(s) \cap (0,L)}|\gamma'(y) - \gamma'_{B_r(s)}|  \, dy &\leq C| B_r(s) \cap (0,L) |^{\alpha+1} \left(E_{\rm tp}^p(\gamma)\right)^{\frac{1}{p}}.
    \end{align}
    It remains to relate the left-hand side of \eqref{eq:intbound} to the supremum norm using standard Campanato estimates. 

    \textbf{Step 2:} As $\gamma\in C^1([0,L];\R^d)$, every $s\in[0,L]$ is a Lebesgue point of $\gamma'$. Therefore, for any $s,t\in[0,L]$ with $R:= |s-t| < \frac{\delta}{4}$ we obtain 
    \begin{align}
        |\gamma'(s) - \gamma'(t)| &\leq \left| \sum\limits_{k=0}^{+\infty} \left( \gamma'_{B_{R2^{-k}}(s)} - \gamma'_{B_{R2^{1-k}}(s)} + \gamma'_{B_{R2^{1-k}}(t)} - \gamma'_{B_{R2^{-k}}(t)}  \right) + \gamma'_{B_{2R}(s)} - \gamma'_{B_{2R}(t)} \right| \notag\\
        &\leq \sum\limits_{k=0}^{+\infty} \left| \gamma'_{B_{R2^{-k}}(s)} - \gamma'_{B_{R2^{1-k}}(s)}\right| + \sum\limits_{k=0}^{+\infty} \left| \gamma'_{B_{R2^{-k}}(t)} - \gamma'_{B_{R2^{1-k}}(t)}\right| + \left| \gamma'_{B_{2R}(s)} - \gamma'_{B_{2R}(t)} \right|. \label{eq:seriesexp}
    \end{align}
    Moreover, by \eqref{eq:intbound} it holds that 
    \begin{align*}
        &\left| \gamma'_{B_{2R}(s)} - \gamma'_{B_{2R}(t)} \right| \\
        \leq& \frac{1}{| B_{2R}(s) \cap B_{2R}(t) \cap(0,L) |} \left| \int\limits_{B_{2R}(s) \cap B_{2R}(t) \cap(0,L) } \gamma'_{B_{2R}(s)} - \gamma'(z) + \gamma'(z) - \gamma'_{B_{2R}(t)} \, dz \right|\\
        \leq&  \frac{1}{| B_{2R}(s) \cap B_{2R}(t) \cap(0,L) |}\left( \int\limits_{B_{2R}(s)\cap(0,L) } \left|  \gamma'_{B_{2R}(s)} - \gamma'(z) \right|\, dz +  \int\limits_{B_{2R}(t)\cap(0,L) } \left| \gamma'(z) - \gamma'_{B_{2R}(t)} \right| \, dz \right)\\
        \leq& \frac{C}{| B_{2R}(s) \cap B_{2R}(t) \cap(0,L) |}\left( |B_{2R}(s) \cap (0,L)|^{\alpha+1} + |B_{2R}(t) \cap (0,L)|^{\alpha+1}  \right) \left(E_{\rm tp}^p(\gamma)\right)^{\frac{1}{p}}. 
    \end{align*}
    As $|B_{2R}(s) \cap (0,L)| \leq 4R = 4|s-t|$ and $|B_{2R}(s) \cap B_{2R}(t) \cap (0,L)| \geq 2R = 2|s-t|$ we get
    \begin{align}
        \left| \gamma'_{B_{2R}(s)} - \gamma'_{B_{2R}(t)} \right| &\leq 2^{2\alpha+2}C|s-t|^{\alpha} \left(E_{\rm tp}^p(\gamma)\right)^{\frac{1}{p}}. \label{eq:diffpoints}
    \end{align}
    Similarly we obtain
    \begin{align}
        \left| \gamma'_{B_{R2^{-k}}(s)} - \gamma'_{B_{R2^{1-k}}(s)} \right| &\leq \frac{C}{|B_{R2^{-k}}(s)\cap(0,L)|}\left( |B_{R2^{-k}}(s)\cap(0,L)|^{\alpha+1} + |B_{R2^{1-k}}(s)\cap(0,L)|^{\alpha+1} \right) \left(E_{\rm tp}^p(\gamma)\right)^{\frac{1}{p}} \notag \\
        &\leq 2^{\alpha+1}(1+2^{\alpha+1})C \left(R2^{-k}\right)^{\alpha} \left(E_{\rm tp}^p(\gamma)\right)^{\frac{1}{p}} \label{eq:diffrad}
    \end{align}
    for $R < \frac{\delta}{4}$ since $|B_{R2^{1-k}}(s)\cap(0,L)| \leq R2^{1-k}$ and $|B_{R2^{-k}}(s)\cap(0,L)| \geq R2^{-k}$ in this case.
    Together, \eqref{eq:seriesexp}, \eqref{eq:diffpoints} and \eqref{eq:diffrad} lead to 
    \begin{align}\label{eq:smalldistbound}
        |\gamma'(s) - \gamma'(t)| \leq \left( 2^{\alpha+1}(1+2^{\alpha+1}) \sum\limits_{k=0}^{+\infty} 2^{-\alpha k} + 2^{2\alpha + 2}\right)C|s-t|^\alpha \left(E_{\rm tp}^p(\gamma)\right)^{\frac{1}{p}} = \Tilde{C}|s-t|^\alpha \left(E_{\rm tp}^p(\gamma)\right)^{\frac{1}{p}}
    \end{align}
    for a constant $\Tilde{C}>0$ only depending on $\alpha$.

    \textbf{Step 3:} Suppose now that  $s,t\in[0,L]$ are such that $\delta \geq |s-t| \geq \frac{\delta}{4}$. Then we can find a partition $s=:s_0 < s_1 < \ldots < s_N := t$ with $|s_k - s_{k-1}| < \frac{\delta}{4}$. Applying \eqref{eq:smalldistbound} and using the triangle inequality and Jensen's inequality we get
    \begin{align*}
        |\gamma'(s) - \gamma'(t)| &\leq  \Tilde{C} \sum\limits_{k=1}^{N} |s_k - s_{k-1}|^\alpha \left(E_{\rm tp}^p(\gamma)\right)^{\frac{1}{p}} \leq \Tilde{C} N^{1-\alpha} |s-t|^\alpha \left(E_{\rm tp}^p(\gamma)\right)^{\frac{1}{p}}.
    \end{align*}
    Note that, since $|s-t| \leq \delta$, $N$ is bounded from above independent of $\gamma$ (e.g. taking $N=5$ with points defined as $s_k=s + (t-s)\frac{k}{5}$), thus implying a uniform bound and concluding the proof.
\end{proof}

\subsection{Existence of minimizers for varying preferential path}\label{WellPsdnVP}

We first generalize the compactness presented in Theorem \ref{thm:compact} and the lower semicontinuity result to sequences of measures defined on varying paths. This is the content of the next proposition.

\begin{proposition}[Compactness for varying paths]
\label{thm:compactVarPath}
Given $\gamma^n,\gamma : [0,1] \rightarrow \Omega$ regular, non self-intersecting $C^1$-curves and a sequence $(J_t^n,V_t^n,\rho_t^n,\mu_t^n,f_t^n,(G_t^0)^n,(G_t^1)^n) \in {\rm CE}(\Gamma^n)$ such that $\gamma^n \rightarrow \gamma$ uniformly and
\begin{align*}
    \sup_{n\in \mathbb{N}} \int_0^1 \calA_{\Gamma^n}( J_t^n,\mathcal{V}_t^n,\rho_t^n,\mu_t^n,f_t^n,(\mathcal{G}^0_t)^n,(\mathcal{G}_t^1)^n)dt < \infty,
\end{align*}
there exists $(J_t,V_t,\rho_t,\mu_t,f_t, G_t^0,G_t^1) \in \mathcal{D}_{\rm adm}(\Gamma)
$ such that, up to subsequences, 
\begin{itemize}
    \item[i)] $\rho^n_t \rightharpoonup \rho_t$ narrowly in $\Omega$, for every $t \in [0,1]$
    \item[ii)]  $\bar \mu^n_t  \rightharpoonup \bar \mu_t$ narrowly in $\Gamma$, for every  $t \in [0,1]$
\end{itemize} 
and 
\begin{itemize}
    \item[iii)] $J^n \rightharpoonup J \text{ narrowly in } [0,1] \times \Omega$
    \item[iv)] $\bar{\mathcal{V}}^n  \rightharpoonup \bar{\mathcal{V}}  \text{ narrowly in } [0,1] \times \Omega$
    \item[v)] $\bar f^n  \rightharpoonup \bar f  \text{ narrowly in } [0,1] \times \Omega$
    \item[vi)] $( \bar{\mathcal{G}}^i)^n \rightharpoonup  \bar{\mathcal{G}}^i \text{ narrowly in } [0,1] \times \Omega$ for $i=1,2$.
\end{itemize}
As before, the un-subscripted quantities denote the measures defined on the product spaces given by their respective disintegration.
\end{proposition}

\begin{proof}
    Since it does not depend on $\Gamma$, the proof of $i)$ and $iii)$ is similar to the case of fixed preferential paths. We now prove $ii)$. From the equiboundedness of the energy we can assume (up to extracting a subsequence) that $\mu^n_t \ll \lambda_{\Gamma^n}$  and thanks to Proposition \ref{ExistenceDensities} and Assumption \ref{ass:additional} it holds that $0 \leq \frac{d\mu^n_t}{d\lambda_{\Gamma^n}} \leq b^{\Gamma}$ for every $n\in \mathbb{N}$ and every $t \in [0,1]$. In particular, from Assumption \ref{ass:additional} we get
    \begin{align}
        \sup_{t} \|\bar \mu_t^n \|_{M(\Omega)} < \infty
    \end{align}
    independent of $n$. Moreover, applying Proposition \ref{prop:holder} implies 
    \begin{align}
        \|\mu_t^{\gamma^n} - \mu_s^{\gamma^n}\|_{C^1([0,1])^*} \leq C|t-s|^{1/2}
    \end{align}
    where the constant $C$ can be chosen independent of $n$ by Assumption \ref{ass:additional}.
  Therefore, thanks to Proposition \ref{ExistenceDensities} and by applying a generalized Arzelà-Ascoli theorem (\cite[Proposition A.4]{bredies2020optimal}) we conclude that there exists a narrowly continuous family of Borel measures $\mu^*_t \in M([0,1])$ such that 
     \begin{align}\label{eq:convweak}
     (\mu^n_t)^{\gamma^n} \rightharpoonup \mu^*_t, \quad \text{for every } t \in [0,1].  
     \end{align}
     Define now $\mu_t = \gamma_\# \mu^*_t \in M(\Gamma)$ and note that $\mu_t^\gamma = \mu^*_t$. Moreover, $\mu_t^n \res \Gamma_n \rightharpoonup \mu_t \res \Gamma$ for every $t \in [0,1]$. Indeed, for every $\varphi \in C(\Omega)$ it holds that 
    \begin{align*}
        \lim_n &  \left|\int_{\Gamma_n} \varphi(x) d\mu^n_t - \int_{\Gamma} \varphi(x) d\mu_t\right| = \lim_n \left| \int_{\Gamma} \varphi((\gamma \circ \gamma^{-1})(x)) d\mu_t - \int_{\Gamma^n} \varphi((\gamma^n \circ (\gamma^n)^{-1})(x)) d\mu^n_t\right|\\
        & = \lim_n\left| \int_0^1 \varphi(\gamma(s)) d(\mu^{\gamma}_t) - \int_0^1 \varphi(\gamma^n(s)) d(\mu_t^n)^{\gamma^n}\right| \\
        & \leq \lim_n\left| \int_0^1 \varphi(\gamma(s)) d(\mu^{\gamma}_t) - \int_0^1 \varphi(\gamma(s)) d(\mu_t^n)^{\gamma^n}\right| + \left| \int_0^1 |\varphi(\gamma(s)) - \varphi(\gamma^n(s))| d(\mu_t^n)^{\gamma^n}\right| \\
        & \leq   \lim_n\left| \int_0^1 \varphi(\gamma(s)) d(\mu^{\gamma}_t) - \int_0^1 \varphi(\gamma(s)) d(\mu_t^n)^{\gamma^n}\right| + \sup_{s \in [0,1]}  |\varphi(\gamma(s)) - \varphi(\gamma^n(s))| \|(\mu_t^n)^{\gamma^n}\|_{M([0,1])} = 0
    \end{align*}
   due to \eqref{eq:convweak}, the uniform convergence of $\gamma^n$ to $\gamma$ and the uniform continuity of $\varphi$.  
We now prove $iv)$.  Note that $|\mathcal{V}_t^n| \ll \lambda_{\Gamma^n}$ for almost every $t \in [0,1]$. Then 
    \begin{align}
        \int_{X_{\Gamma^n}} \left|\frac{d\mathcal{V}^n_t}{d\lambda_{\Gamma^n}}\right| d \lambda_{\Gamma^n}\, dt &\leq \int_{X_{\Gamma^n}}  \sqrt{m_{\Gamma}\left(\frac{d\mu^n_t}{d\lambda_\Gamma}\right)}\sqrt{\Psi_{\Gamma}\left(\frac{d\mu^n_t}{d\lambda_\Gamma},\frac{d\mathcal{V}^n_t}{d\lambda_{\Gamma^n}}\right) }d\lambda_\Gamma^n\, dt \\
        & \leq \left(\int_{X_{\Gamma^n}}  m_{\Gamma}\left(\frac{d\mu^n_t}{d\lambda_\Gamma^n}\right) d\lambda_{\Gamma^n} dt\right)^{1/2}
       \left(\int_{X_\Gamma}   \Psi_{\Gamma}\left(\frac{d\mu^n_t}{d\lambda_\Gamma^n},\frac{d\mathcal{V}^n_t}{d\lambda_{\Gamma^n}}\right) d\lambda_{\Gamma^n} dt\right)^{1/2}.
    \end{align}
 Due to Assumption \ref{ass:additional} it holds that $0\leq \frac{d\mu^n_t}{d\lambda_\Gamma^n} \leq b^{\Gamma}$. From Remark \ref{rem:bound}, the equiboundedness of the energy and item $2$ in Assumption \ref{ass:additional}, it follows that $\sup_n\|\bar{\mathcal{V}}^n\|_{M(X_\Omega)} = \sup_n |\mathcal{V}^n|(X_{\Gamma}) <\infty$. In particular, there exists $\tilde{\mathcal{V}} \in M(X_\Omega)$ such that $\bar{\mathcal{V}}^n \rightharpoonup \tilde{\mathcal{V}}$. Note that, thanks to the uniform convergence of $\gamma^n$ to $\gamma$, the measure $\tilde{\mathcal{V}}$ is supported in $X_\Gamma$. Therefore, by defining $\mathcal{V} := \tilde{\mathcal{V}} \res X_\Gamma \in M(X_\Gamma)$ it is immediate to see that $\bar{\mathcal{V}}^n   \rightharpoonup \bar{\mathcal{V}}$ as we wanted to prove. The proof of $v)$ and $vi)$ follows by similar arguments.

 Finally, using the uniform bounds on the total variation of $\mathcal{V}_t^n$, $f_t^n$ and $(\mathcal{G}_t^i)^n$, following the same arguments as in Theorem \ref{thm:compact} (see also \cite[Lemma 4.5]{Dolbeault_2008}), we obtain that  
 $V$,$f$ and $G^i$ disintegrate in time and thus $(J_t,V_t,\rho_t,\mu_t,f_t, G_t^0,G_t^1) \in \mathcal{D}_{\rm adm}(\Gamma)$ concluding the proof.
   \end{proof}

   \begin{proposition}[Lower semicontinuity of the action functional for varying paths]\label{prop:LowSemiContvar}
    Let $\gamma^n,\gamma\in C^{1}([0,1];\Omega)$ be non self-intersecting, regular $C^1$- curves such that $\gamma^n \rightarrow \gamma$ uniformly. Assume additionally that $(J_t^n,\mathcal{V}_t^n,\rho_t^n,\mu_t^n,f_t^n,(\mathcal{G}_t^0)^n,(\mathcal{G}_t^1)^n) \in {\rm CE}(\Gamma^n)$ converges to $(J_t,V_t,\rho_t,\mu_t,f_t,G_t^0,G_t^1) \in {\rm CE}(\Gamma)$ according to the convergences stated in Theorem \ref{thm:compactVarPath}. Then it holds that 
    \begin{align}
        \liminf_n \int_0^1 \calA_{\Gamma^n}(J_t^n,\mathcal{V}_t^n,\rho_t^n,\mu_t^n,f_t^n,(\mathcal{G}^0_t)^n,(\mathcal{G}_t^1)^n)dt \geq \int_0^1 \calA_{\Gamma}( J_t,\mathcal{V}_t,\rho_t,\mu_t,f_t,\mathcal{G}^0_t,\mathcal{G}_t^1)dt.
    \end{align}
\end{proposition}
\begin{proof}
Fix $t \in [0,1]$. It holds that 
\begin{align}
    \liminf_n \calA_{\Gamma^n}(J_t^n,\mathcal{V}_t^n,\rho_t^n,\mu_t^n,f_t^n,(\mathcal{G}^0_t)^n,(\mathcal{G}_t^1)^n)dt & \geq \calA_{\Gamma}( J_t,\mathcal{V}_t,\rho_t,\mu_t,f_t,\mathcal{G}^0_t,\mathcal{G}_t^1)\nonumber 
\end{align}
thanks to Assumption \ref{ass:additional} and \cite[Section 2, Theorem 1]{Lisini2010}. The statement then follows from the application of Fatou's lemma since $\mathcal{A}_{\Gamma^n}$ and $\mathcal{A}_{\Gamma}$ are non-negative. 
\end{proof}

We now prove the lower semicontinuity of the regularizer $R$. 

\begin{proposition}[Lower semicontinuity of the regularizer]\label{LowSemiCont}
    Let $\gamma^n\in C^{1}([0,1];\Omega)$ be a given family of curves with $\sup_{n\in \mathbb{N}}R(\gamma^n)\leq M$ for a constant $M>0$ and let $\gamma\in C^{1}([0,1];\Omega)$ be given such that $\gamma^n\to\gamma$ in $C^{1}([0,1];\Omega)$ as $n\to\infty$. Then it holds that
    \begin{align*}
        R(\gamma) &\leq \liminf\limits_{n\to+\infty} R(\gamma^n).
    \end{align*}
\end{proposition}

\begin{proof}
    By $C^{1}([0,1];\Omega)$-convergence it holds that $\gamma^n \rightarrow \gamma$ and ${\gamma^n}' \rightarrow \gamma'$ uniformly. Therefore, we can apply the dominated convergence theorem to conclude the convergence of the length functionals. Moreover, by continuity and the uniform bound on $R$ it holds that $\log{|\gamma^n(0) - \gamma^n(1)|} \to \log{|\gamma(0) - \gamma(1)|}$. 
    It remains to show lower semicontinuity of the Tangent-Point energy. 
    It suffices to show pointwise convergence of the integrand almost everywhere, in particular outside of the diagonal of $[0,1]\times[0,1]$, in order to apply Fatou's lemma. Since $\sup_{n\in \mathbb{N}}E^p_{\rm tp}(\gamma^n)\leq M$ we know that for each $n\in\mathbb{N}$ the set $\{s,t\in[0,1]\ | \ r_{\rm tp}^{\gamma^n}(s,t) = 0 \}$ is of Lebesgue measure zero and thus the countable union 
    \begin{align*}
        \calN := \bigcup\limits_{n\in\mathbb{N}} \{s,t\in[0,1]\ | \ r_{\rm tp}^{\gamma^n}(s,t) = 0 \}
    \end{align*}
    is as well. 
    The continuity on $[0,1]^2\setminus \calN$ follows by observing that
    \begin{align*}
        \Pi_{\gamma^n(s)}(\gamma^n(s) - \gamma^n(t)) &= \gamma^n(s) - \gamma^n(t) - \skp{\gamma^n(s) - \gamma^n(t), {\gamma^n}'(s)}{\gamma^n}'(s),
    \end{align*}
     together with formula \eqref{eq:repradius}  and the convergence results established above.
\end{proof}

With all of the considerations above, we are able to show the existence of minimizers. 

\begin{theorem}\label{thm:compn}
    Let $p>2$ and suppose that the infimum defined in \eqref{eq:BBgamma} is finite. Then there exists a minimizer $\gamma \in C^{1, 1 - \frac{2}{p}}([0,1];\Omega)$ and $(J_t, V_t, \rho_t, \mu_t, f_t, G^0_t, G^1_t) \in {\rm CE}(\Gamma)$ of the problem \eqref{eq:BBgamma}.
\end{theorem}

\begin{proof}   
    Let $\gamma^n \in C^{1}([0,1];\Omega)$ and $(J_t^n, V_t^n, \rho_t^n, \mu^n, f_t^n, (G_t^0)^n, (G_t^1)^n)_{n\in\mathbb{N}} \in {\rm CE}(\Gamma^n)$ be a minimizing sequence of \eqref{eq:BBgamma} and denote by $L(n)$ the length of $\gamma^n$. Note that by Proposition \ref{EnergySpaces} the curve $\gamma^n$ is regular and non self-intersecting, thus allowing to define ${\rm CE}(\Gamma^n)$.
    Since $R$ is invariant under reparametrization, we consider the arclength parametrizations $\gamma^n: [0,L(n)]\to\Omega$ with uniformly bounded Tangent-Point energy, that is there exists a constant $M>0$ such that $\sup_n E_{\rm tp}^p(\gamma^n) \leq M$.
    Applying Proposition \ref{EnergySpaces} we obtain that \begin{align}\label{eq:estimate2}
        [(\gamma^n)']^p_{W^{1 - \frac{1}{p}, p}([0,L(n)];\Omega)} \leq C(p,M).
    \end{align} 
    Defining $\tilde \gamma^n : [0,1]\to\R^d$ as $\Tilde{\gamma}^n(t) := \gamma^n(L(n)\cdot t)$ for all $n\in\mathbb{N}$ gives
    \begin{align}\label{eq:secone}
        [ (\Tilde{\gamma}^n)']^p_{W^{1-\frac{1}{p},p}([0,1];\Omega)} &\leq L(n)^{2p-2}C(p,M).
    \end{align}
    Since $p>2$ by assumption, the previous estimate implies uniform boundedness of the Gagliardo seminorms. 
    Moreover, by compactness of $\Omega$, the $L^p([0,1];\Omega)$-norms are bounded and by
    \begin{align}
    \int_0^1 |(\tilde \gamma^n)'(t)|^p\, dt =  L(n)^p \int_0^1 | (\gamma^n)'(L(n)t)|^p\, dt = L(n)^p  
    \end{align}
    the sequence of $L(n)$ is as well, implying a uniform bound on the $W^{2-\frac{1}{p},p}([0,1];\Omega)$-norms. 
    Let $0<\alpha'<1-\frac{2}{p}$. From the compact embedding $W^{2-\frac{1}{p},p}([0,1];\Omega)$ into $C^{1,\alpha'}([0,1];\Omega)$, see Theorem \ref{CmpctEmbedding}, we deduce that $(\Tilde{\gamma}^n)_{n\in\mathbb{N}}$ admits a $C^{1,\alpha'}([0,1];\Omega)$-convergent subsequence, still denoted by $(\Tilde{\gamma}^n)_{n\in\mathbb{N}}$, with limit $\gamma \in C^{1,\alpha'}([0,1];\Omega)$.

    From Proposition \ref{thm:compactVarPath} we conclude the existence of another subsequence of measures, again denoted by $(J_t^n, V_t^n, \rho_t^n, \mu_t^n, f_t^n, ({G_t^0})^n, ({G_t^1})^n)_{n\in\mathbb{N}} \in {\rm CE}(\Gamma^n)$, converging to $(J_t, V_t, \rho_t, \mu_t, f_t, ({G_t^0}), ({G_t^1})) \in \mathcal{D}_{\rm adm}(\Gamma)$ according to the convergences stated in Proposition \ref{thm:compactVarPath}. 
    From the lower semicontinuity established in Proposition \ref{LowSemiCont} it follows that $R(\gamma)$ is finite for the limit curve and therefore, thanks to Proposition \ref{EnergySpaces}, $\gamma$ is non self-intersecting and regular with $\gamma\in C^{1,1-\frac{2}{p}}([0,1];\Omega)$ by the continuous embedding from Theorem \ref{CmpctEmbedding}.
    
    It remains to prove that $(J_t, V_t, \rho_t, \mu_t, f_t, ({G_t^0}), ({G_t^1}))$ satisfies the coupled continuity equations. To this end note that, since $\gamma$ is non self-intersecting and regular, the pushforward measures $\mathcal{V}^\gamma$, $f^\gamma$ and $({\mathcal{G}^i})^\gamma$ are well-defined and, by $\gamma^n \rightarrow \gamma$ uniformly, it holds that 
    \begin{align}\label{eq:ccdd}
        (\mathcal{V}^n)^{\gamma^n} \rightharpoonup (\mathcal{V})^{\gamma}, \ \ (f^n)^{\gamma^n} \rightharpoonup (f)^{\gamma}, \ \ ((\mathcal{G}^i)^n)^{\gamma^n} \rightharpoonup ({\mathcal{G}^i})^{\gamma}
    \end{align}
    narrowly. Moreover, inspecting the proof of Proposition \ref{thm:compactVarPath} we also have that $(\mu^n)^\gamma \rightharpoonup \mu^\gamma$. Finally note that for every test function $\psi$ it holds that 
    \begin{align*}
        & \left|\int_0^1\int_0^1 \partial_s \psi(t,s) |(\gamma^n)'(s)|^{-1} \, d\mathcal{V}_t^{\gamma^n} dt - \int_0^1\int_0^1 \partial_s \psi(t,s) |\gamma'(s)|^{-1} \, d\mathcal{V}_t^{\gamma} dt \right| \\
        & \leq \left|\int_0^1\int_0^1 \partial_s \psi(t,s) |\gamma'(s)|^{-1} \, d\mathcal{V}_t^{\gamma^n} dt - \int_0^1\int_0^1 \partial_s \psi(t,s) |\gamma'(s)|^{-1} \, d\mathcal{V}_t^{\gamma} dt \right| \\
        & \qquad  + \left|\int_0^1\int_0^1 \partial_s \psi(t,s) |(\gamma^n)'(s)|^{-1} \, d\mathcal{V}_t^{\gamma^n} dt - \int_0^1\int_0^1 \partial_s \psi(t,s) |\gamma'(s)|^{-1} \, d\mathcal{V}_t^{\gamma^n} dt \right|\\
        & \leq \left|\int_0^1\int_0^1 \partial_s \psi(t,s) |\gamma'(s)|^{-1} \, d\mathcal{V}_t^{\gamma^n} dt - \int_0^1\int_0^1 \partial_s \psi(t,s) |\gamma'(s)|^{-1} \, d\mathcal{V}_t^{\gamma} dt \right| \\ 
        & \qquad + \|\psi\|_{C^1([0,1]^2)} \|\mathcal{V}^{\gamma^n}_t\|_{M([0,1]^2)} \sup_s ||(\gamma^n)'(s)|^{-1} - |\gamma'(s)|^{-1}| \rightarrow 0
    \end{align*}
    due to the narrow convergence $(\mathcal{V}^n)^{\gamma^n} \rightharpoonup (\mathcal{V})^{\gamma}$ and the $C^1$ convergence of the curves. The last inequality follows since  $||(\gamma^n)'(s)|^{-1} - |\gamma'(s)|^{-1}| \leq C|(\gamma^n)'(s) - \gamma'(s)|$ for a positive constant $C$ independent on $n$.

    This, together with the convergences in \eqref{eq:ccdd} and the ones obtained in Proposition \ref{thm:compactVarPath}, allows to conclude that $(J_t, V_t, \rho_t, \mu_t, f_t, ({G_t^0}), ({G_t^1})) \in {\rm CE}(\Gamma)$ as we wanted to prove. 
\end{proof}

\subsection{The case of linear mobilities and beyond}\label{LinMobilities}

Despite the focus on general concave mobility functions with bounded domain, it is possible to extend the results of this paper to linear mobilities, giving action functionals 
\begin{align}
    \Psi_{\rm linear}(z,w) = \left\{
    \begin{array}{cc}
         \frac{|w|^2}{z}& z \neq 0  \\
         0 & w = 0 \text{ and } z = 0\\
         +\infty &  \text{ otherwise } 
    \end{array}
    \right.
\end{align}
and thus recovering traditional Optimal Transport dynamics both on $\Omega$ and $\Gamma$. Since the adaptation is straightforward and many arguments follow from the estimates provided in \cite{monsaingeon2021new}, we decide to discuss it only briefly in this section.
First, we remark that an analogue of Proposition \ref{prop:holder} follows by similar estimates and adapting \cite[Proposition 3.5]{monsaingeon2021new}. Such estimates lead to suitable compactness (adapting Theorem \ref{thm:compact}) and thus the existence of minimizers for a fixed preferential path.
Similar considerations hold for the existence of minimizers in the case of varying paths. 
Interestingly, the case of linear mobilities allows to prove finiteness of the infimum using Wasserstein and Fischer-Rao geodesics as in \cite[Lemma 3.7]{monsaingeon2021new}. Instead, as it is customary for general mobilities \cite{Lisini2010}, this is an additional assumption in our results, both for fixed and varying preferential paths.

We are also confident that similar approaches following \cite{Dolbeault_2008, Carrillo2010} would allow to extend the results of this paper to the case of non-decreasing mobilities with unbounded domain recovering, for instance, action functionals for the following form:
\begin{align}
    \Psi_{\rm concave}(z,w) = \left\{
    \begin{array}{cc}
         \frac{|w|^2}{z^\alpha}& z \neq 0  \\
         0 & w = 0 \text{ and } z = 0\\
         +\infty &  \text{ otherwise } 
    \end{array}
    \right.
\end{align}
for $\alpha \in (0,1)$. Here one can rely on the compactness estimates and lower semicontinuity results of \cite{Dolbeault_2008} with the only verification needed being at the level of the compactness for varying preferential paths.

\section{Numerics}\label{Numerics}
In this section we present a primal-dual optimization algorithm solving the coupled Optimal Transport problem introduced and analyzed in the previous sections, restricted to linear mobilities as discussed in Subsection \ref{LinMobilities} assuming $\calG_t^i=0$. We will first apply the method to the case of a fixed curve. Our method is based on the classical ALG2 algorithm introduced in \cite{augmented_1983}. However, in order to apply this method, we first need to reformulate the minimization problem into a saddle-point problem. In order to include the case of varying paths as well, we then embed the algorithm into a projected gradient descend step with respect to the discretized curve. We conclude this section by applying both algorithms to several examples, highlighting the influence of the cost parameters and the geometry of the problem to the shape of geodesics. Most of the details can be found in Appendix \ref{AppNumerics} as they follow from classical derivations, except the discretization which we discuss in more detail.

As a first step, we consider the case of a fixed preferential path. 
For the numerical investigation, we introduce an Augmented Lagragangian formulation of the minimization problem. The derivation can be found in Appendix \ref{derivAugmLagr} and follows standard arguments as presented for example in \cite{fu_high_2023}. The resulting reformulation of our problem reads 
\begin{align}\label{eq:OptAugmL}
    \sup\limits_{(\rho_t, J_t, \mu_t, \mathcal{V}_t, f_t)} \inf\limits_{\genfrac{}{}{0pt}{}{(\phi_t, \phi_t^{1d}),}{(\rho_t^*, J_t^*, \mu_t^*, \mathcal{V}_t^*, f_t^*)}} &L^{\alpha_1, \alpha_2}(\rho_t, J_t, \mu_t, \mathcal{V}_t, f_t, \phi_t, \phi_t^{1d}, \rho_t^*, J_t^*, \mu_t^*, \mathcal{V}_t^*, f_t^*)
\end{align}
for the Augmented Lagrangian functional 
\begin{align*}
    &L^{\alpha_1, \alpha_2}(\rho_t, J_t, \mu_t, \mathcal{V}_t, f_t, \phi_t, \phi_t^{1d}, \rho_t^*, J_t^*, \mu_t^*, \mathcal{V}_t^*, f_t^*)\\
    &:= \int_0^1\left(\calA^{\alpha_1, \alpha_2}\right)^*(\rho_t^*, J_t^*, \mu_t^*, \mathcal{V}_t^*, f_t^*)  \, dt + \calG(\phi_t, \phi_t^{1d}) \\
    &+ \skp{(\partial_t\phi_t - \rho_t^*,\nabla\phi_t - J_t^*),(\rho_t, J_t)}_\Omega \\
    &+ \skp{(\partial_t\phi_t^{1d} - \mu_t^*, \nabla\phi_t^{1d} - \mathcal{V}_t^*, \phi_t^{1d}-\phi_t\circ\gamma - f_t^*),(\mu_t, \mathcal{V}_t, f_t)}_\Gamma \\
    &+ \frac{r_1}{2} \skp{(\partial_t\phi_t - \rho_t^*,\nabla\phi_t - J_t^*),(\partial_t\phi_t - \rho_t^*,\nabla\phi_t - J_t^*)}_\Omega \\
    &+ \frac{r_2}{2} \skp{(\partial_t\phi_t^{1d} - \mu_t^*, \nabla\phi_t^{1d} - \mathcal{V}_t^*, \phi_t^{1d}-\phi_t\circ\gamma - f_t^*),(\partial_t\phi_t^{1d} - \mu_t^*, \nabla\phi_t^{1d} - \mathcal{V}_t^*, \phi_t^{1d}-\phi_t\circ\gamma - f_t^*)}_\Gamma
\end{align*}
with parameters $r_1, r_2>0$, allowing the application of ALG2. However, we need to adapt this method to our setting.
To this extend, we will follow the structure of the method introduced in \cite{benamou2000computational}, modified to account for the coupling along the preferential path. 

The main idea of the algorithm is to decouple the optimization of $(\phi_t, \phi_t^{1d}), (\rho_t^*, J_t^*, \mu_t^*, \mathcal{V}_t^*, f_t^*)$ and $(\rho_t, J_t, \mu_t, V_t, f_t)$. 
Each iteration follows three steps:
\begin{enumerate}
    \item Update $(\phi_t, \phi_t^{1d})$.
    \item Update $(\rho_t^*, J_t^*, \mu_t^*, \mathcal{V}_t^*, f_t^*)$.
    \item Update $(\rho_t, J_t, \mu_t, \mathcal{V}_t, f_t)$.
\end{enumerate}
These updates follow a standard structure, which details can be found in Appendix \ref{DetailsAlg}. The resulting algorithm is summarized in Algorithm \ref{alg:AugmLagr} where we omit the explicit time-dependence of the variables and specify the current iteration instead.

One iteration of these three steps will be called one step of the Augmented Lagrangian method, ALG for short. 
It can be extended to include minimization over the preferential path by embedding the ALG method into a projected gradient descent loop for the objective functional from \eqref{eq:BBgamma}, using central differences to approximate the gradient. 
Let $\Gamma_k$ be the image of the preferential path in iteration $k\in\mathbb{N}_0$ and let $p_k$ be a normalized descent direction. Then, $\Gamma_k$ is updated as
\begin{align*}
    \Tilde{\Gamma}_{k+1} &= \Gamma_k + \varepsilon_k p_k, \\
    \Gamma_{k+1} &= \Pi_\Omega \Tilde{\Gamma}_{k+1},
\end{align*}
where $\Pi_\Omega: \R^d \to \Omega$ is the orthogonal projection onto $\Omega$ and the stepsize $\varepsilon_k>0$ is allowed to vary.

\begin{algorithm}[t]
    \caption{Augmented Lagrangian Method}\label{alg:AugmLagr}
    \begin{algorithmic}
        \WHILE{$m\leq it_{\max}^{ALG}$ and $err_{ALG}>tol$}
            \STATE Solve $A_h (\phi_{m+1}, \phi^{1d}_{m+1}) = F^m_h$ for $\phi_{m+1}$ and $\phi^{1d}_{m+1}$
            \STATE $\eta_{\rho^*} \leftarrow \partial_t\phi_{m+1} + \frac{1}{r_1}\rho_m$ and $\eta_{J^*} \leftarrow \nabla\phi_{m+1} + \frac{1}{r_1}J_m$\\ $\eta_{\mu^*} \leftarrow \partial_t\phi^{1d}_{m+1} + \frac{1}{r_2}\mu_m$ and $\eta_{\mathcal{V}^*} \leftarrow \nabla\phi^{1d}_{m+1} + \frac{1}{r_2}\mathcal{V}_m$ and $\eta_{f^*} \leftarrow \phi^{1d}_{m+1} - \phi_{m+1}\circ\gamma + \frac{1}{r_2}f_m$
            \IF{$\eta_{\rho^*} + \frac{\eta_{J^*}^2}{2}\leq0$}
                \STATE $\rho^*_{m+1}\leftarrow \eta_{\rho^*}$ and $J^*_{m+1}\leftarrow \eta_{J^*}$
            \ELSE
                \STATE Solve $0 = x_\Omega^3 + 2 (1+\eta_{\rho^*}) x_\Omega - 2 | \eta_{J^*}|$
                \STATE $\rho^*_{m+1}\leftarrow -\frac{1}{2}x_\Omega^2$ and $J^*_{m+1}\leftarrow x_\Omega\frac{\eta_{J^*}}{| \eta_{J^*} |}$
            \ENDIF
            \IF{$\eta_{\mu^*} + \frac{1}{2}\left( \frac{\eta_{\mathcal{V}^*}^2}{\alpha_1} + \frac{\eta_{f^*}^2}{\alpha_2} \right)\leq0$}
                \STATE $\mu^*_{m+1}\leftarrow \eta_{\mu^*}$ and $V^*_{m+1}\leftarrow \eta_{\mathcal{V}^*}$ and $f^*_{m+1}\leftarrow \eta_{f^*}$
            \ELSE
                \STATE Solve $\begin{cases}
                0 &= x_\Gamma^3 + \frac{\alpha_1}{\alpha_2} x_\Gamma y_\Gamma^2 + 2 \left( \alpha_1^2 + \alpha_1\eta_{\mu^*}\right) x_\Gamma - 2 \alpha_1^2 |\eta_{V^*}|\\
                0 &= y_\Gamma^3 + \frac{\alpha_2}{\alpha_1} x_\Gamma^2 y_\Gamma + 2 \left( \alpha_2^2 + \alpha_2\eta_{\mu^*}\right) y_\Gamma - 2 \alpha_2^2 \eta_{f^*}
            \end{cases}$
                \STATE $\mu^*_{m+1}\leftarrow -\frac{1}{2}\left(\frac{x_\Gamma^2}{\alpha_1} + \frac{y_\gamma^2}{\alpha_2} \right)$ and $\mathcal{V}^*_{m+1}\leftarrow  x_\Gamma \frac{\eta_{\mathcal{V}^*}}{|\eta_{\mathcal{V}^*}|}$ and $f^*_{m+1}\leftarrow y_\Gamma$
            \ENDIF
            \STATE $\rho_{m+1} \leftarrow \rho_{m} + r_1 \left( \partial_t\phi_{m+1} - \rho^*_{m+1} \right)$
            \STATE $J_{m+1} \leftarrow J_{m} + r_1 \left( \nabla\phi_{m+1} - J^*_{m+1} \right)$
            \STATE $\mu_{m+1} \leftarrow \mu_{m} + r_2 \left( \partial_t\phi^{1d}_{m+1} - \mu^*_{m+1}\right)$
            \STATE $\mathcal{V}_{m+1} \leftarrow \mathcal{V}_{m} + r_2 \left( \nabla\phi^{1d}_{m+1} - \mathcal{V}^*_{m+1} \right)$
            \STATE $f_{m+1} \leftarrow f_{m} + r_2 \left( \phi^{1d}_{m+1} - \phi_{m+1}\circ\gamma - f^*_{m+1} \right)$
            \STATE $err_\Omega \leftarrow \max\left(\|\partial_t\phi_{m+1} - \rho^*_{m+1}\|_\infty, \|\nabla\phi_{m+1} - J^*_{m+1}\|_\infty\right)$ 
            \STATE $ err_\Gamma \leftarrow \max\left(\|\partial_t\phi^{1d}_{m+1} - \mu^*_{m+1}\|_\infty, \|\nabla\phi^{1d}_{m+1} - \mathcal{V}^*_{m+1} \|_\infty, \| phi^{1d}_{m+1} - \phi_{m+1}\circ\gamma - f^*_{m+1} \|_\infty \right)$
            \STATE $err_{ALG} \leftarrow err_\Omega + err_\Gamma$
            \STATE $m \leftarrow m+1$
        \ENDWHILE
    \end{algorithmic}
\end{algorithm}

\subsection{Discretization}\label{Discret}

Throughout all examples, we fix $\Omega=[0,1]^2$ and initial and final data to be compactly supported inside $\Omega$. Note that the first assumption is not restrictive. For initial or final data supported compactly but outside of $\Omega$, extending the domain $\Omega$ and rescaling the system recovers the assumption.

The image of the curve $\Gamma\subset\Omega$ is approximated using a piecewise linear curve. Let $\gamma_i\in\Gamma\subset\Omega$ for $i=0,\ldots,n_\gamma$ be a family of points inscribed in $\Gamma$ and let $L_i := [\gamma_{i-1},\gamma_i]$ for $i=1,\ldots,n_\gamma$ be the straight line between $\gamma_{i-1}$ and $\gamma_i$, see \autoref{fig:polgyon} for an illustration. We describe such an approximation by a vector $\gamma_h\in \R^{2n\gamma+2}$ where each entry corresponds to the coordinate of one of the points $\gamma_i$ and the corresponding image by $\Gamma_h$. 

As discretization of the optimal transport problem, we employ the finite element method introduced in \cite{fu_high_2023}, adapted to coupled equations.
Let $\calT = \{T_i\}_{i=1}^N$ denote a triangulation of $\Omega$ such that the polygonal curve is formed by a set of edges and extend this triangulation by prism elements of the form $I\times T$ for $T_i\in\calT$ and $I\in \calI$ to the unit cube $[0,1]^3\subset\R^3$. Here, $\calI := \{[t_{i-1}, t_i] \  \ t_i = \frac{i}{n_t}, \, i=0,\ldots,n_t\}$ is the partition of the unit interval $[0,1]\subset\R$ into $n_t$ intervals of the same length $\frac{1}{n_t}$.

For $k\geq 0$ and $I\in\calI$ let $\calP^k(I)$ be the set of polynomials of order $k$ on $I$ and define $\calP^k(T)$ for $T\in\calT$ in the same way.

We will need the following finite element spaces:
\begin{align*}
    V_h &:= \{ v \in H^1(X_\Omega) \ | \ v_{|I\times T} \in \calP^0(I)\otimes\calP^1(T) \quad\forall I\in\calI, T\in\calT\}\\
    W_h &:= \{ w \in L^2(X_\Omega) \ | \ v_{|I\times T} \in \calP^0(I)\otimes\calP^0(T) \quad\forall I\in\calI, T\in\calT\}\\
    V_h^{1d} &:= \{ v \in H^1(X_\Gamma) \ | \ v_{|I\times T} \in \calP^0(I)\otimes\calP^1(T) \quad\forall I\in\calI, T\in\calT\}\\
    W_h^{1d} &:= \{ w \in L^2(X_\Gamma) \ | \ v_{|I\times T} \in \calP^0(I)\otimes\calP^0(T) \quad\forall I\in\calI, T\in\calT\}    
\end{align*}
with $V_h$ and $V_h^{1d}$ $H^1$-conforming on $\calI\otimes\calT$.

We are looking for $\rho_h, \rho^*_h\in W_h$, $J_h, J_h^*\in W_h\times W_h$, $\phi_h\in V_h$, $\mu_h, \mu_h^*, f_h, f_h^*\in W_h^{1d}$, $\mathcal{V}_h, \mathcal{V}^*_h \in W_h^{1d}\times W_h^{1d}$ and $\phi^{1d}_h\in V_h^{1d}$ solving \eqref{eq:OptAugmL} where the choice of spaces allows for pointwise updates. Denote by $A_h, F^m_h$ the discrete matrix and vector representation of the bilinear- and linearform from the update of $\phi_{m+1}$ and $\phi^{1d}_{m+1}$, see Appendix \ref{DetailsAlg}. 
Initial and final data are defined as in the beginning of Section \ref{OptPrefPath} and approximated by gridfunctions from the spaces $V_h$ and $V_h^{1d}$.

\begin{figure}
    \centering
    \begin{tikzpicture}[scale=.5]
        \path 
            (0,0) coordinate (A) 
            (2,1) coordinate (B)
            (3,3) coordinate (C)
            (1,4) coordinate (D)
            (-2,3) coordinate (E)
            (-1,1) coordinate (F);

        \node at (A) [below] {$\gamma_0$};
        \node at (B) [right] {$\gamma_1$};
        \node at (C) [right] {$\gamma_2$};
        \node at (D) [above] {$\gamma_3$};
        \node at (E) [left] {$\gamma_4$};
        \node at (F) [below] {$\gamma_5$};
        
        \draw plot [smooth, tension=1.2] coordinates {(A) (B) (C) (D) (E) (F)};
        
        \draw[blue] (A) -- (B) -- (C) -- (D) -- (E) -- (F);
        
        \foreach \p in {A, B, C, D, E, F}
            \fill[black] (\p) circle (2pt);
    \end{tikzpicture}
    \caption{Polygonal curve inscribed in $\Gamma$}
    \label{fig:polgyon}
\end{figure}
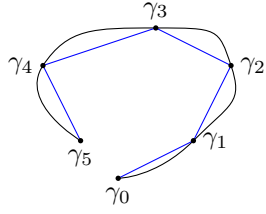

Now, consider the discretization of the penalty functional for varying curves. 
In order to apply cross-product formulas for the tangent point radius as in \cite{bartels_simple_2018} or \cite{yu_repulsive_2021}, we extend the parametrization of the curve into three dimensions by introducing a zero-component.
The integral functionals are approximated using the two-dimensional trapezoidal rule. The remaining penalty terms can be calculated explicitly for piecewise linear curves. We will use the discretization of the Tangent-Point energy from \cite{yu_repulsive_2021}.
The length of each interval is given by $l_i := | \gamma_i - \gamma_{i-1}|$ and the discrete tangent by $\tau_i := \frac{\gamma_i - \gamma_{i-1}}{l_i}$.
With this notation, the penalty functional reads
\begin{align*}
    R_h(\gamma_h) &= \sum\limits_{i=1}^{n_\gamma} \sum\limits_{L_j\cap L_i = \emptyset}  k_p(i,j, \tau_i)l_i l_j  - \log|\gamma_0 - \gamma_{n_\gamma} | + \sum\limits_{i=1}^{n_\gamma} l_i
\end{align*}
where 
\begin{align*}
    k_p(i,j,\tau_i) &:= \frac{1}{4} \left( \frac{| \tau_i \times (\gamma_{i-1} - \gamma_{j-1})|^p}{| \gamma_{i-1} - \gamma_{j-1} |^{2p}} + \frac{| \tau_i \times (\gamma_{i-1} - \gamma_{j})|^p}{| \gamma_{i-1} - \gamma_{j} |^{2p}} + \frac{| \tau_i \times (\gamma_{i} - \gamma_{j-1})|^p}{| \gamma_{i} - \gamma_{j-1} |^{2p}} + \frac{| \tau_i \times (\gamma_{i} - \gamma_{j})|^p}{| \gamma_{i} - \gamma_{j} |^{2p}} \right) 
\end{align*}
is the average of the integrand evaluated in the boundary points of $L_i\times L_j$. 
Using this approximation, we get rid of the singularity at intervals $L_i\cap L_j \neq \emptyset$.

Similarly, we approximate the action from \eqref{eq:BBgamma} using the two dimensional trapezoidal rule, resulting in the discrete functional $\calA_{\Gamma_h}(J_h, \mathcal{V}_h, \rho_h, \mu_h, f_h)$ where time-integration is discretized as well.

It remains to approximate the gradient of the objective functional $\calA_{\Gamma_h}(J_h, \mathcal{V}_h, \rho_h, \mu_h, f_h) + c R(\gamma_h)$ and calculate a descent direction close to this gradient. We define the central difference operators
\begin{align*}
    D_i \calA_{\Gamma_h}(J_h, \mathcal{V}_h, \rho_h, \mu_h, f_h) &:= \frac{\calA_{\Gamma_h + \varepsilon e_i}(J_h^+, \mathcal{V}_h^+, \rho_h^+, \mu_h^+, f_h^+) - \calA_{\Gamma_h- \varepsilon e_i}(J_h^-, \mathcal{V}_h^-, \rho_h^-, \mu_h^-, f_h^-)}{2\varepsilon}
    \intertext{and}
    D_i R(\gamma_h) &:= \frac{R(\gamma_h + \varepsilon e_i) - R(\gamma_h- \varepsilon e_i)}{2\varepsilon}
\end{align*}
for $e_i\in\R^{2n_\gamma+2}$ the $i$-th unit vector and a small pertubation $\varepsilon>0$. 
In order to calculate the resulting function values, the change in measures needs to be accounted for by calculating the values using one iteration of the augmented Lagrangian method for fixed curves introduced above. 
As a search direction choose
\begin{align*}
    p_k^i &:= \begin{cases}
        \frac{D_i \calA_{\Gamma_h}(J_h, \mathcal{V}_h, \rho_h, \mu_h, f_h) + c D_i R_h(\gamma_h)}{| D_i \calA_{\Gamma_h}(J_h, \mathcal{V}_h, \rho_h, \mu_h, f_h) + c D_i R_h(\gamma_h) |} & \text{if }   |D_i \calA_{\Gamma_h}(J_h, \mathcal{V}_h, \rho_h, \mu_h, f_h)| > c|D_i R_h(\gamma_h)| \\
        0 & \text{else}
    \end{cases}
\end{align*}
for $i\in\{1,\ldots, 2n_\gamma+2\}$, allowing to update if and only if the change in action dominates the change in the penalty term. 
Note that this choice allows for $p_k=0$ and if $p_k=0$ for $n_{iter}$ consecutive iterations, the scaling parameter $c>0$ is updated to $\frac{c}{2}$ if it remains above the threshold $c_{low}>0$. 

After each curve update, the domain $\Omega$ is remeshed. To allow this process to be consistent, we impose the additional assumption that there exists some fixed $\delta>0$ such that $dist(\Gamma, \partial\Omega)\geq\delta$. This distance will specify the bounds used in the projection step onto the closed interval $[\delta, 1-\delta]$ leading to
\begin{align*}
    \Pi_{[\delta,1-\delta]} x &= \min(1-\delta, \max(x,\delta)).
\end{align*}
Applying this projection to each component of the vector $\gamma_h$ defines the projection of the piecewise linear curve onto $[\delta,1-\delta]^2\subset\Omega$ and we denote this component wise projection by $\Pi_{[\delta,1-\delta]}$. 
The update steps described above are summarized in Algorithm \ref{alg:MinCurve}.

\begin{algorithm}[t]
    \caption{Minimization over the preferential path}\label{alg:MinCurve}
    \begin{algorithmic}
        \STATE Initialize mesh and initial, final data
        \WHILE{$k \leq it^\gamma_{\max}$ and $err> tol$}
        \IF{$p_j=0$ for all $j\in\{k-n_{iter},\ldots,k\}$ and $c>c_{low}$}
            \STATE $c \leftarrow \frac{c}{2}$
        \ENDIF
            \FOR{$i=1,\ldots, 2n_\gamma+2$}
                \STATE \autoref{alg:AugmLagr} with $\gamma=\gamma_{k} + \varepsilon e_i$
                \STATE \autoref{alg:AugmLagr} with $\gamma=\gamma_{k} - \varepsilon e_i$
                \STATE $g_{k}^i \leftarrow D_i \calA_{\Gamma_{k}} + c D_i R(\gamma_{k})$
                \IF{$|D_i \calA_{\Gamma_{k}} | > c|D_i R(\gamma_{k})|$ }
                    \STATE $p_{k}^i \leftarrow g_{k}^i$
                \ELSE
                    \STATE $p_{k}^i \leftarrow 0$
                \ENDIF
            \ENDFOR
            \STATE $p_{k} \leftarrow \frac{p_{k}}{| p_{k}|}$
            \STATE $\gamma_{k+1} \leftarrow \Pi_{[\delta,1-\delta]}(\gamma_{k}+\varepsilon_k p_{k})$
            \STATE Remesh with $\gamma=\gamma_{k+1}$
            \STATE \autoref{alg:AugmLagr} with $\gamma=\gamma_{k+1}$
            \STATE $err \leftarrow \| \gamma_{k+1} - \gamma_{k} \|_\infty + err_{AL}$
            \STATE $k \leftarrow k+1$
        \ENDWHILE
    \end{algorithmic}
\end{algorithm}

\subsection{Examples}

This section highlights how the coupling mechanisms and the geometry of the problem influence the transport dynamics. In the case of a fixed curve we will compare different cost parameters and the case of varying curves is studied in two different configurations as well.
We will apply the algorithms introduced in the previous subsections to several examples on $\Omega=[0,1]^2$ and initial and final data modelled using the two dimensional gaussian
\begin{equation}\label{eq:Gaussian}
    f_{m_x, m_y, \sigma}(x,y) := \exp\left(-\frac{(x-m_x)^2+(y-m_y)^2}{2\sigma^2}\right)
\end{equation}
with mean $m_x,m_y\in\R$ and variance $\sigma^2>0$.

\subsubsection{Fixed curve}

We test Algorithm \ref{alg:AugmLagr}, thus assuming the preferential path to be given by a fixed curve. All meshes are constructed using the meshsize $h=0.02$ in temporal and spatial direction.

\begin{example}\label{ucurve}
    Consider 
    \begin{align*}
        \rho_0 (x,y) := f_{0.5, 0.2, 0.1}(x,y) \quad \text{and} \quad \rho_1 (x,y) = \frac{1}{2}\left(f_{0.2, 0.8, 0.1}(x,y)  + f_{0.8, 0.8, 0.1}(x,y)\right). 
    \end{align*}
    The (fixed) preferential path is given by the piecewise linear interpolation between the points
    \begin{align*}
        (0.3,0.7), (0.4,0.3), (0.6,0.3), (0.7,0.7) \in [0,1]^2.
    \end{align*}
    Both initial and final data are restricted and scaled as described in the beginning of \autoref{OptPrefPath}, see \autoref{fig:ucurve_data} for an illustration of the resulting measures.

    We will compare two sets of parameters $\alpha_1$ and $\alpha_2$. 
    As a first choice, let $\alpha_1=0.01=\alpha_2$, implying that transport along the curve is cheaper than transport in the bulk region. We compare the results to the choices $\alpha_1=100=\alpha_2$, meaning that transport along the curve is more costly. 
    The resulting evolutions are represented in \autoref{fig:ucurve_smallalph} for $\alpha_1=0.01=\alpha_2$ and \autoref{fig:ucurve_largealph} for $\alpha_1=100=\alpha_2$.
    As expected, small parameter choices lead to increased transport along the curve. For the particular value $0.01$, mass is transported almost exclusively along $\Gamma$. Moreover, transport along the curve is faster compared to \autoref{fig:ucurve_largealph}. 
    For the second pair of parameters, most of the mass is transported in the bulk domain. Since initial and final data are supported on the curve, mass needs to be transported along $\Gamma$ as well.  
\end{example}

\begin{figure}
     \centering
     \begin{subfigure}[b]{0.45\textwidth}
         \centering
         \includegraphics[width=\textwidth]{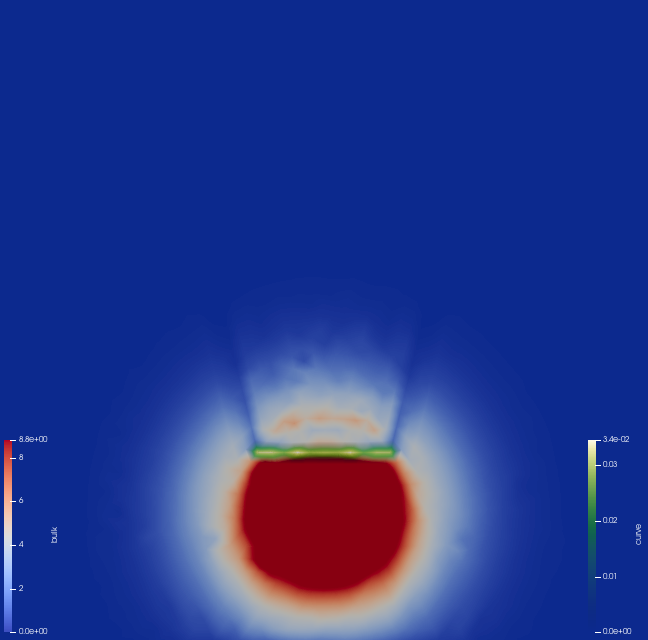}
         \caption{Initial data $\rho_0$}
         \label{fig:ucruve_rho0_fixed}
     \end{subfigure}
     \hfill
     \begin{subfigure}[b]{0.45\textwidth}
         \centering
         \includegraphics[width=\textwidth]{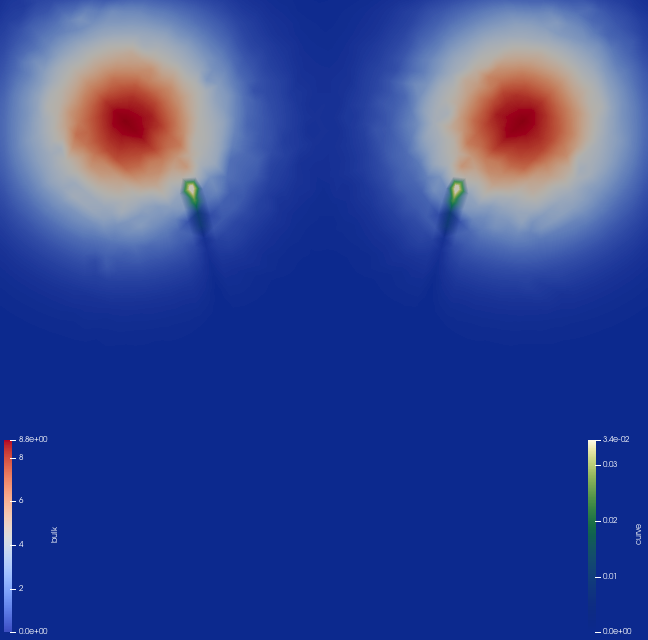}
         \caption{Final data $\rho_1$}
         \label{fig:ucurve_rho1_fixed}
     \end{subfigure}
     \caption{Initial and final data for a u-shaped curve}
     \label{fig:ucurve_data}
\end{figure}

\begin{figure}
     \centering
     \begin{subfigure}[t]{0.3\textwidth}
         \centering
         \includegraphics[height=4cm]{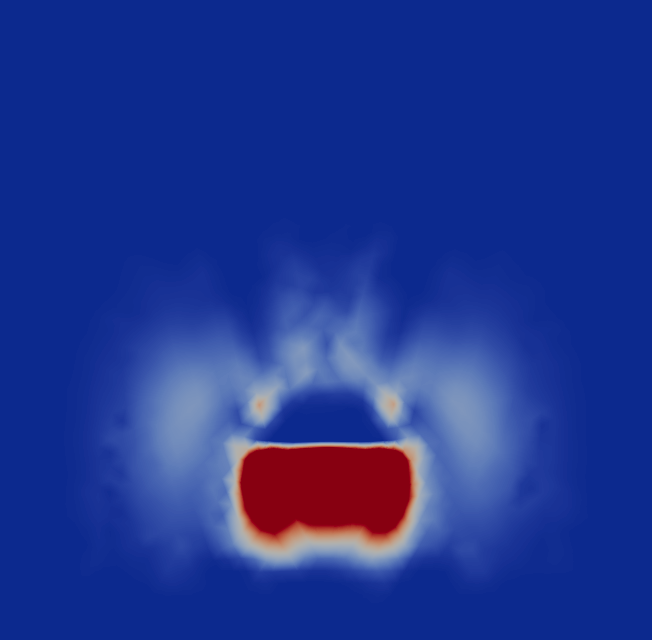}
         \caption{$\rho$ at $t=0.2$}
     \end{subfigure}
     \hfill
     \begin{subfigure}[t]{0.3\textwidth}
         \centering
         \includegraphics[height=4cm]{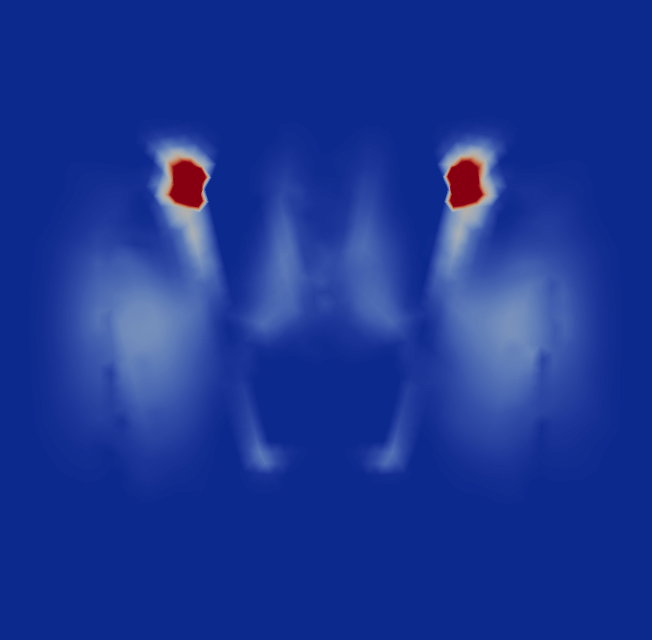}
         \caption{$\rho$ at $t=0.5$}
     \end{subfigure}
     \hfill
     \begin{subfigure}[t]{0.3\textwidth}
         \centering
         \includegraphics[height=4cm]{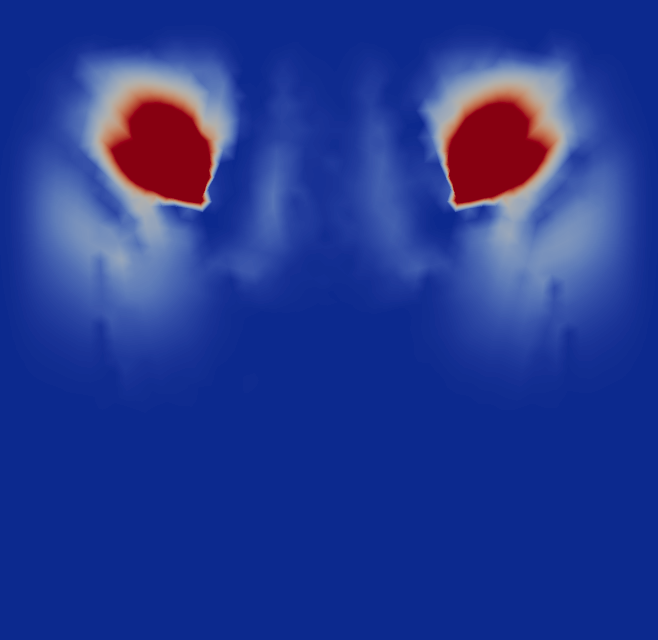}
         \caption{$\rho$ at $t=0.8$}
     \end{subfigure}
     \begin{subfigure}[t]{0.05\textwidth}
         \centering\includegraphics[height=4cm]{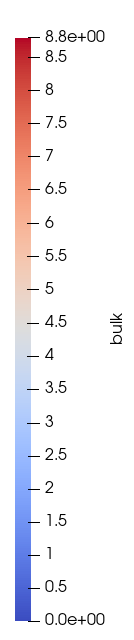}
     \end{subfigure}

     \centering
     \begin{subfigure}[t]{0.3\textwidth}
         \centering
         \includegraphics[height=4cm]{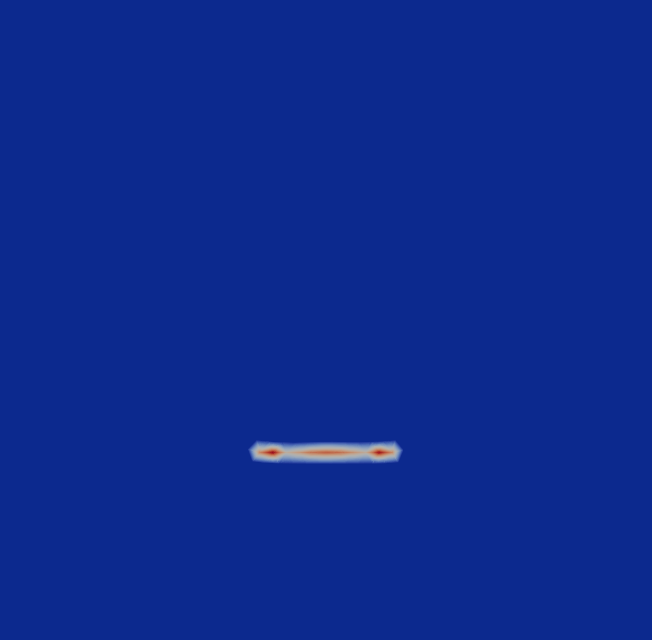}
         \caption{$\mu$ at $t=0.2$}
     \end{subfigure}
     \hfill
     \begin{subfigure}[t]{0.3\textwidth}
         \centering
         \includegraphics[height=4cm]{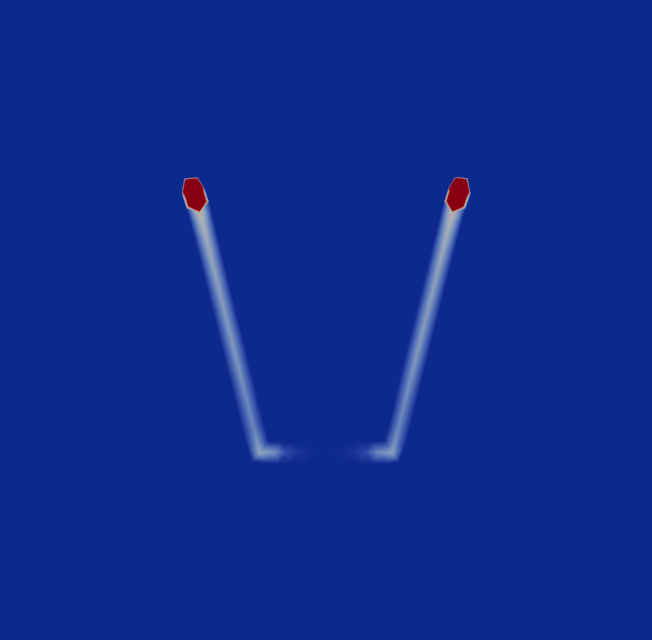}
         \caption{$\mu$ at $t=0.5$}
     \end{subfigure}
     \hfill
     \begin{subfigure}[t]{0.3\textwidth}
         \centering
         \includegraphics[height=4cm]{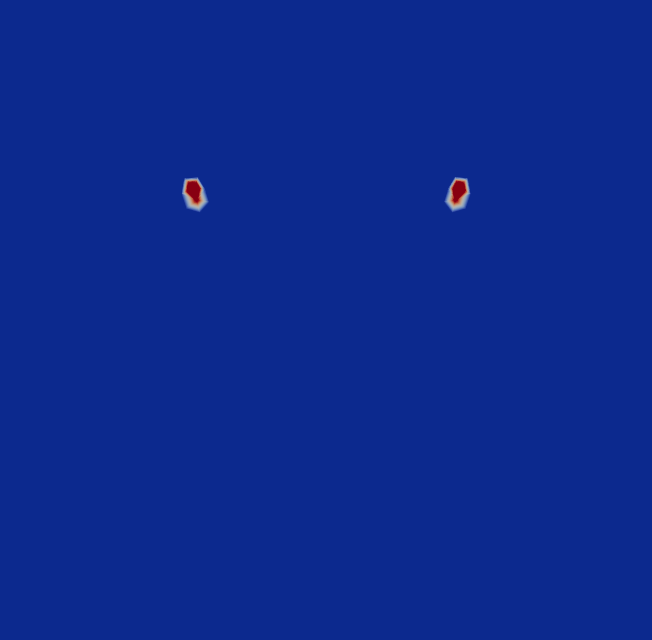}
         \caption{$\mu$ at $t=0.8$}
     \end{subfigure}
     \begin{subfigure}[t]{0.05\textwidth}
         \centering\includegraphics[height=4cm]{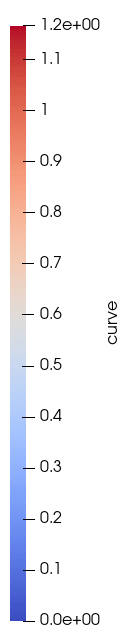}
     \end{subfigure}
     \caption{Time-evolution for $\alpha_1=0.01=\alpha_2$ and a u-shaped curve}
     \label{fig:ucurve_smallalph}
\end{figure}

\begin{figure}
     \centering
     \begin{subfigure}[t]{0.3\textwidth}
         \centering
         \includegraphics[height=4cm]{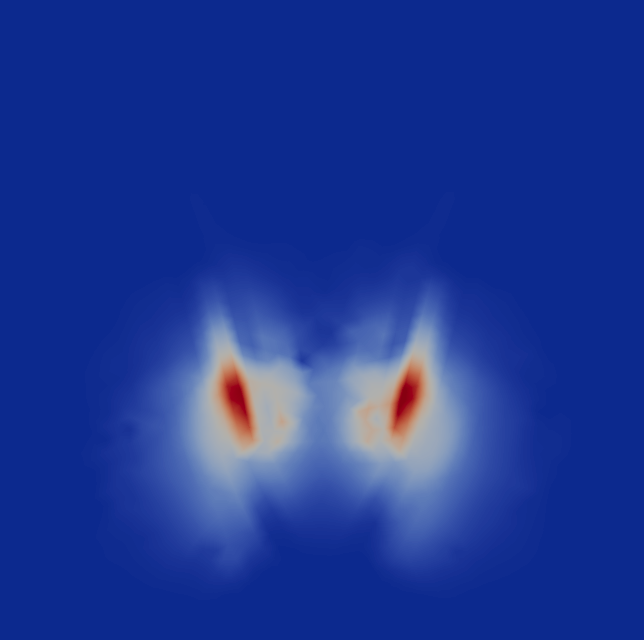}
         \caption{$\rho$ at $t=0.2$}
     \end{subfigure}
     \hfill
     \begin{subfigure}[t]{0.3\textwidth}
         \centering
         \includegraphics[height=4cm]{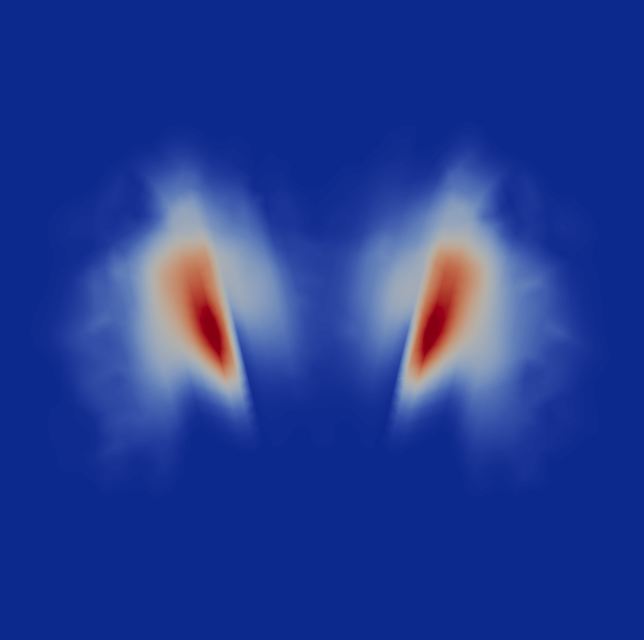}
         \caption{$\rho$ at $t=0.5$}
     \end{subfigure}
     \hfill
     \begin{subfigure}[t]{0.3\textwidth}
         \centering
         \includegraphics[height=4cm]{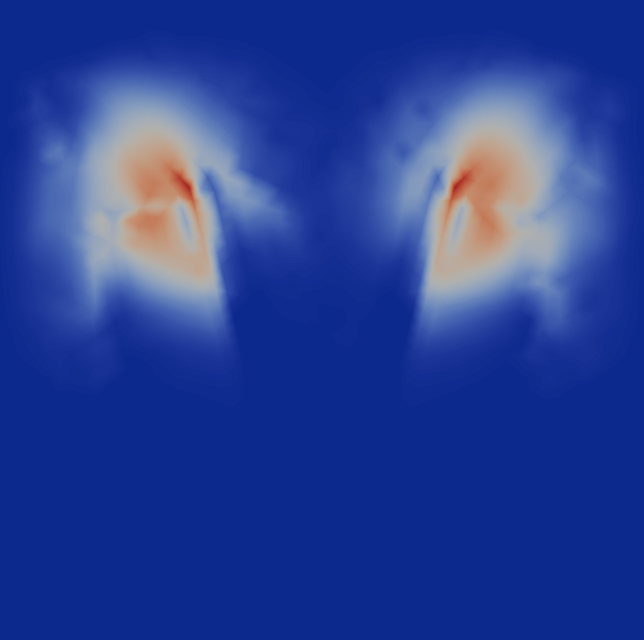}
         \caption{$\rho$ at $t=0.8$}
     \end{subfigure}
     \begin{subfigure}[t]{0.05\textwidth}
         \centering\includegraphics[height=4cm]{plots/ucurve/smallalph_bulk_cb.png}
     \end{subfigure}

     \centering
     \begin{subfigure}[t]{0.3\textwidth}
         \centering
         \includegraphics[height=4cm]{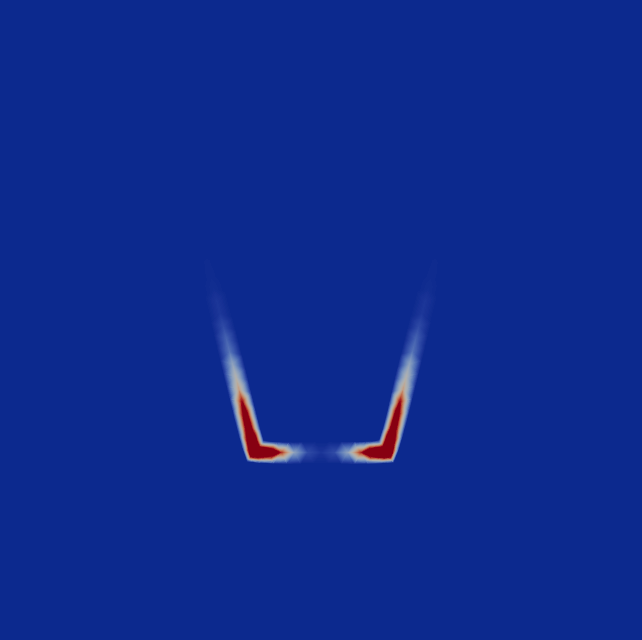}
         \caption{$\mu$ at $t=0.2$}
     \end{subfigure}
     \hfill
     \begin{subfigure}[t]{0.3\textwidth}
         \centering
         \includegraphics[height=4cm]{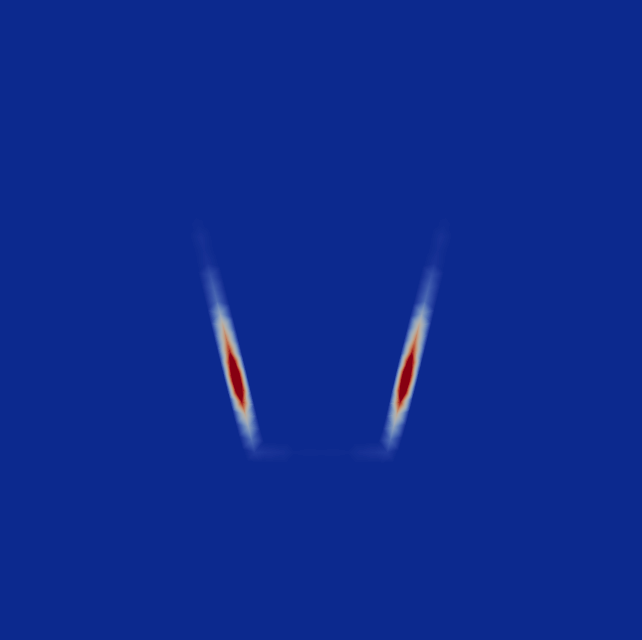}
         \caption{$\mu$ at $t=0.5$}
     \end{subfigure}
     \hfill
     \begin{subfigure}[t]{0.3\textwidth}
         \centering
         \includegraphics[height=4cm]{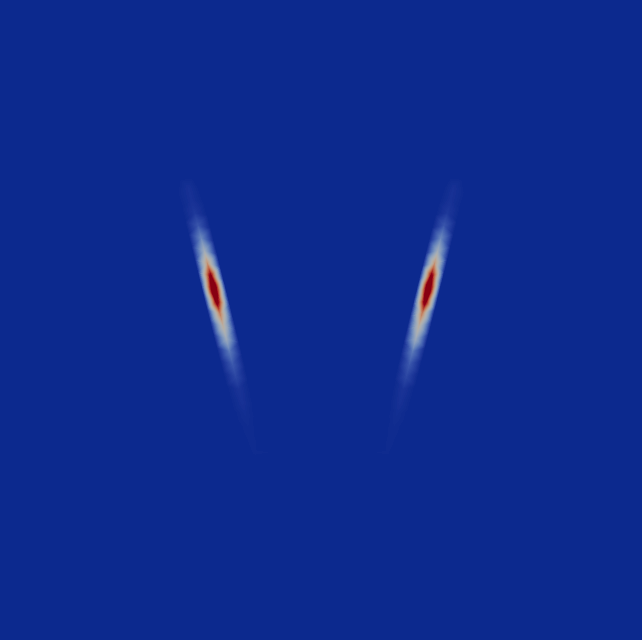}
         \caption{$\mu$ at $t=0.8$}
     \end{subfigure}
     \begin{subfigure}[t]{0.05\textwidth}
         \centering\includegraphics[height=4cm]{plots/ucurve/smallalph_curve_cb.png}
     \end{subfigure}
     \caption{Time-evolution for $\alpha_1=100=\alpha_2$ and a u-shaped curve}
     \label{fig:ucurve_largealph}
\end{figure}

\begin{example}\label{quadcurve}
    Let
    \begin{align*}
        \rho_0 (x,y) &:=  f_{0.2, 0.2, 0.1}(x,y)
        \intertext{and}
        \rho_1 (x,y) &:= f_{0.8, 0.5, 0.1}(x,y)  + f_{0.2, 0.8, 0.05}(x,y) + f_{0.4, 0.8, 0.05}(x,y) + f_{0.6, 0.8, 0.05}(x,y) 
    \end{align*}
    be given. 
    The (fixed) preferential path is chosen as the piecewise linear interpolation of the mapping $\gamma(t) = t^2+0.2$ between the points $t\in \{\frac{k}{10} \ | \ k =1,\ldots,8\}$.
    Again, we restrict and rescale the initial and final data illustrated in \autoref{fig:quadcurve_data} and compare the choices $\alpha_1=0.01=\alpha_2$ and $\alpha_1=100=\alpha_2$. 
    \autoref{fig:quadcurve_smallalph} and \autoref{fig:quadcurve_largealph} show the numerical results. 
    As expected, for the small parameter values we are able to observe mass transportation along the curve that is faster. In contrast to Example \ref{ucurve}, mass is transported through the bulk as well, especially between the initial data and the part of the final data supported above the curve, thus illustrating the influence of the particular geometry of the problem. 
    For the second set of parameters transport along the curve is slower. This behaviour is similar to the one observed for large parameter choices in the previous example.
\end{example}

\begin{figure}
     \centering
     \begin{subfigure}[b]{0.45\textwidth}
         \centering
         \includegraphics[width=\textwidth]{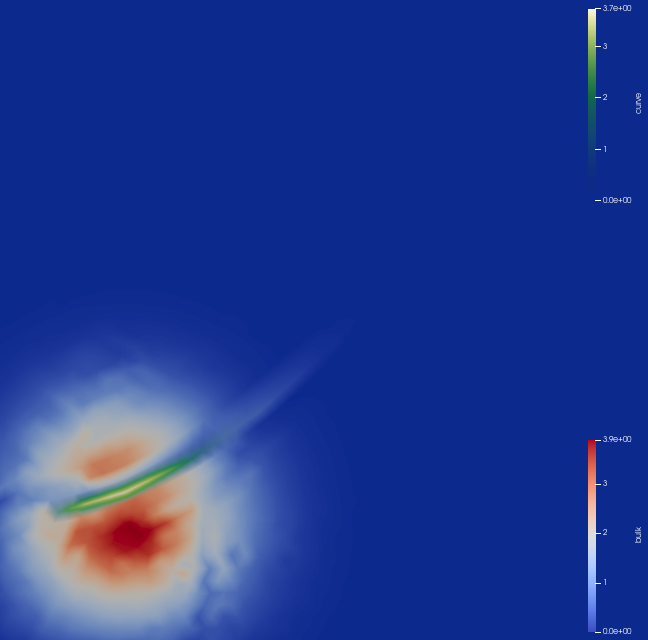}
         \caption{Initial data $\rho_0$}
         \label{fig:ucruve_rho0}
     \end{subfigure}
     \hfill
     \begin{subfigure}[b]{0.45\textwidth}
         \centering
         \includegraphics[width=\textwidth]{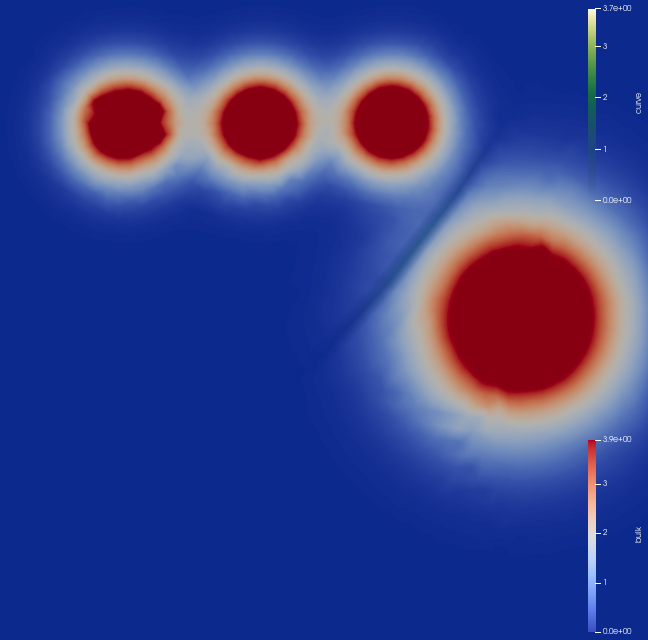}
         \caption{Final data $\rho_1$}
         \label{fig:ucurve_rho1}
     \end{subfigure}
     \caption{Initial and final data for the branch of a quadratic curve}
     \label{fig:quadcurve_data}
\end{figure}

\begin{figure}
     \centering
     \begin{subfigure}[t]{0.3\textwidth}
         \centering
         \includegraphics[height=4cm]{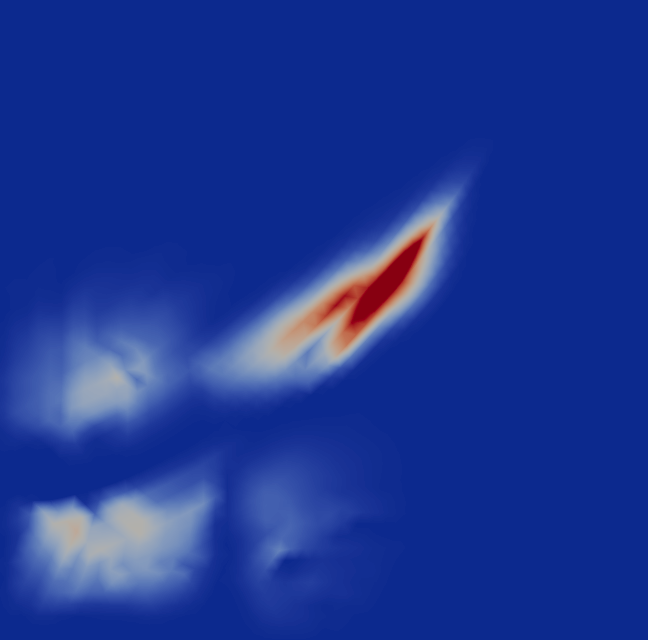}
         \caption{$\rho$ at $t=0.2$}
     \end{subfigure}
     \hfill
     \begin{subfigure}[t]{0.3\textwidth}
         \centering
         \includegraphics[height=4cm]{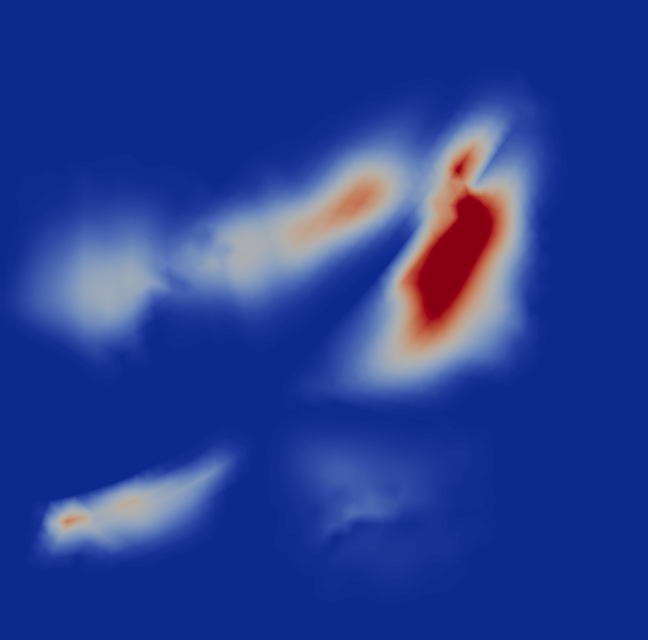}
         \caption{$\rho$ at $t=0.5$}
     \end{subfigure}
     \hfill
     \begin{subfigure}[t]{0.3\textwidth}
         \centering
         \includegraphics[height=4cm]{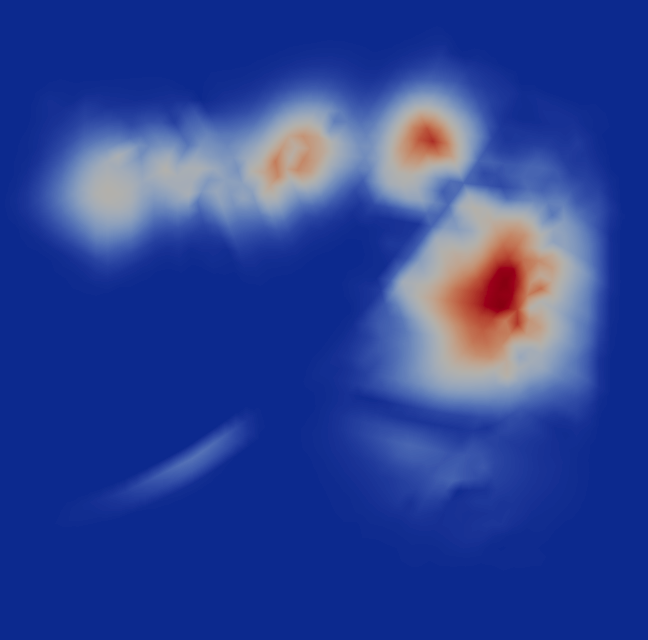}
         \caption{$\rho$ at $t=0.8$}
     \end{subfigure}
     \begin{subfigure}[t]{0.05\textwidth}
         \centering\includegraphics[height=4cm]{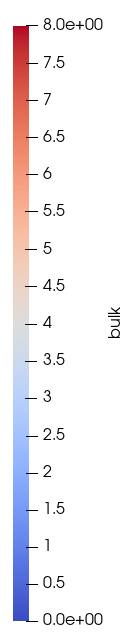}
     \end{subfigure}

     \centering
     \begin{subfigure}[t]{0.3\textwidth}
         \centering
         \includegraphics[height=4cm]{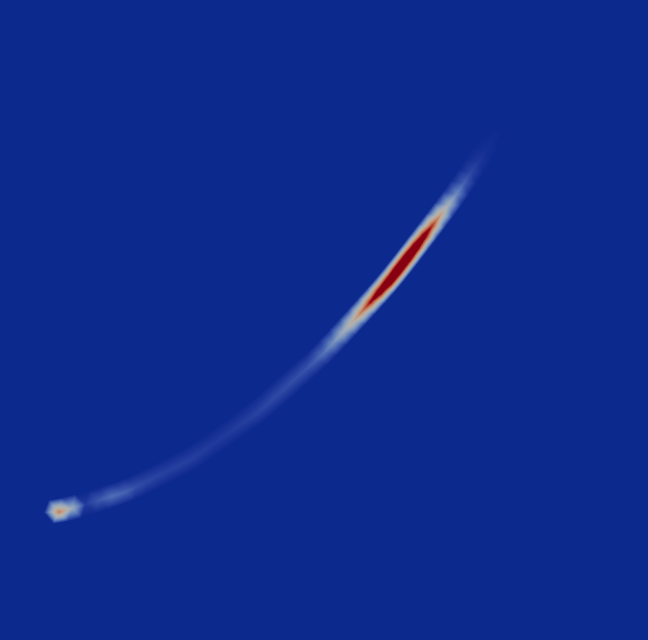}
         \caption{$\mu$ at $t=0.2$}
     \end{subfigure}
     \hfill
     \begin{subfigure}[t]{0.3\textwidth}
         \centering
         \includegraphics[height=4cm]{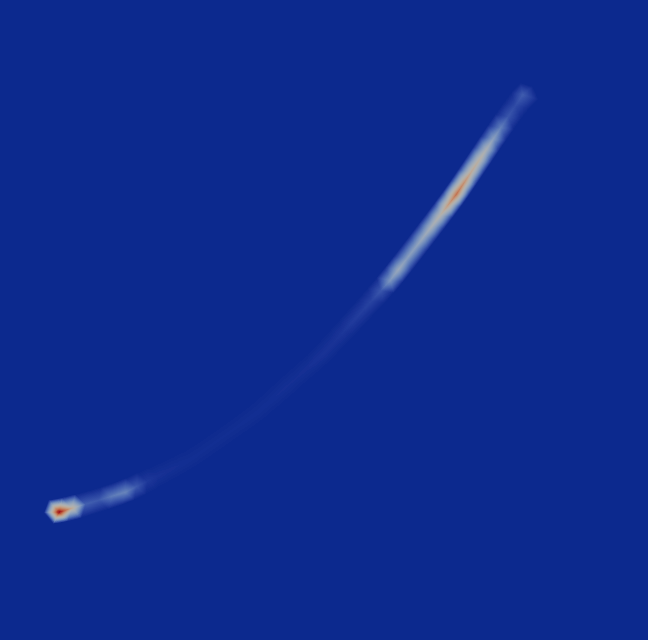}
         \caption{$\mu$ at $t=0.5$}
     \end{subfigure}
     \hfill
     \begin{subfigure}[t]{0.3\textwidth}
         \centering
         \includegraphics[height=4cm]{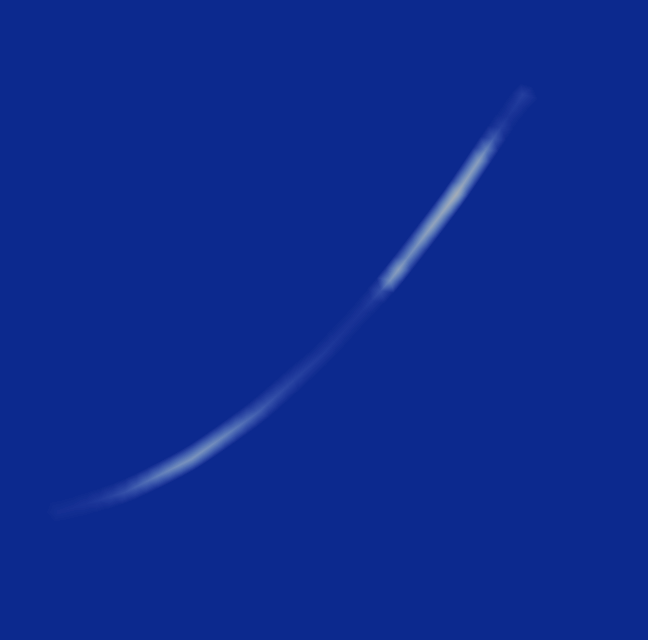}
         \caption{$\mu$ at $t=0.8$}
     \end{subfigure}
     \begin{subfigure}[t]{0.05\textwidth}
         \centering\includegraphics[height=4cm]{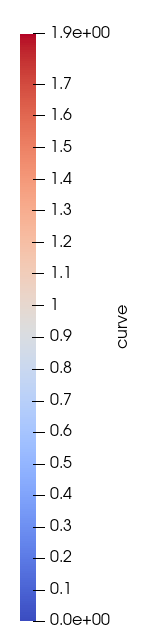}
     \end{subfigure}
     \caption{Time-evolution for $\alpha_1=0.01=\alpha_2$ and the branch of a quadratic curve}
     \label{fig:quadcurve_smallalph}
\end{figure}

\begin{figure}
     \centering
     \begin{subfigure}[t]{0.3\textwidth}
         \centering
         \includegraphics[height=4cm]{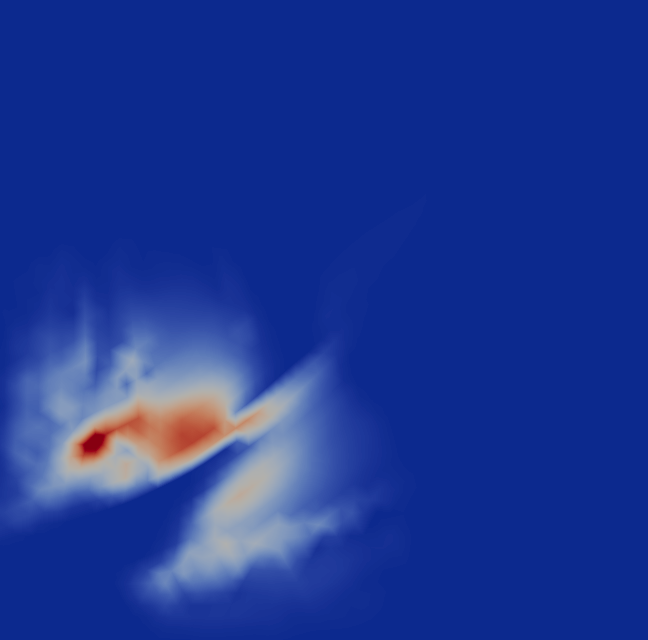}
         \caption{$\rho$ at $t=0.2$}
     \end{subfigure}
     \hfill
     \begin{subfigure}[t]{0.3\textwidth}
         \centering
         \includegraphics[height=4cm]{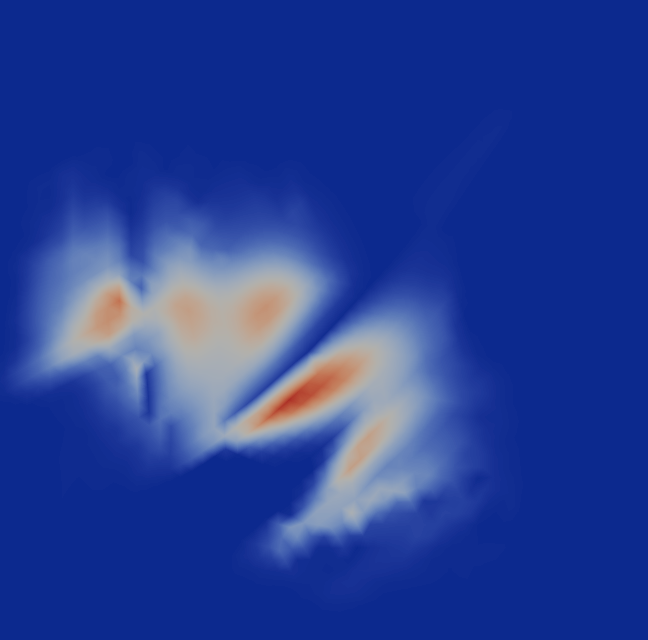}
         \caption{$\rho$ at $t=0.5$}
     \end{subfigure}
     \hfill
     \begin{subfigure}[t]{0.3\textwidth}
         \centering
         \includegraphics[height=4cm]{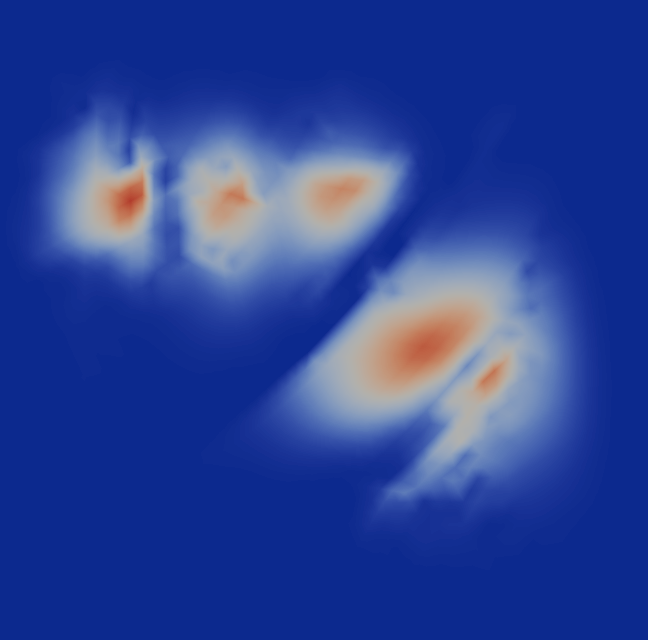}
         \caption{$\rho$ at $t=0.8$}
     \end{subfigure}
     \begin{subfigure}[t]{0.05\textwidth}
         \centering\includegraphics[height=4cm]{plots/quadraticcurve/bulk_cb.png}
     \end{subfigure}

     \centering
     \begin{subfigure}[t]{0.3\textwidth}
         \centering
         \includegraphics[height=4cm]{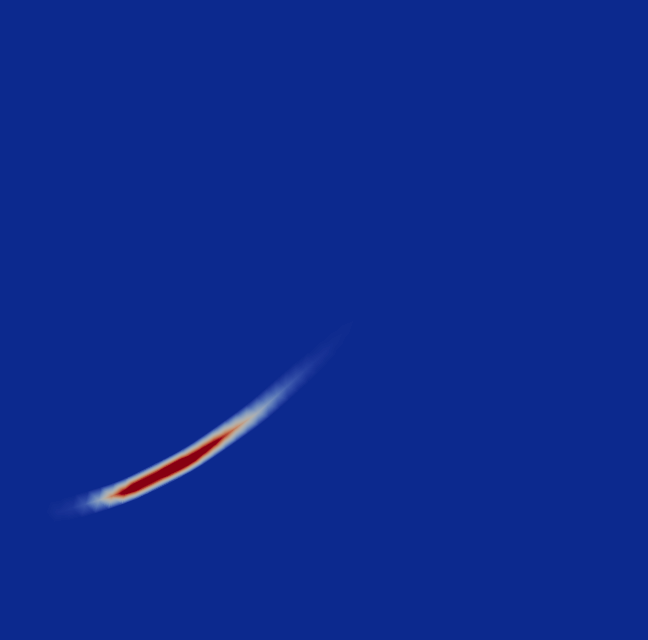}
         \caption{$\mu$ at $t=0.2$}
     \end{subfigure}
     \hfill
     \begin{subfigure}[t]{0.3\textwidth}
         \centering
         \includegraphics[height=4cm]{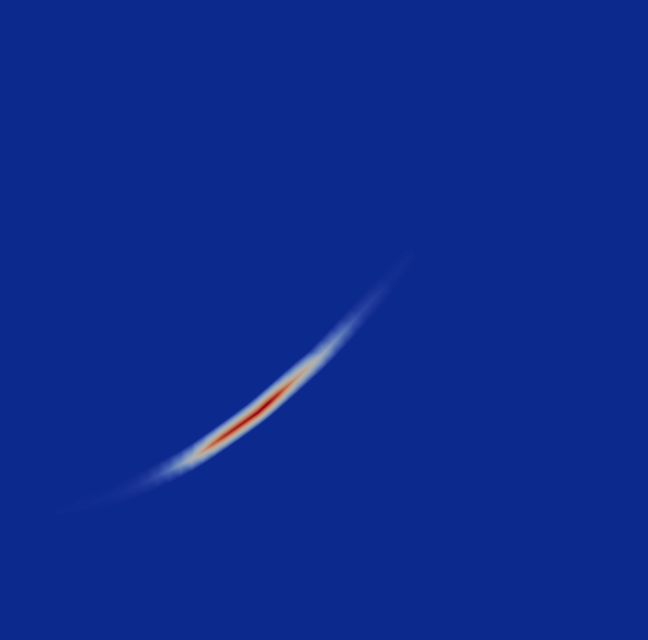}
         \caption{$\mu$ at $t=0.5$}
     \end{subfigure}
     \hfill
     \begin{subfigure}[t]{0.3\textwidth}
         \centering
         \includegraphics[height=4cm]{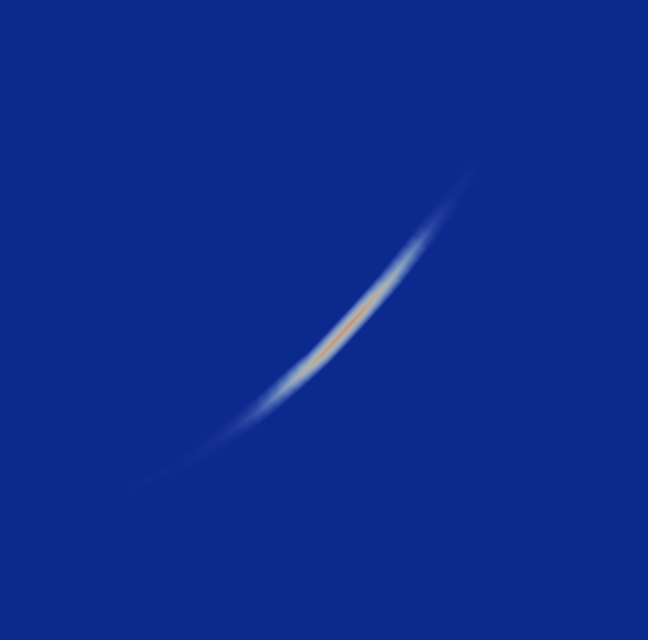}
         \caption{$\mu$ at $t=0.8$}
     \end{subfigure}
     \begin{subfigure}[t]{0.05\textwidth}
         \centering\includegraphics[height=4cm]{plots/quadraticcurve/curve_cb.png}
     \end{subfigure}
     \caption{Time-evolution for $\alpha_1=100=\alpha_2$ and the branch of a quadratic curve}
     \label{fig:quadcurve_largealph}
\end{figure}

\subsubsection{Varying preferential path}

In what follows, we test Algorithm \ref{alg:MinCurve} for several examples of varying preferential paths. All meshes have meshsize $h=0.1$ in spatial and $h=0.02$ in temporal direction. Compared to the fixed curve examples, we chose a coarser discretization in spatial direction because of increased memory usage.

\begin{example}\label{line}
   Let 
   \begin{align*}
       \rho_0(x,y) := f_{0.25,0.25,0.05}(x,y) \quad \text{and} \quad \rho_1(x,y) := f_{0.75,0.75,0.05}(x,y)
   \end{align*}
   be given. As an initial guess for the optimal curve $\gamma$ we start with $\gamma^0(t) \equiv \frac{1}{2}$ approximated using three points. When applying Algorithm \ref{alg:MinCurve}, we expect the curve to rotate in the direction of $\gamma(t) = t$ connecting $\rho_0$ and $\rho_1$.
   For the stepsize $0.01$ and scaling parameter $c=0.001$, \autoref{fig:varcurveline} displays every 20th iteration of the curve as well as the transport costs. The irregularities of the transport cost is explained by manual changes in the stepsize and scaling. This prevents the curve from jumping between two configurations of similar costs without improving costs. 
   As expected, we recover a straight line. The length penalty compresses the curve in the sense that it does not connect the centre of initial and final data. 
   In \autoref{fig:meas_varlincurve} the measures resulting after 20000 iterations of Algorithm \ref{alg:AugmLagr} with the last curve configuration fixed are displayed. At first mass belonging to the final data supported in the bulk is transported along the curve and through the bulk domain, in a second phase, the mass remaining on the curve is transported. Overall, transport along the curve is faster.
\end{example} 

\begin{figure}[h!]
    \centering
     \begin{subfigure}[b]{0.4\textwidth}
         \centering
         \includegraphics[width=\textwidth]{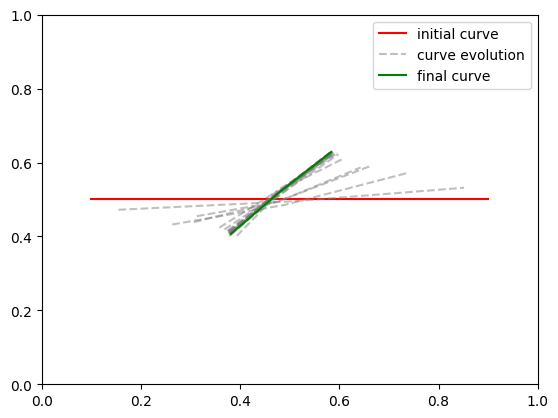}
         \caption{Evolution of the curve}
         \label{fig:varcurveline}
     \end{subfigure}
     \hfill
     \begin{subfigure}[b]{0.4\textwidth}
         \centering
         \includegraphics[width=\textwidth]{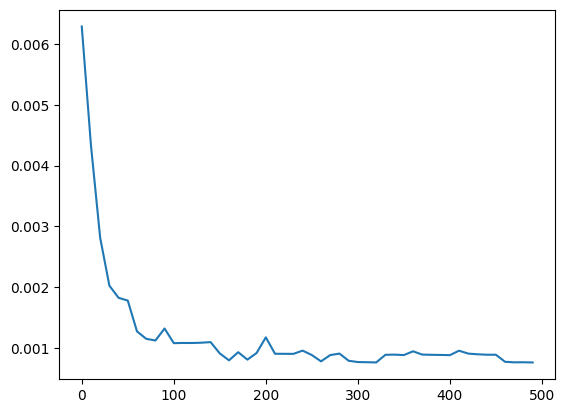}
         \caption{Evolution of the transport costs}
     \end{subfigure}
    \caption{Example \ref{line}}
    \label{fig:varcurve_line}
\end{figure}

\begin{figure}[h!]
     \centering
     \begin{subfigure}[t]{0.3\textwidth}
         \centering
         \includegraphics[height=4cm]{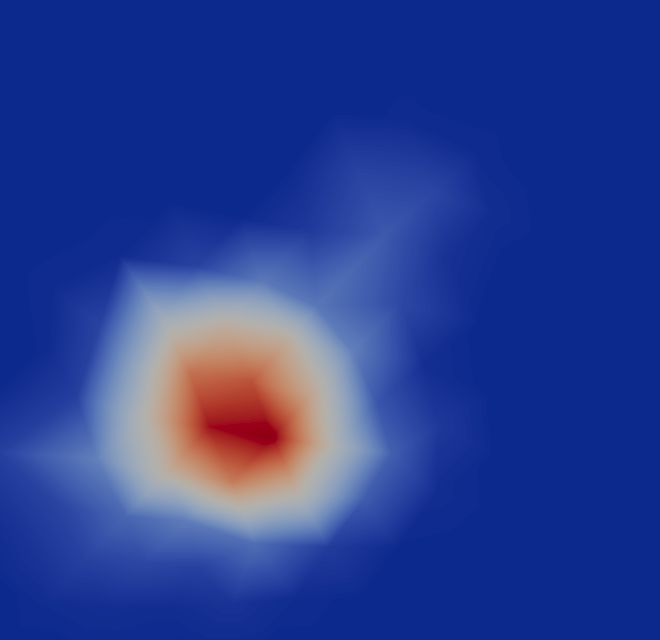}
         \caption{$\rho$ at $t=0.2$}
     \end{subfigure}
     \hfill
     \begin{subfigure}[t]{0.3\textwidth}
         \centering
         \includegraphics[height=4cm]{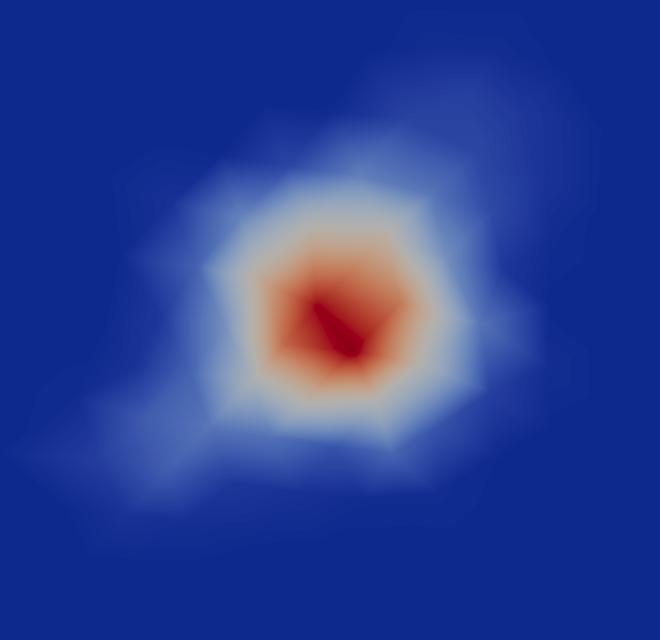}
         \caption{$\rho$ at $t=0.5$}
     \end{subfigure}
     \hfill
     \begin{subfigure}[t]{0.3\textwidth}
         \centering
         \includegraphics[height=4cm]{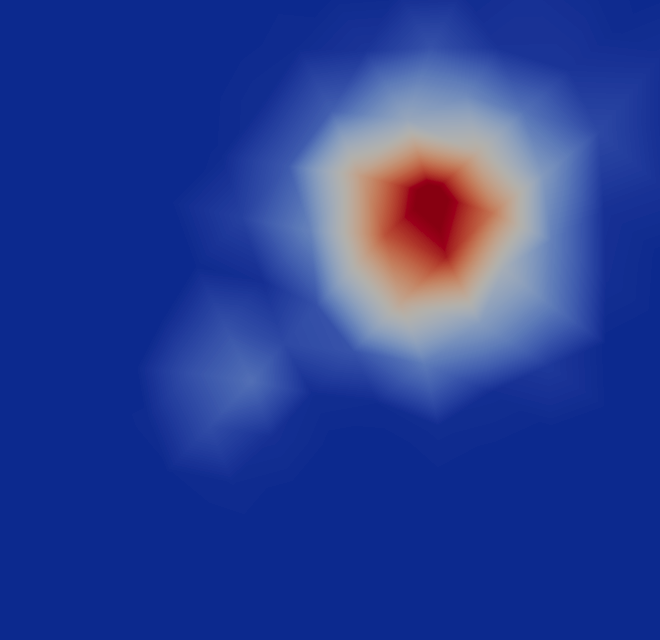}
         \caption{$\rho$ at $t=0.8$}
     \end{subfigure}
     \begin{subfigure}[t]{0.05\textwidth}
         \centering\includegraphics[height=4cm]{plots/quadraticcurve/bulk_cb.png}
     \end{subfigure}

     \centering
     \begin{subfigure}[t]{0.3\textwidth}
         \centering
         \includegraphics[height=4cm]{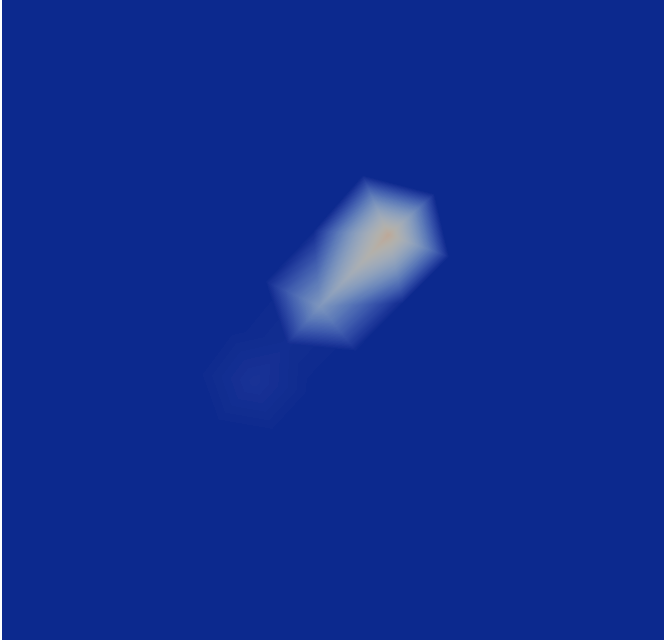}
         \caption{$\mu$ at $t=0.2$}
     \end{subfigure}
     \hfill
     \begin{subfigure}[t]{0.3\textwidth}
         \centering
         \includegraphics[height=4cm]{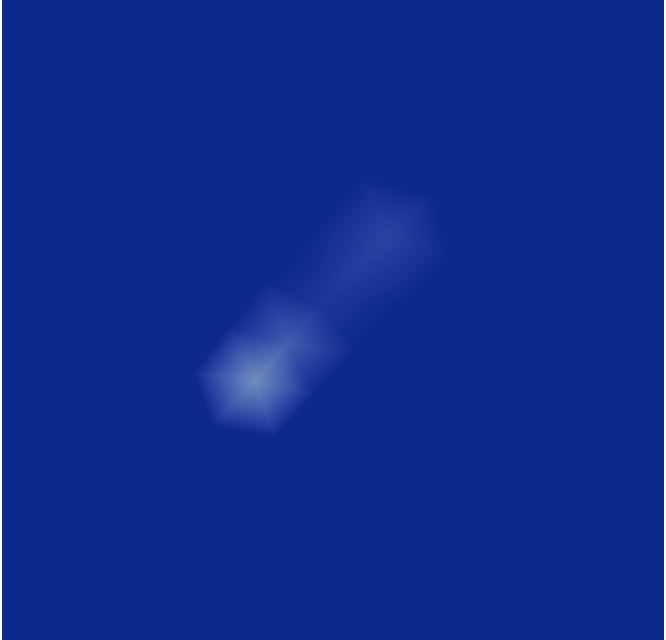}
         \caption{$\mu$ at $t=0.5$}
     \end{subfigure}
     \hfill
     \begin{subfigure}[t]{0.3\textwidth}
         \centering
         \includegraphics[height=4cm]{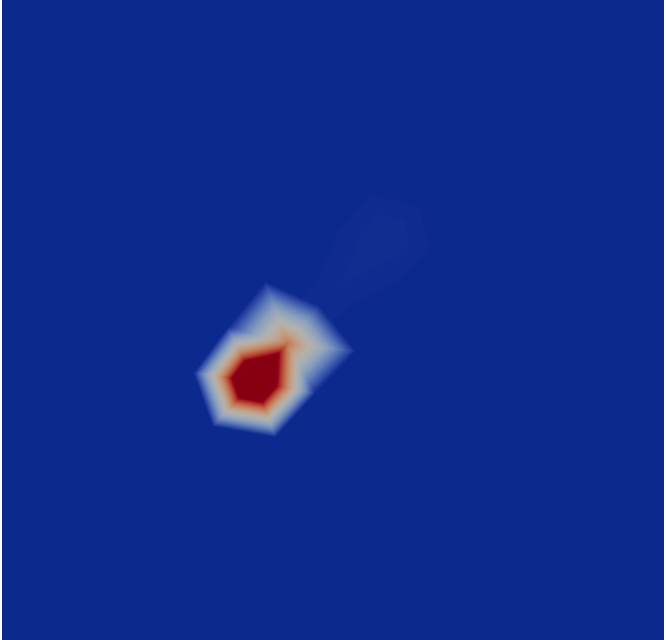}
         \caption{$\mu$ at $t=0.8$}
     \end{subfigure}
     \begin{subfigure}[t]{0.05\textwidth}
         \centering\includegraphics[height=4cm]{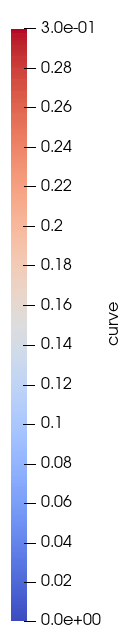}
     \end{subfigure}
     \caption{Time-evolution of measures for $\alpha_1=0.01=\alpha_2$ and the final curve from Example \ref{line}}
     \label{fig:meas_varlincurve}
\end{figure}

\begin{example}\label{varying_ucurve}
    As a second example, we consider initial and final data as in Example \ref{ucurve}. We expect to recover a curve similar to the one used for the fixed-curve example. Since such a curve differs from a straight line, we need to use small scaling parameters to allow for such geometries in the minimization. Again, let $\gamma^0(t) \equiv \frac{1}{2}$ be the initial curve discretized by three points. 
    In \autoref{fig:varcurve_u} every 20th step of the evolution using an initial stepsize of $0.01$ and a scaling parameter $c=0.001$ is shown. Indeed, the curve evolves towards a v-shaped structure. However, the length and boundary constraints prevent it from being symmetric. Again, we adapt stepsize and scaling parameter manually. 
    \autoref{fig:meas_varucurve} shows the measures arising as numerical solutions of the coupled system for the curve after the last minimization step and 20000 iterations. Similar to Example \ref{ucurve}, most of the mass is transported along the curve. However, there still is transport in the bulk region becacuse of the non-symmetric structure of the resulting curve.
\end{example}

\begin{figure}[h!]
    \centering
     \begin{subfigure}[b]{0.4\textwidth}
         \centering
         \includegraphics[width=\textwidth]{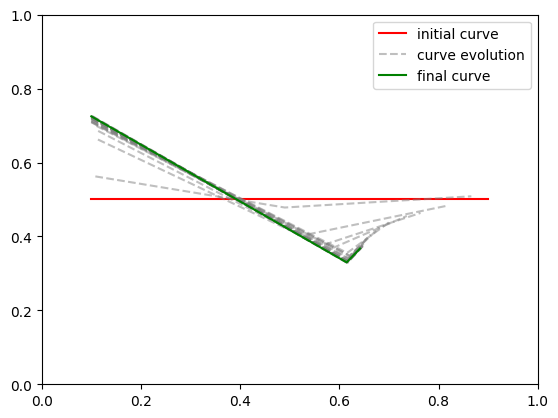}
         \caption{Evolution of the curve}
     \end{subfigure}
     \hfill
     \begin{subfigure}[b]{0.4\textwidth}
         \centering
         \includegraphics[width=\textwidth]{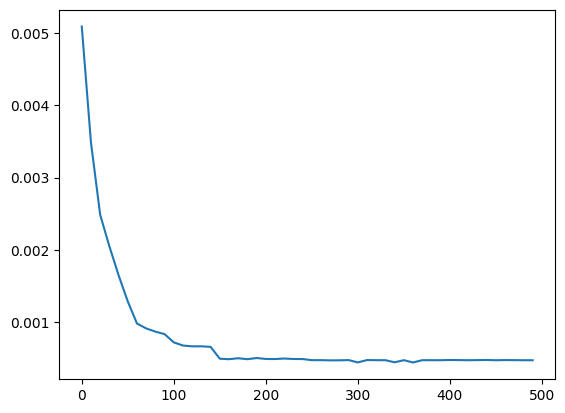}
         \caption{Evolution of the transport costs}
    \end{subfigure}
    \caption{Example \ref{varying_ucurve}}
    \label{fig:varcurve_u}
\end{figure}

\begin{figure}[h]
     \centering
     \begin{subfigure}[t]{0.3\textwidth}
         \centering
         \includegraphics[height=4cm]{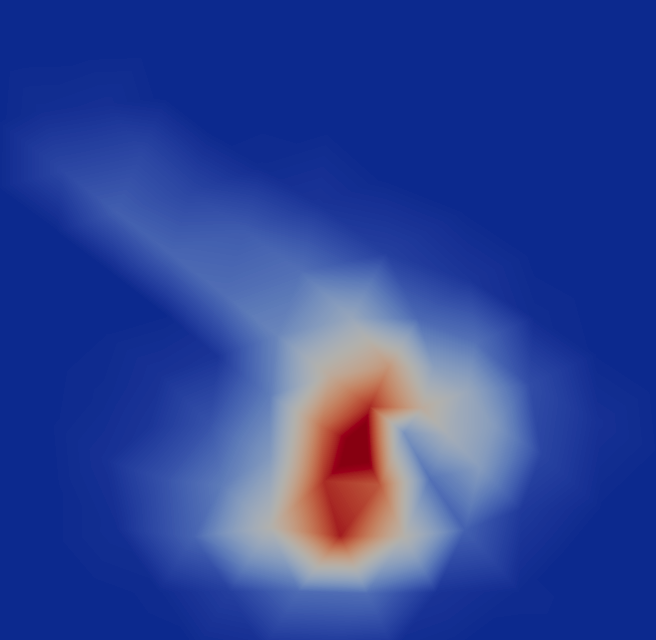}
         \caption{$\rho$ at $t=0.2$}
     \end{subfigure}
     \hfill
     \begin{subfigure}[t]{0.3\textwidth}
         \centering
         \includegraphics[height=4cm]{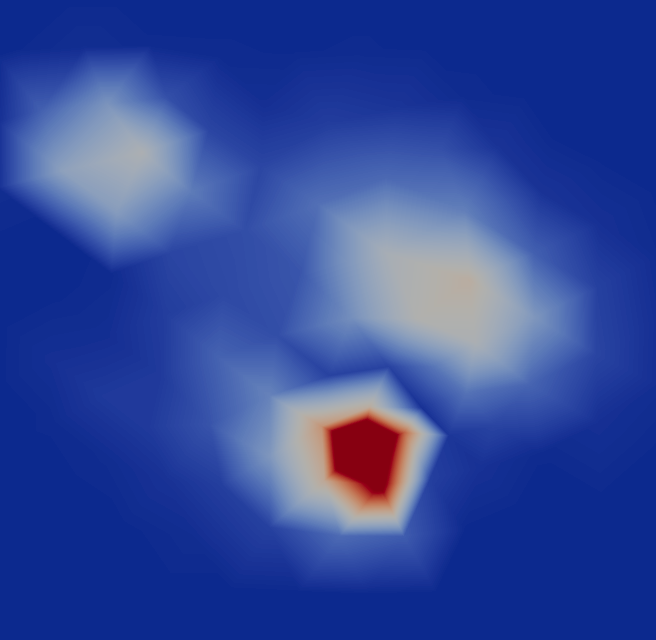}
         \caption{$\rho$ at $t=0.5$}
     \end{subfigure}
     \hfill
     \begin{subfigure}[t]{0.3\textwidth}
         \centering
         \includegraphics[height=4cm]{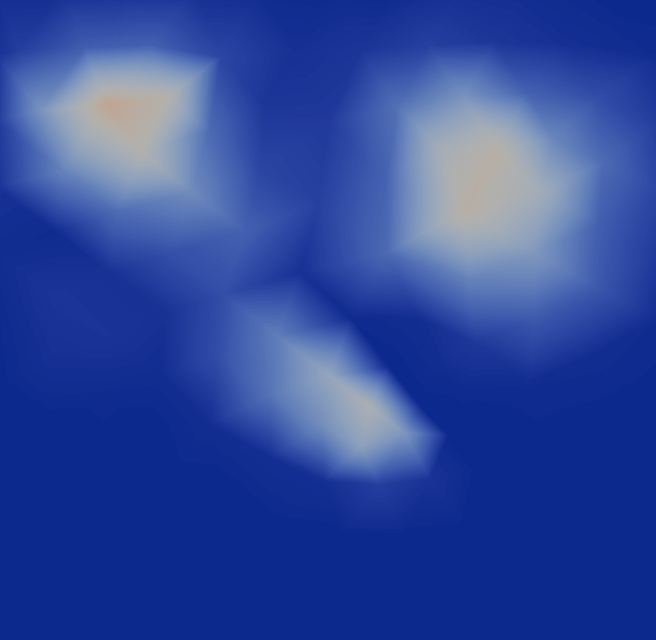}
         \caption{$\rho$ at $t=0.8$}
     \end{subfigure}
     \begin{subfigure}[t]{0.05\textwidth}
         \centering\includegraphics[height=4cm]{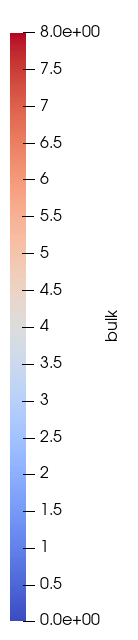}
     \end{subfigure}

     \centering
     \begin{subfigure}[t]{0.3\textwidth}
         \centering
         \includegraphics[height=4cm]{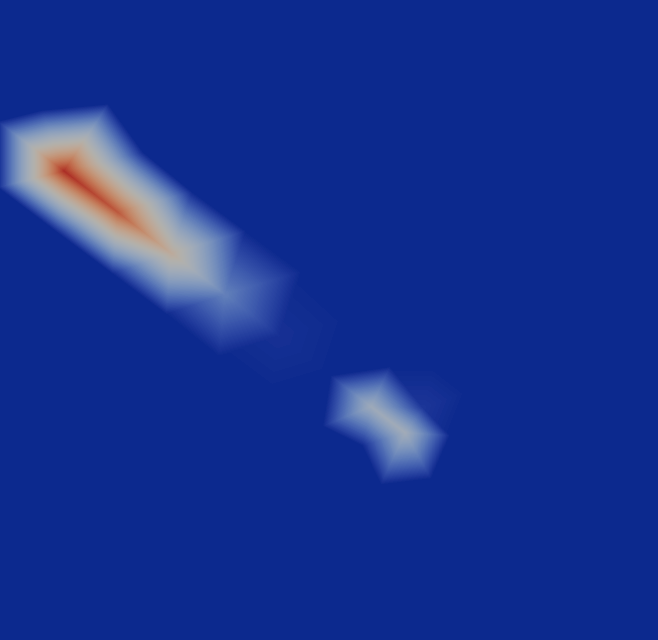}
         \caption{$\mu$ at $t=0.2$}
     \end{subfigure}
     \hfill
     \begin{subfigure}[t]{0.3\textwidth}
         \centering
         \includegraphics[height=4cm]{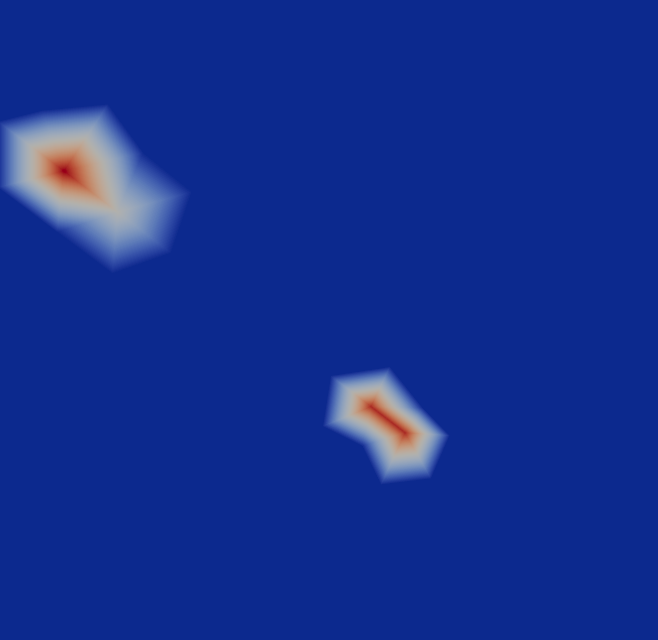}
         \caption{$\mu$ at $t=0.5$}
     \end{subfigure}
     \hfill
     \begin{subfigure}[t]{0.3\textwidth}
         \centering
         \includegraphics[height=4cm]{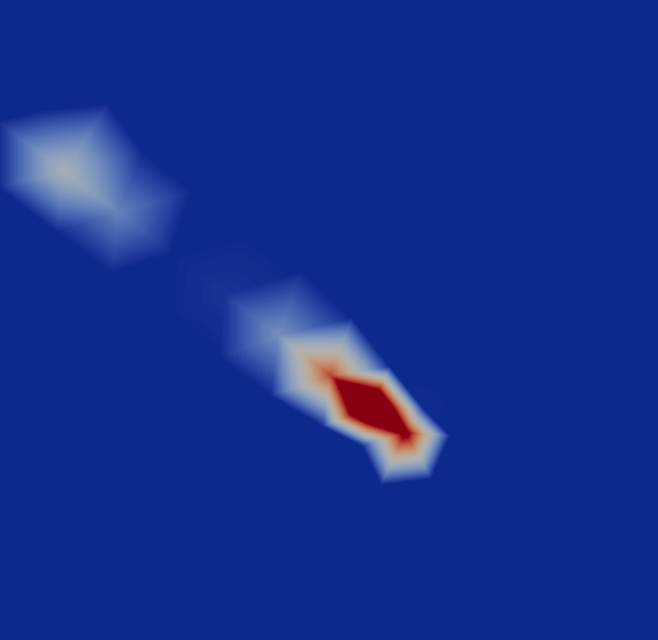}
         \caption{$\mu$ at $t=0.8$}
     \end{subfigure}
     \begin{subfigure}[t]{0.05\textwidth}
         \centering\includegraphics[height=4cm]{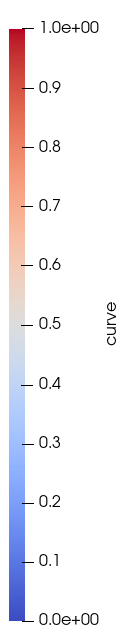}
     \end{subfigure}
     \caption{Time-evolution of measures for $\alpha_1=0.01=\alpha_2$ and the final curve from Example \ref{varying_ucurve}}
     \label{fig:meas_varucurve}
\end{figure}

\newpage 

\section*{Acknowledgements}
The authors like to thank Elias Döhrer (TU Chemnitz) for fruitful discussions and advice on Tangent-Point energies and Guosheng Fu (University of Notre Dame) for providing their code as a starting point for our implementation. 
M.C.'s research was supported by the NWO-M1 grant Curve Ensemble Gradient Descents for Sparse Dynamic Problems (Grant Number OCENW.M.22.302).
\section*{Data Availability}
No datasets were generated or analyzed during the current study.

\section*{Conflicts of Interest }
All authors declare that they have no conflict of interests related to this publication.

\bibliographystyle{alpha}
\bibliography{sample}

@article{Papadakis2014,
author = {Papadakis, Nicolas and Peyré, Gabriel and Oudet, Edouard},
title = {Optimal Transport with Proximal Splitting},
journal = {SIAM Journal on Imaging Sciences},
volume = {7},
number = {1},
pages = {212-238},
year = {2014},
doi = {10.1137/130920058},
URL = {https://doi.org/10.1137/130920058},
eprint = {https://doi.org/10.1137/130920058}
}

@article{Carrillo2010,
title = {Nonlinear mobility continuity equations and generalized displacement convexity},
journal = {Journal of Functional Analysis},
volume = {258},
number = {4},
pages = {1273-1309},
year = {2010},
issn = {0022-1236},
doi = {https://doi.org/10.1016/j.jfa.2009.10.016},
url = {https://www.sciencedirect.com/science/article/pii/S0022123609004261},
author = {José A. Carrillo and Stefano Lisini and Giuseppe Savaré and Dejan Slepčev},
}

@article {Yan2018,
    AUTHOR = {Yan, Ming},
     TITLE = {A new primal-dual algorithm for minimizing the sum of three
              functions with a linear operator},
   JOURNAL = {J. Sci. Comput.},
  FJOURNAL = {Journal of Scientific Computing},
    VOLUME = {76},
      YEAR = {2018},
    NUMBER = {3},
     PAGES = {1698--1717},
      ISSN = {0885-7474},
       DOI = {10.1007/s10915-018-0680-3},
}

@article{Carrillo2022_PrimalDual,
author = {Carrillo, Jos\'{e} A. and Craig, Katy and Wang, Li and Wei, Chaozhen},
title = {Primal Dual Methods for {W}asserstein Gradient Flows},
year = {2022},
issue_date = {Apr 2022},
publisher = {Springer-Verlag},
address = {Berlin, Heidelberg},
volume = {22},
number = {2},
issn = {1615-3375},
url = {https://doi.org/10.1007/s10208-021-09503-1},
doi = {10.1007/s10208-021-09503-1},
journal = {Found. Comput. Math.},
month = {apr},
pages = {389–443},
numpages = {55}
}

@article {Cances2020_FVSchemeJKO,
    AUTHOR = {Canc\`es, Cl\'{e}ment and Gallou\"{e}t, Thomas O. and Todeschi,
              Gabriele},
     TITLE = {A variational finite volume scheme for {W}asserstein gradient
              flows},
   JOURNAL = {Numer. Math.},
  FJOURNAL = {Numerische Mathematik},
    VOLUME = {146},
      YEAR = {2020},
    NUMBER = {3},
     PAGES = {437--480},
      ISSN = {0029-599X},
   MRCLASS = {65M08 (35K65 49M29 49Q22 65M12)},
  MRNUMBER = {4169480},
MRREVIEWER = {Yinhua Xia},
       DOI = {10.1007/s00211-020-01153-9},
       URL = {https://doi.org/10.1007/s00211-020-01153-9},
}

@inproceedings{dawood2010continuity,
  title={A continuity equation based optical flow method for cardiac motion correction in 3D PET data},
  author={Dawood, Mohammad and Brune, Christoph and Jiang, Xiaoyi and B{\"u}ther, Florian and Burger, Martin and Schober, Otmar and Sch{\"a}fers, Michael and Sch{\"a}fers, Klaus P},
  booktitle={Medical Imaging and Augmented Reality: 5th International Workshop, MIAR 2010, Beijing, China, September 19-20, 2010. Proceedings 5},
  pages={88--97},
  year={2010},
  organization={Springer}
}

@article{schmitzer2019dynamic,
  title={Dynamic cell imaging in PET with optimal transport regularization},
  author={Schmitzer, Bernhard and Sch{\"a}fers, Klaus P and Wirth, Benedikt},
  journal={IEEE Transactions on Medical Imaging},
  volume={39},
  number={5},
  pages={1626--1635},
  year={2019},
  publisher={IEEE}
}

@article{bredies2023generalized,
  title={A generalized conditional gradient method for dynamic inverse problems with optimal transport regularization},
  author={Bredies, Kristian and Carioni, Marcello and Fanzon, Silvio and Romero, Francisco},
  journal={Foundations of Computational Mathematics},
  volume={23},
  number={3},
  pages={833--898},
  year={2023},
  publisher={Springer}
}

@inproceedings{onken2021ot,
  title={Ot-flow: Fast and accurate continuous normalizing flows via optimal transport},
  author={Onken, Derek and Fung, Samy Wu and Li, Xingjian and Ruthotto, Lars},
  booktitle={Proceedings of the AAAI Conference on Artificial Intelligence},
  volume={35},
  pages={9223--9232},
  year={2021}
}

@article{carlier2008optimal,
  title={Optimal transportation with traffic congestion and Wardrop equilibria},
  author={Carlier, Guillaume and Jimenez, Cristian and Santambrogio, Filippo},
  journal={SIAM Journal on Control and Optimization},
  volume={47},
  number={3},
  pages={1330--1350},
  year={2008},
  publisher={SIAM}
}

@article{stephanovitch2024optimal,
  title={Optimal transport through a toll station},
  author={Stephanovitch, Arthur and Dong, Anqi and Georgiou, Tryphon T},
  journal={European Journal of Applied Mathematics},
  pages={1--25},
  year={2024},
  publisher={Cambridge University Press}
}

@article{maury2010macroscopic,
  title={A macroscopic crowd motion model of gradient flow type},
  author={Maury, Bertrand and Roudneff-Chupin, Aude and Santambrogio, Filippo},
  journal={Mathematical Models and Methods in Applied Sciences},
  volume={20},
  number={10},
  pages={1787--1821},
  year={2010},
  publisher={World Scientific}
}

@article{kondratyev2016new,
title = {{A new optimal transport distance on the space of finite Radon measures}},
volume = {21},
journal = {Advances in Differential Equations},
number = {11/12},
publisher = {Khayyam Publishing, Inc.},
pages = {1117 -- 1164},
year = {2016},
doi = {10.57262/ade/1476369298},
author={Kondratyev, Stanislav and Monsaingeon, L{\'e}onard and Vorotnikov, Dmitry},
}

@article{liero2018optimal,
  title={Optimal entropy-transport problems and a new Hellinger--Kantorovich distance between positive measures},
  author={Liero, Matthias and Mielke, Alexander and Savar{\'e}, Giuseppe},
  journal={Inventiones mathematicae},
  volume={211},
  number={3},
  pages={969--1117},
  year={2018},
  publisher={Springer}
}

@article{slepvcev2023nonlocal,
  title={Nonlocal wasserstein distance: Metric and asymptotic properties},
  author={Slep{\v{c}}ev, Dejan and Warren, Andrew},
  journal={Calculus of Variations and Partial Differential Equations},
  volume={62},
  number={9},
  pages={238},
  year={2023},
  publisher={Springer}
}

@book{gamkrelidze2013principles,
  title={Principles of optimal control theory},
  author={Gamkrelidze, R.},
  volume={7},
  year={2013},
  publisher={Springer Science \& Business Media}
}

@article{benamou2000computational,
  title={A computational fluid mechanics solution to the Monge-Kantorovich mass transfer problem},
  author={Benamou, Jean-David and Brenier, Yann},
  journal={Numerische Mathematik},
  volume={84},
  number={3},
  pages={375--393},
  year={2000},
  publisher={Springer-Verlag Berlin/Heidelberg}
}

@article{chizat2018unbalanced,
  title={Unbalanced optimal transport: Dynamic and Kantorovich formulations},
  author={Chizat, Lenaic and Peyr{\'e}, Gabriel and Schmitzer, Bernhard and Vialard, Fran{\c{c}}ois-Xavier},
  journal={Journal of Functional Analysis},
  volume={274},
  number={11},
  pages={3090--3123},
  year={2018},
  publisher={Elsevier}
}

@Article{BurgerHumpertPietschmann2023,
  author           = {Martin Burger and Ina Humpert and Jan-Frederik Pietschmann},
  title            = {Dynamic Optimal Transport on Networks},
  doi              = {10.1051/cocv/2023027},
  pages            = {54},
  url              = {https://doi.org/10.1051/cocv/2023027},
  volume           = {29},
  journal          = {{ESAIM}: Control, Optimisation and Calculus of Variations},
  modificationdate = {2024-10-29T10:24:59},
  publisher        = {{EDP} Sciences},
  year             = {2023},
}

@article{Fazeny2024,
  title={Optimal transport on gas networks},
  author={Fazeny, Ariane and Burger, Martin and Pietschmann, Jan-F},
  journal={European Journal of Applied Mathematics},
  pages={1--33},
  year={2025},
  publisher={Cambridge University Press}
}

@book{AGS2008,
  title={Gradient flows: in metric spaces and in the space of probability measures},
  author={Ambrosio, Luigi and Gigli, Nicola and Savar{\'e}, Giuseppe},
  year={2008},
  publisher={Springer Science \& Business Media}
}

@article{Otto2001,
  title={The geometry of dissipative evolution equations: the porous medium equation},
  author={Otto, Felix},
  journal={Comm. Partial Differential Equations},
  volume={26},
  year={2001},
  publisher={Taylor \& Francis}
}

@article{Erbar2022,
title = {Gradient flow formulation of diffusion equations in the Wasserstein space over a Metric graph},
journal = {Networks and Heterogeneous Media},
volume = {17},
number = {5},
pages = {687-717},
year = {2022},
issn = {1556-1801},
doi = {10.3934/nhm.2022023},
url = {https://www.aimsciences.org/article/id/62aabe4d2d80b75d2b6d88ce},
author = {Matthias Erbar and Dominik Forkert and Jan Maas and Delio Mugnolo},
}

@article{JKO98,
  title={The variational formulation of the {F}okker--{P}lanck equation},
  author={Jordan, Richard and Kinderlehrer, David and Otto, Felix},
  journal={SIAM {J}ournal on {M}athematical {A}nalysis},
  volume={29},
  number={1},
  pages={1--17},
  year={1998},
  publisher={SIAM}
}

@misc{Heinze2024,
Author = {Georg Heinze and Jan-Frederik Pietschmann and André Schlichting},
Title = {Gradient flows on metric graphs with reservoirs: Microscopic derivation and multiscale limits},
Year = {2024},
Eprint = {arXiv:2412.16775},
note = {arXiv:2412.16775 [math]}
}

@article{Lisini2010,
  doi = {10.1007/s00229-010-0371-3},
  url = {https://doi.org/10.1007/s00229-010-0371-3},
  year = {2010},
  month = jun,
  publisher = {Springer Science and Business Media {LLC}},
  volume = {133},
  number = {1-2},
  pages = {197--224},
  author = {Stefano Lisini and Antonio Marigonda},
  title = {On a class of modified Wasserstein distances induced by concave mobility functions defined on bounded intervals},
  journal = {manuscripta mathematica}
}

@article{Dolbeault_2008,
   title={A new class of transport distances between measures},
   volume={34},
   ISSN={1432-0835},
   url={http://dx.doi.org/10.1007/s00526-008-0182-5},
   DOI={10.1007/s00526-008-0182-5},
   number={2},
   journal={Calculus of Variations and Partial Differential Equations},
   publisher={Springer Science and Business Media LLC},
   author={Dolbeault, Jean and Nazaret, Bruno and Savaré, Giuseppe},
   year={2008},
   month={Jun},
   pages={193–231}
}

@article{monsaingeon2021new,
  title={A new transportation distance with bulk/interface interactions and flux penalization},
  author={Monsaingeon, L{\'e}onard},
  journal={Calculus of Variations and Partial Differential Equations},
  volume={60},
  number={3},
  pages={101},
  year={2021},
  publisher={Springer}
}

@book{ambrosio2000functions,
  title={Functions of bounded variation and free discontinuity problems},
  author={Ambrosio, Luigi and Fusco, Nicola and Pallara, Diego},
  year={2000},
  publisher={Oxford university press}
}

@article{bredies2020optimal,
  title={An optimal transport approach for solving dynamic inverse problems in spaces of measures},
  author={Bredies, Kristian and Fanzon, Silvio},
  journal={ESAIM: Mathematical Modelling and Numerical Analysis},
  volume={54},
  number={6},
  pages={2351--2382},
  year={2020},
  publisher={EDP Sciences}
}

@book{dal2012introduction,
  title={An introduction to $\Gamma$-convergence},
  author={Dal Maso, Gianni},
  volume={8},
  year={2012},
  publisher={Springer Science \& Business Media}
}

@article{Blatt2013,
author = {Blatt, Simon},
title = {THE ENERGY SPACES OF THE TANGENT POINT ENERGIES},
journal = {Journal of Topology and Analysis},
volume = {05},
number = {03},
pages = {261-270},
year = {2013},
doi = {10.1142/S1793525313500131},
URL = {https://doi.org/10.1142/S1793525313500131},
eprint = {https://doi.org/10.1142/S1793525313500131}
}

@article{STRZELECKI_2012,
   title={Tangent-Point Self-Avoidance Energies for Curves},
   volume={21},
   ISSN={1793-6527},
   url={http://dx.doi.org/10.1142/S0218216511009960},
   DOI={10.1142/s0218216511009960},
   number={05},
   journal={Journal of Knot Theory and Its Ramifications},
   publisher={World Scientific Pub Co Pte Lt},
   author={Strzelecki, Paweł and von Der Mosel, Heiko},
   year={2012},
   month=apr, pages={1250044} }

@article{DINEZZA2012521,
title = {Hitchhiker's guide to the fractional Sobolev spaces},
journal = {Bulletin des Sciences Math\'{e}matiques},
volume = {136},
number = {5},
pages = {521-573},
year = {2012},
issn = {0007-4497},
doi = {https://doi.org/10.1016/j.bulsci.2011.12.004},
url = {https://www.sciencedirect.com/science/article/pii/S0007449711001254},
author = {Eleonora {Di Nezza} and Giampiero Palatucci and Enrico Valdinoci}
}

@article{walker_shape_2016,
	title = {Shape optimization of self-avoiding curves},
	volume = {311},
	issn = {0021-9991},
	url = {https://www.sciencedirect.com/science/article/pii/S0021999116000656},
	doi = {10.1016/j.jcp.2016.02.011},
	urldate = {2024-11-20},
	journal = {Journal of Computational Physics},
	author = {Walker, Shawn W.},
	month = apr,
	year = {2016},
	pages = {275--298}
}

@article{ohara_family_1992,
	title = {Family of energy functionals of knots},
	volume = {48},
	issn = {0166-8641},
	url = {https://www.sciencedirect.com/science/article/pii/016686419290023S},
	doi = {10.1016/0166-8641(92)90023-S},
	number = {2},
	urldate = {2024-11-20},
	journal = {Topology and its Applications},
	author = {O'Hara, Jun},
	month = dec,
	year = {1992},
	pages = {147--161}
}

@article{gonzalez_global_1999,
	title = {Global curvature, thickness, and the ideal shapes of knots},
	volume = {96},
	url = {https://www.pnas.org/doi/abs/10.1073/pnas.96.9.4769},
	doi = {10.1073/pnas.96.9.4769},
	number = {9},
	urldate = {2024-11-20},
	journal = {Proceedings of the National Academy of Sciences},
	author = {Gonzalez, Oscar and Maddocks, John H.},
	month = apr,
	year = {1999},
	note = {Publisher: Proceedings of the National Academy of Sciences},
	pages = {4769--4773}
}

@misc{bartels_stability_2018,
	title = {Stability of a simple scheme for the approximation of elastic knots and self-avoiding inextensible curves},
	url = {http://arxiv.org/abs/1804.02206},
	language = {en},
	urldate = {2024-11-19},
	publisher = {arXiv},
	author = {Bartels, Sören and Reiter, Philipp},
	month = apr,
	year = {2018},
	note = {arXiv:1804.02206 [math]}
}

@article{bartels_simple_2018,
	title = {A simple scheme for the approximation of self-avoiding inextensible curves},
	volume = {38},
	issn = {0272-4979},
	url = {https://doi.org/10.1093/imanum/drx021},
	doi = {10.1093/imanum/drx021},
	number = {2},
	urldate = {2024-11-20},
	journal = {IMA Journal of Numerical Analysis},
	author = {Bartels, Sören and Reiter, Philipp and Riege, Johannes},
	month = apr,
	year = {2018},
	pages = {543--565}
}

@article{fu_high_2023,
	title = {High order computation of optimal transport, mean field planning, and mean field games},
	volume = {491},
	issn = {00219991},
	url = {http://arxiv.org/abs/2302.02308},
	doi = {10.1016/j.jcp.2023.112346},
	language = {en},
	urldate = {2024-11-21},
	journal = {Journal of Computational Physics},
	author = {Fu, Guosheng and Liu, Siting and Osher, Stanley and Li, Wuchen},
	month = oct,
	year = {2023},
	note = {arXiv:2302.02308 [math]},
	pages = {112346}
}

@misc{blatt_regularity_2012,
  title={Regularity theory for tangent-point energies: The non-degenerate sub-critical case}, 
      author={Simon Blatt and Philipp Reiter},
      year={2012},
note = {arXiv:1208.3605 [math]}
}

@article{yu_repulsive_2021,
	title = {Repulsive {Curves}},
	volume = {40},
	issn = {0730-0301, 1557-7368},
	url = {https://dl.acm.org/doi/10.1145/3439429},
	doi = {10.1145/3439429},
    language = {en},
	number = {2},
	urldate = {2024-11-20},
	journal = {ACM Transactions on Graphics},
	author = {Yu, Chris and Schumacher, Henrik and Crane, Keenan},
	month = apr,
	year = {2021},
	pages = {1--21}
}

@book{augmented_1983,
    author = {Fortin, Michel and Glowinski, Roland},
    title = {Augmented Lagrangian methods: {Applications} to the {Numerical} {Solution} of {Boundary}-{Value} {Problems}},
    year = {1983}, 
    series = {Studies in mathematics and its Applications},
    volume = {15},
    publisher = {North-Holland Publishing Co., Amsterdam-New York},
}

@book{demengel_functional_2012,
	address = {London},
	series = {Universitext},
	title = {Functional {Spaces} for the {Theory} of {Elliptic} {Partial} {Differential} {Equations}},
	copyright = {https://www.springernature.com/gp/researchers/text-and-data-mining},
	isbn = {978-1-4471-2806-9 978-1-4471-2807-6},
	url = {https://link.springer.com/10.1007/978-1-4471-2807-6},
	language = {en},
	urldate = {2025-08-19},
	publisher = {Springer London},
	author = {Demengel, Françoise and Demengel, Gilbert},
	year = {2012},
	doi = {10.1007/978-1-4471-2807-6}
}

\appendix

\section{Derivation of the numerical method}\label{AppNumerics}

\subsection{Augmented Lagrangian formulation}\label{derivAugmLagr}
Following \cite{fu_high_2023}, we can rewrite the minimization problem \eqref{eq:BB2} as an unconstrained one by employing a suitable Lagrangian.
Let 
\begin{align*}
    \calG(\phi_t, \phi_t^{1d}) &:= -\int_\Omega \left[ \phi_1\, d\mu_1 - \phi_0\, d\mu_0 \right] -\int_\Gamma \left[ \phi^{1d}_1\, d\rho_1 - \phi^{1d}_0\, d\rho_0 \right], \\
    \skp{\phi_t, \rho_t}_\Omega &:=\int_{0}^1 \int_\Omega \phi_t \, d\rho_t \, dt \quad \text{and} \quad 
    \skp{\phi^{1d}_t, \mu_t}_\Gamma := \int_{0}^1 \int_\Gamma \phi^{1d}_t\, d\mu_t \, dt
\end{align*}
where $\phi_t : [0,1]\times \Omega \rightarrow \R$ and $\phi^{1d}_t : [0,1]\times\Gamma \rightarrow \R$ are smooth functions. 
Defining $\phi_t$ and $\phi_t^{1d}$ as test functions in Definition \ref{def:conteq} allows to rewrite \eqref{eq:BB2} as the saddle-point problem
\begin{align*}
    \inf\limits_{(\rho_t, J_t, \mu_t, \mathcal{V}_t, f_t) \in \mathcal{D}_{\rm adm}(\Gamma)} \sup\limits_{(\phi_t, \phi^{1d}_t)} &\int_0^1 \calA_\Gamma^{\alpha_1, \alpha_2}(\rho_t, J_t, \mu_t, \mathcal{V}_t, f_t)  \, dt - \calG(\phi_t, \phi^{1d}_t) \\
    &- \skp{(\partial_t\phi_t,\nabla\phi_t),(\rho_t, J_t)}_\Omega - \skp{(\partial_t\phi^{1d}_t, \nabla\phi^{1d}_t, \phi^{1d}_t-\phi_t\circ\gamma),(\mu_t, {\mathcal{V}_t}, f_t)}_\Gamma .
\end{align*}
To shorten notation, let
\begin{align*}
    F^{\alpha_1, \alpha_2}(u_\Omega, v_\Omega, u_\Gamma,v_\Gamma, w_\Gamma) := \Psi(u_\Omega, v_\Omega) + \alpha_1 \Psi(u_\Gamma,v_\Gamma) + \alpha_2 \Psi(u_\Gamma,w_\Gamma)
\end{align*}
for 
\begin{align*}
    \Psi(u,v) &:= \begin{cases}
        \frac{|v|^2}{2u} &\text{ if }  u>0\\
        0 &\text{ if } u=0 \text{ and } v=0\\
        +\infty &\text{ else}
    \end{cases}
\end{align*}
which corresponds to the integrand of \eqref{eq:actionformula} for linear mobilities and $\alpha_3=0$. 
By convexity and lower semicontinuity of $F^{\alpha_1, \alpha_2}$, the identity 
\begin{align*}
    F^{\alpha_1, \alpha_2}(u_\Omega, v_\Omega, u_\Gamma,v_\Gamma, w_\Gamma) & = \sup\limits_{\genfrac{}{}{0pt}{}{u_\Omega, v_\Omega}{u_\Gamma,v_\Gamma, w_\Gamma}} \skp{(u^*_\Omega, v^*_\Omega), (u_\Omega, v_\Omega)} + \skp{(u^*_\Gamma, v^*_\Gamma, w^*_\Gamma), (u_\Gamma, v_\Gamma, w_\Gamma)} \\
    & \qquad \qquad - \left(F^{\alpha_1, \alpha_2}\right)^*(u_\Omega, v_\Omega, u_\Gamma,v_\Gamma, w_\Gamma)
\end{align*}
holds, where $ \left(F^{\alpha_1, \alpha_2}\right)^*$ is the convex conjugate of $F^{\alpha_1, \alpha_2}$. Substituting this expression into our objective functional leads to the formulation
\begin{align*}
    \sup\limits_{(\rho_t, J_t, \mu_t, \mathcal{V}_t, f_t)} \inf\limits_{\genfrac{}{}{0pt}{}{(\phi_t, \phi_t^{1d}),}{(\rho_t^*, J_t^*, \mu_t^*, \mathcal{V}_t^*, f_t^*)}} &\int_0^1 \int \left(F^{\alpha_1, \alpha_2}\right)^*(\rho_t^*, J_t^*, \mu_t^*, \mathcal{V}_t^*, f_t^*)  \, dx dt + \calG(\phi_t, \phi_t^{1d}) \\
    &+ \skp{(\partial_t\phi_t - \rho_t^*,\nabla\phi_t - J_t^*),(\rho_t, J_t)}_\Omega \\
    &+ \skp{(\partial_t\phi^{1d}_t - \mu_t^*, \nabla\phi_t^{1d} - \mathcal{V}_t^*, \phi_t^{1d}-\phi_t\circ\gamma - f_t^*),(\mu_t, \mathcal{V}_t, f_t)}_\Gamma 
\end{align*}
fitting the necessary structure for applying the classical method ALG2 from \cite{augmented_1983}. Moreover, calculating the convex conjugate of $F^{\alpha_1, \alpha_2}$ gives
\begin{align*}
    &\left(F^{\alpha_1, \alpha_2}\right)^*(u^*_\Omega, v^*_\Omega, u^*_\Gamma,v^*_\Gamma, w^*_\Gamma)\\
    =& \sup\limits_{u_\Omega, v_\Omega, u_\Gamma,v_\Gamma, w_\Gamma} \skp{u^*_\Omega, u_\Omega} + \skp{v^*_\Omega, v_\Omega} + \skp{u^*_\Gamma, u_\Gamma} + \skp{v^*_\Gamma, v_\Gamma} + \skp{w^*_\Gamma, w_\Gamma} - F^{\alpha_1, \alpha_2}(u_\Omega, v_\Omega, u_\Gamma,v_\Gamma, w_\Gamma)\\
    =& \sup\limits_{u_\Omega, v_\Omega} \skp{u^*_\Omega, u_\Omega} + \skp{v^*_\Omega, v_\Omega} - \Psi(u_\Omega, v_\Omega) + \sup\limits_{u_\Gamma,v_\Gamma, w_\Gamma} \skp{u^*_\Gamma, u_\Gamma} + \skp{v^*_\Gamma, v_\Gamma} + \skp{w^*_\Gamma, w_\Gamma} - \alpha_1 \Psi(u_\Gamma, v_\Gamma) - \alpha_2 \Psi(u_\Gamma, w_\Gamma)\\
    =& \, \iota_{\{u_\Omega^*+\frac{|v_\Omega^*|^2}{2}\leq 0\}} + \iota_{\{u_\Gamma^* + \frac{1}{2}\left( \frac{|v_\Gamma^*|^2}{\alpha_1} + \frac{|w_\Gamma^*|^2}{\alpha_2} \right)\leq 0\}}
\end{align*}
for the convex indicator function 
\begin{align*}
    \iota_C (x) &:= \begin{cases}
        0 & \text { if } x\in C\\ +\infty & \text{ else}
    \end{cases}
\end{align*}
implying
\begin{align*}
    \left(\calA^{\alpha_1, \alpha_2}\right)^*(\rho_t^*, J_t^*, \mu_t^*, \mathcal{V}_t^*, f_t^*) &:= \int_\Omega \iota_{\{\rho_t^*+\frac{|J_t^*|^2}{2}\leq 0\}} \, dx + \int_\Gamma \iota_{\{\mu_t^* + \frac{1}{2}\left( \frac{|\mathcal{V}_t^*|^2}{\alpha_1} + \frac{|f_t^*|^2}{\alpha_2}\right)\leq 0\}} \, dx .
\end{align*}

In this formulation, the variables $\rho_t, J_t, \mu_t, \mathcal{V}_t$ and $f_t$ can be understood as Lagrange multipliers of the constraints $\partial_t\phi_t - \rho_t^* = 0$, $\nabla\phi_t - J_t^* = 0$, $\partial_t\phi^{1d}_t - \mu_t^*=0$, $\nabla\phi_t^{1d} - \mathcal{V}_t^* = 0$ and $\phi_t^{1d} - \phi_t - f_t^* = 0$.

\subsection{Algorithm}\label{DetailsAlg}

In the classical ALG2 method from \cite{augmented_1983}, the aim is to solve a saddle-point problem. Here, the optimization is decoupled and in each step the variables are considered separately. As introduced in Section \ref{Numerics}, there are three main steps to this algorithm. In what follows, we give a detailed explanation of each of these steps that result in Algorithm \ref{alg:AugmLagr}. Throughout, $(\phi_m, \phi^{1d}_m), (\rho^*_m, J^*_m, \mu^*_m, \mathcal{V}^*_m, f^*_m)$ and $(\rho_m, J_m, \mu_m, \mathcal{V}_m, f_m)$ denote the approximate solutions after $m$ iterations and, to improve readability, we omit the explicit time-dependence of the iterates.

\medskip 
\noindent \textbf{Step 1:} \underline{Update $(\phi_t, \phi_t^{1d})$}.

\noindent For this update, we aim at solving
\begin{align*}
    \inf\limits_{(\phi_{m+1}, \phi^{1d}_{m+1})} L^{\alpha_1, \alpha_2}(\rho_m, J_m, \mu_m, \mathcal{V}_m, f_m, \phi_{m+1}, \phi^{1d}_{m+1}, \rho^*_m, J^*_m, \mu^*_m, \mathcal{V}^*_m, f^*_m).
\end{align*}
In particular, the first variation with respect to $\phi_{m+1}$ and $\phi^{1d}_{m+1}$ needs to be zero as necessary condition, leading to the following system of equations
\begin{align*}
    &\frac{r_1}{r_2}\skp{\partial_t\phi_{m+1}, \partial_t\psi_t}_\Omega + \frac{r_1}{r_2}\skp{\nabla\phi_{m+1}, \nabla\psi_t}_\Omega + \skp{\phi_{m+1}\circ\gamma, \psi_{t}\circ\gamma}_\Gamma - \skp{\phi^{1d}_{m+1}, \psi_{t}\circ\gamma}_\Gamma\\
    =& \frac{1}{r_2}\skp{\psi_1,\rho_1} - \frac{1}{r_2}\skp{\psi_0, \rho_0} + \skp{\psi_t\circ\gamma, \frac{1}{r_2}f_m -f^*_m}_\Gamma + \skp{\partial_t\psi_t, \frac{r_1}{r_2}\rho^*_m - \frac{1}{r_2}\rho_m}_\Omega + \skp{\nabla\psi_t, \frac{r_1}{r_2} J^*_m-\frac{1}{r_2}J_m}_\Omega \\
    &\skp{\phi^{1d}_{m+1}, \psi_t^{1d}}_\Gamma + \skp{\partial_t\phi^{1d}_{m+1}, \partial_t\psi_t^{1d}}_\Gamma + \skp{\nabla\phi^{1d}_{m+1}, \nabla\psi_t^{1d}}_\Gamma - \skp{\phi_{m+1}\circ\gamma, \psi_t^{1d}}_\Gamma \\
    +& \frac{1}{r_2}\skp{\psi^{1d}_1, \mu_1}_\Gamma - \frac{1}{r_2}\skp{\psi^{1d}_0, \mu_0}_\Gamma + \skp{\psi^{1d}_t, f^*_m - \frac{1}{r_2}f_m}_\Gamma + \skp{\partial_t\psi^{1d}_t, \mu^*_m - \frac{1}{r_2}\mu_m}_\Gamma + \skp{\nabla\psi_t^{1d}, \mathcal{V}^*_m - \frac{1}{r_2}{\mathcal{V}_m}}_\Gamma
\end{align*}
for test functions $\psi_t$ and $\psi_t^{1d}$.
Introducing the bilinear-form 
\begin{align*}
    &A(\phi_{m+1}, \psi_t, \phi^{1d}_{m+1}, \psi^{1d}_t) \\
    :=& \frac{r_1}{r_2}\skp{\partial_t\phi_{m+1}, \partial_t\psi_t}_\Omega + \frac{r_1}{r_2}\skp{\nabla\phi_{m+1}, \nabla\psi_t}_\Omega + \skp{\phi_{m+1}\circ\gamma, \psi_t\circ\gamma}_\Gamma - \skp{\phi^{1d}_{m+1}, \psi_t\circ\gamma}_\Gamma\\
    +& \skp{\phi^{1d}_{m+1}, \psi_t^{1d}}_\Gamma + \skp{\partial_t\phi^{1d}_{m+1}, \partial_t\psi_t^{1d}}_\Gamma + \skp{\nabla\phi^{1d}_{m+1}, \nabla\psi_t^{1d}}_\Gamma - \skp{\phi_{m+1}\circ\gamma, \psi_t^{1d}}_\Gamma 
\end{align*}
and the linear form
\begin{align*}
    &F^m(\psi_t, \psi_t^{1d})\\
    :=& \frac{1}{r_2}\skp{\psi_1,\rho_1} - \frac{1}{r_2}\skp{\psi_0, \rho_0} + \skp{\psi_t\circ\gamma, \frac{1}{r_2}f_m -f^*_m}_\Gamma + \skp{\partial_t\psi_t, \frac{r_1}{r_2}\rho^*_m - \frac{1}{r_2}\rho_m} + \skp{\nabla\psi_t, \frac{r_1}{r_2} J^*_m-\frac{1}{r_2}J_m}\\
    +& \frac{1}{r_2}\skp{\psi^{1d}_1, \mu_1}_\Gamma - \frac{1}{r_2}\skp{\psi^{1d}_0, \mu_0}_\Gamma + \skp{\psi_t^{1d}, f^*_m - \frac{1}{r_2}f_m}_\Gamma + \skp{\partial_t\psi_t^{1d}, \mu^*_m - \frac{1}{r_2}\mu}_\Gamma + \skp{\nabla\psi_t^{1d}, \mathcal{V}^*_m - \frac{1}{r_2}{\mathcal{V}_m}}_\Gamma
\end{align*}
the optimality system reads
\begin{align*}
    A(\phi_{m+1}, \psi_t, \phi^{1d}_{m+1}, \psi^{1d}_t) &= F^m(\psi_t, \psi^{1d}_t).
\end{align*}
\textbf{Step 2:} \underline{Update $(\rho_t^*, J_t^*, \mu_t^*, \mathcal{V}_t^*, f_t^*)$}.

\noindent The new iterates $(\rho^*_{m+1}, J^*_{m+1}, \mu^*_{m+1}, \mathcal{V}^*_{m+1}, f^*_{m+1})$ are defined as the solution of
\begin{align*}
    \inf\limits_{(\rho^*_{m+1}, J^*_{m+1}, \mu^*_{m+1}, \mathcal{V}^*_{m+1}, f^*_{m+1})} &\int_0^1 \left(\calA^{\alpha_1, \alpha_2}\right)^*(\rho^*_{m+1}, J^*_{m+1}, \mu^*_{m+1}, \mathcal{V}^*_{m+1}, f^*_{m+1}) \\
    &- \skp{\rho^*_{m+1}, \rho_m + r_1\partial_t\phi_{m+1}}_\Omega + \frac{r_1}{2}|\rho^*_{m+1}|_\Omega^2\\
    &- \skp{J^*_{m+1}, J_m + r_1\nabla\phi_{m+1}}_\Omega + \frac{r_1}{2}|J^*_{m+1}|_\Omega^2\\
    &- \skp{\mu^*_{m+1}, \mu_{m} + r_2\partial_t\phi^{1d}_m}_\Gamma + \frac{r_2}{2} |\mu^*_{m+1}|^2_\Gamma\\
    &- \skp{\mathcal{V}^*_{m+1}, \mathcal{V}_{m} + r_2\nabla\phi^{1d}_m}_\Gamma + \frac{r_2}{2} |\mathcal{V}^*_{m+1}|^2_\Gamma\\
    &- \skp{f^*_{m+1}, f_{m} + r_2\left( \phi^{1d}_{m+1} - \phi_{m+1}\circ\gamma \right)}_\Gamma + \frac{r_2}{2} |f^*_{m+1}|^2_\Gamma
\end{align*}
composed of two independent problems, one in $\Omega$ behaving as in \cite{fu_high_2023} and the other one on $\Gamma$.
To simplify notation, we introduce the following variables containing the information from the previous iteration
\begin{align*}
    \eta_{\rho^*} &:= \partial_t\phi_{m+1} + \frac{1}{r_1}\rho_m,\\
    \eta_{J^*} &:= \nabla\phi_{m+1} + \frac{1}{r_1}J_m,\\
    \eta_{\mu^*} &:= \partial_t\phi^{1d}_{m+1} + \frac{1}{r_2}\mu_m,\\
    \eta_{\mathcal{V}^*} &:= \nabla\phi^{1d}_{m+1} + \frac{1}{r_2}\mathcal{V}_m,\\
    \eta_{f^*} &:= \phi^{1d}_{m+1} - \phi_{m+1}\circ\gamma + \frac{1}{r_2}f_m
\end{align*}
and $\eta_\Omega := (\eta_{\rho^*}, \eta_{J^*})$, $\eta_\Gamma := (\eta_{\mu^*}, \eta_{\mathcal{V}^*}, \eta_{f^*})$.
The minimization problems read
\begin{align*}
    \inf\limits_{(\rho^*_{m+1}, J^*_{m+1})}& \iota_{\{\rho^*_{m+1} + \frac{1}{2}|J^*_{m+1}|\leq 0 \}} - r_1\skp{\eta_\Omega, (\rho^*_{m+1}, J^*_{m+1})} + \frac{r_1}{2}\left( |\rho^*_{m+1}|^2 + |J^*_{m+1}|^2\right)\\
    \inf\limits_{(\mu^*_{m+1}, \mathcal{V}^*_{m+1}, f^*_{m+1})}& \iota_{\{ \mu^*_{m+1} + \frac{1}{2}\left( \frac{|\mathcal{V}^*_{m+1}|^2}{\alpha_1} + \frac{|f^*_{m+1}|^2}{\alpha_2}\right)\leq 0\}} - r_2\skp{\eta_\Gamma, (\mu^*_{m+1}, V^*_{m+1}, f^*_{m+1})} \\
    &+ \frac{r_2}{2}\left( |\mu^*_{m+1}|^2 + |\mathcal{V}^*_{m+1}|^2 +|f^*_{m+1}|^2\right).
\end{align*}
If the pair $(\eta_\Omega, \eta_\Gamma)$ is admissible, then $(\rho^*_{m+1}, J^*_{m+1}, \mu^*_{m+1}, \mathcal{V}^*_{m+1}, f^*_{m+1}) = (\eta_\Omega, \eta_\Gamma)$ is optimal. Otherwise let
\begin{align*}
    x_\Omega &:= |J^*_{m+1}|\\
    y_\Omega &:= J^*_{m+1,1}\\
    x_\Gamma &:= |\mathcal{V}^*_{m+1}|\\
    y_\Gamma &:= f^*_{m+1}
\end{align*}
where $J^*_{m+1,1}$ is the first component of $J^*_{m+1}$. Substituting these into the minimization and formulating the KKT system leads to 
\begin{align*}
    0 &= x_\Omega^3 + 2 (1+\eta_{\rho^*}) x_\Omega - 2 | \eta_{J^*}|\\
    0 &= x_\Gamma^3 + \frac{\alpha_1}{\alpha_2} x_\Gamma y_\Gamma^2 + 2 \left( \alpha_1^2 + \alpha_1\eta_{\mu^*}\right) x_\Gamma - 2 \alpha_1^2 |\eta_{\mathcal{V}^*}|\\
    0 &= y_\Gamma^3 + \frac{\alpha_2}{\alpha_1} x_\Gamma^2 y_\Gamma + 2 \left( \alpha_2^2 + \alpha_2\eta_{\mu^*}\right) y_\Gamma - 2 \alpha_2^2 \eta_{f^*}
\end{align*}
instead. This system can be solved using Newton's method. 

\medskip

\noindent \textbf{Step 3:} \underline{Update $(\rho_t, J_t, \mu_t, \mathcal{V}_t, f_t)$}.

\noindent The new iterate will be defined as the solution of
\begin{align*}
    &\sup\limits_{(\rho_{m+1}, J_{m+1}, \mu_{m+1}, V_{m+1}, f_{m+1})} \skp{(\partial_t\phi_{m+1} - \rho^*_{m+1},\nabla\phi_{m+1} - J^*_{m+1}),(\rho, J)}_\Omega\\
    &+\skp{(\partial_t\phi^{1d}_{m+1} - \mu^*_{m+1}, \nabla\psi_{m+1} - \mathcal{V}^*_{m+1}, \phi^{1d}_{m+1}-\phi_{m+1}\circ\gamma - f^*_{m+1}),(\mu_{m+1}, \mathcal{V}_{m+1}, f_{m+1})}_\Gamma.
\end{align*}
and approximated using one gradient descent step given as
\begin{align*}
    \rho_{m+1} &:= \rho_{m} + r_1 \left( \partial_t\phi_{m+1} - \rho^*_{m+1} \right),\\
    J_{m+1} &:= J_{m} + r_1 \left( \nabla\phi_{m+1} - J^*_{m+1} \right),\\
    \mu_{m+1} &:= \mu_{m} + r_2 \left( \partial_t\phi^{1d}_{m+1} - \mu^*_{m+1} \right),\\
    \mathcal{V}_{m+1} &:= \mathcal{V}_m + r_2 \left( \nabla\phi^{1d}_{m+1} - \mathcal{V}^*_{m+1} \right),\\
    f_{m+1} &:= f_{m} + r_2 \left( \phi^{1d}_{m+1} - \phi_{m+1} - f^*_{m+1} \right).
\end{align*}

\end{document}